\documentclass[a4paper, 12pt, twoside]{article}

\usepackage{enumerate}
\usepackage{enumitem}

\usepackage[section]{placeins}	
\usepackage{epigraph}
\usepackage[font={small}]{caption}
\usepackage{appendix}

\usepackage{amsmath,amssymb,mathrsfs}
\usepackage{amsfonts}			
\usepackage{nicefrac}			

\usepackage{bbm} 

\usepackage{stmaryrd}
\usepackage{calc,ifthen,xspace}

\usepackage{graphicx}			

\usepackage{tikz}
\usepackage{tikz-cd}

\usepackage{xspace}
\usepackage{textcomp}
\usepackage{array}
\usepackage{hyphenat}

\usepackage{subfigure}
\usepackage[subfigure]{tocloft}

\usepackage{etoolbox} 
\patchcmd{\thebibliography}
  {\settowidth}
  {\setlength{\itemsep}{-2pt plus 0.1pt}\settowidth}
  {}{}
\apptocmd{\thebibliography}
  {\small}
  {}{}

\usepackage[pdftex,pdfborder={0 0 0},
			colorlinks=true,
			linkcolor=blue,
			citecolor=red,
			pagebackref=true,
			]{hyperref}	

\usepackage[top=2.1cm, bottom=2cm, left=2.5cm, right=2cm,
			headheight=15pt]{geometry}
\usepackage{cleveref}
\crefformat{section}{\S#2#1#3} 
\crefformat{subsection}{\S#2#1#3}
\crefformat{subsubsection}{\S#2#1#3}

\AtBeginDocument{}

\renewcommand{\leq}{\ensuremath{\leqslant}} 
\renewcommand{\geq}{\ensuremath{\geqslant}}

\newcommand{\Z}{\mathbb Z}

\DeclareMathOperator{\ima}{Im}

\DeclareMathOperator{\Hom}{Hom}
\DeclareMathOperator{\Homeo}{Homeo}
\DeclareMathOperator{\Aut}{Aut}

\DeclareMathOperator{\Inn}{Inn}

\DeclareMathOperator{\End}{End}

\DeclareMathOperator{\qd}{\triangleright}

\DeclareMathOperator{\Stab}{Stab}

\newcommand*{\qdn}[1]{\overset{\scalebox{0.6}{$#1$}}{\triangleright}}

\usepackage{amsthm}

\makeatletter
\newtheorem*{rep@theorem}{\rep@title}
\newcommand{\newreptheorem}[2]{%
\newenvironment{rep#1}[1]{%
 \def\rep@title{#2 \ref{##1}}%
 \begin{rep@theorem}}%
 {\end{rep@theorem}}}
\makeatother

\theoremstyle{plain}
\newtheorem{theo}{Theorem}
\newreptheorem{theo}{Theorem}
\newtheorem{introtheo}{Theorem}

\newtheorem{lem}[theo]{Lemma}
\newtheorem{prop}[theo]{Proposition}
\newreptheorem{prop}{Proposition}

\newtheorem{cor}[theo]{Corollary}
\newreptheorem{cor}{Corollary}
\newtheorem{fact}[theo]{Fact}
\newcounter{CJ}
\newtheorem{conj}[CJ]{Conjecture}

\newreptheorem{pb}{Problem}

\numberwithin{equation}{theo}

\theoremstyle{definition}
\newtheorem{defi}[theo]{Definition}
\newtheorem{nota}[theo]{Notation}
\newtheorem{conv}[theo]{Convention}
\newtheorem{ex}[theo]{Example}
\newtheorem{rmq}[theo]{Remark}

\newcommand{\ev}{\mathrm{ev}}

\let\oldpagenumbering\pagenumbering
\renewcommand{\pagenumbering}[1]{%
	\cleardoublepage
	\oldpagenumbering{#1}
}

\author{Jacques \scshape{Darn\'e}}
\title{Profinite rigidity and quandle colourings}
\date{October 1st, 2026}

\hypersetup{
    pdfauthor={Jacques Darn\'e},
    pdfsubject={Preprint},
    pdftitle={Profinite rigidity and quandle colourings},
    pdfkeywords={algebraic topology, quandles}
}

\begin{document}

\maketitle

\begin{abstract}
We show that the numbers of colourings of a link by all finite quandles is equivalent, as an invariant of links, to the profinite completion of the peripheral system. This allows us to reformulate the classical question of the completeness of the former invariant in terms of profinite rigidity. We also develop some profinite quandle theory, and we show that a topological quandle is profinite if and only if it is compact and totally discontinuous and its group of inner automorphism is relatively compact in its homeomorphism group.
\end{abstract}

\numberwithin{theo}{section}

\section*{Introduction}

The present paper aims at building a bridge between the theory of colourings by finite quandles and the world of profinite rigidity. The possibility of such a construction is not surprising, and it was already known to experts (see for instance~\cite{Agol}). After all, the fundamental quandle $Q_L$ of a link $L$ was introduced to encode nicely the fundamental group $G_L$ of the complement of $L$ in $\mathbb S^3$, together with the so-called \emph{peripheral structure} \cite{Joyce, Matveev}. Thus, the fact that colourings by finite quandles, which are morphisms from $Q_L$ to finite quandles, are linked to finite quotients of $G_L$, encoded in the profinite completion $\widehat G_L$, is to be expected. However, the precise bridge was yet to be built, and its construction turned out to require some non-trivial work, mainly algebraic in nature.

\paragraph{Main results}
The main goal of this paper is to reformulate the following conjecture, stated for knots in~\cite{CESY} (see also~\cite{CSV}) in the language of profinite rigidity:
\begin{conj}
If two links cannot be obtained from one another by taking reversed mirror images of split components, then there is a finite quandle $Q$ such that the number of colourings by elements of $Q$ is different for the two links.
\end{conj}
We show that it is equivalent to:
\begin{conj}
If two links cannot be obtained from one another by taking reversed mirror images of split components, then the profinite completions of their peripheral systems are not equivalent.
\end{conj}
Recall that the \emph{peripheral system} $P_L = (G_L, (P_i, m_i))$ of a link $L$ is the data of the fundamental group $G_L$ of the complement $\mathbb S^3 - L$, together with a subgroup $P_i$ and an element $m_i$ for each component of the link. Precisely, if $N_i$ is a tubular neighbourhood of the $i$-th component $L_i$ of $L$, then $\partial N_i$ is a torus, and $P_i$ is the image of $\pi_1(\partial N_i) \cong \Z^2$ in $G_L$. Then $m_i \in P_i$ is a \emph{meridian}, which is a loop turning once around $L_i$. Up to equivalence (whose definition will be recalled later on), the peripheral system is an invariant of links, known to detect them up to taking reversed mirror images of split components~\cite[Cor.~5.3]{FR_racks_and_links}.

The \emph{profinite completion} $\widehat P_L$ of the peripheral system $P_L$ is defined in the obvious way: replace $G_L$ by its profinite completion $\widehat G_L$, each $m_i$ by its image therein, and each $P_i$ by the closure of its image in $\widehat G_L$. This completion process may result in a weaker invariant: information may have been lost along the way. The above conjecture says that it should not be the case.

The main theorem of the paper is:
\begin{introtheo}\label{main_alg_thm}
For a finitely generated quandles $Q$ and $R$, the following are equivalent:
\begin{enumerate}[itemsep=-0.3em, topsep = 4pt]
    \item For every finite quandle $X$, $\# \Hom(Q, X) = \# \Hom(R, X)$.
    \item The profinite completions of $Q$ and $R$ are (continuously) isomorphic.
    \item The profinite completions of the peripheral systems associated to $Q$ and $R$ are equivalent.
\end{enumerate}
\end{introtheo}
The equivalence of the above two conjectures is a direct application of this theorem to fundamental quandles of links. Indeed, if $Q_L$ is the fundamental quandle of a link $L$, then the number of colourings of $L$ by a finite quandle $X$ is $\# \Hom(Q_L, X)$, whereas the peripheral system associated to $Q_L$ is the peripheral system of $L$. 

\paragraph{Profinite rigidity} 
The profinite completion $\widehat P_L$ of the peripheral system is of course a stronger invariant than the the profinite completion $\widehat G_L$ of the fundamental group of the complement, together with the number of components of the link. So if we know that $\widehat G_L$ distinguishes a knot $K$ from other knots, then $\widehat P_L$ distinguishes it from other links.

Such profinite rigidity has been conjectured for some time for prime knots \cite[Q.~1.3]{Boileau-Friedl}, and in particular for hyperbolic knots. It is known for the figure-eight knot and torus knots \cite[Cor.~1.4 and Thm.~1.5]{Boileau-Friedl} (see also \cite[Th.~A]{Bridson-Reid} for a stronger rigidity statement, of the complement of the figure-eight knot amongst $3$-manifolds). Recently, it has been shown for $2$-bridge knots and $3$-tangle Montesinos links~\cite{Xu3}. Let us also mention the striking result of~\cite[Th.~D]{Xu1} that says every link $L$ is determined by $\widehat G_L$ up to a finite number of possibilities. 

\begin{rmq}
From a quandle-theoretic perspective, the irrelevance of the peripheral structure on $\widehat G_L$ for \emph{prime} links would not be surprising, since a link $L$ is prime exactly when its fundamental quandle $Q_L$ (an invariant equivalent to $P_L$) injects into the fundamental group  $G_L$~\cite{Ryder}.
\end{rmq}

The whole of $\widehat P_L$ (which arises naturally from quandle theory) has appeared in the literature on profinite rigidity of non-necessarily prime links, in particular in Wilkes' PhD thesis (see for instance \cite[Th.~2.7]{Wilkes2}), and in more recent work by Xu. In particular, building on Wilkes' proof of the rigidity (with respect to $\widehat G_L$) of the complement of graph knots (whose complements are graph manifolds)~\cite[Th.~4.1]{Wilkes3}, Xu shows that graph knots are rigid with respect to $\widehat P_L$~\cite[Lem.~10.5]{Xu2}. Profinite peripheral data is also an important character in~\cite[Th.~A]{Xu1}, but also in~\cite{Xu3}: although the latter shows a rigidity result in the classical sense (with respect to $\widehat G_L$), the peripheral data plays a prominent role in this work.

\paragraph{Profinite quandles} Profinite quandles are at the core of the bridge that we are building between the two conjectures. These have appear before, but the literature on this topic is scarce. A closely related topic is the study of \emph{residually finite} quandles, which were studied in~\cite{BSS_knots, BSS_links}. Their main result is the fact that \emph{link quandles are always residually finite}, an important theorem deduced from previous results in the theory of $3$-manifold groups. As for profinite quandles, they were studied notably in relation to Galois groups of number fields~\cite{Takahashi, Szymik2}. They were also studied for themselves in~\cite{Singh, Yetter_et_al}. We recover and improve most of the results of these last two papers, discussing their contributions to the theory, but also their shortcomings; see Remark~\ref{rk_Singh_finite_nb_of_cofinite}, Remark~\ref{rk_Yetter_Pph} and the two paragraphs at the end of~\cref{sec_inner_aut}.

\paragraph{Outline of the paper} 
As the statement of our theorem~\ref{main_alg_thm} suggests, the profinite completion $\widehat Q$ of a quandle $Q$ (defined as the projective limit of its finite quotients) will serve as an intermediate object between morphisms to finite quandles and profinite peripheral system. Precisely, the proof is divided into two distinct parts. 

The first part (\cref{sec_algebras} and \cref{sec_prof_alg}) consists in proving that the function which associates the number $\# \Hom(Q, F)$ to a finite quandle $F$ determines $\widehat Q$ up to isomorphism (the fact that it is determined by it follows directly from the definitions). We do this in the larger context of (topological) algebras over a Lawvere theory $\mathbb T$, by generalizing the main theorem of~\cite{DFPR}, which was the case of groups. We introduce this context in \cref{sec_algebras}, where we prove an essential ingredient in the study of profinite completions: a (topologically) finitely generated algebra has a finite number of (open) congruences of index $n$, for every integer $n \geq 1$ (see~\cref{subsec_finite_nb_of_cong}). 
Profinite completions are then introduced in \cref{subsec_prof_comp_fg}. There, we show that for a topologically finitely generated $\mathbb T$-algebra $X$, the numbers $\# \Hom(X, F)$ (for $F$ ranging over finite discrete $\mathbb T$-algebras) determine the set $\Phi(X)$ of isomorphism classes of finite quotients of $X$ (\cref{par_nb_of_morphisms}). 
We then go on to show that $\Phi(X)$ determines $\widehat X$ up to continuous isomorphism, using the general reasoning from~\cite{DFPR}  (\cref{par_prof_comp_and_morphisms}). 

The second part (\cref{sec_quandles}, \cref{sec_ProQdl} and \cref{sec_ProPph}) consists in proving that the data of $\widehat Q$ up to isomorphism is equivalent to the data of the profinite completion of the peripheral system $\mathcal P(Q)$ associated to $Q$, up to equivalence. In order to do that, we first need to develop some topological quandle theory (\cref{sec_quandles}) and profinite quandle theory (\cref{sec_ProQdl}). In particular, we need to adapt the theory of the enveloping group $G(Q)$ of a quandle $Q$ to both these contexts. 
In \cref{sec_ProPph}, we use these tools to construct the \emph{profinite peripheral system} $\widehat{\mathcal P}(R)$ associated to a profinite quandle $R$, which is closely linked with the action of the profinite enveloping group of $R$ on $R$. This allows us to extend to the profinite world Joyce's construction of the correspondence between quandles are peripheral systems: we show that giving a profinite quandle $R$ (whose set of orbits is finite discrete -- a technical condition needed for the construction, but satisfied for instance by $\widehat Q$ if the (discrete) quandle $Q$ is finitely generated) is equivalent to giving the profinite peripheral system $\widehat{\mathcal P}(R)$. Then we show that for a quandle $Q$, the profinite peripheral system $\widehat{\mathcal P}(\widehat Q)$ is in fact the profinite completion of $\mathcal P(Q)$, which gives the desired conclusion.

In the course of the proof that we just sketched, we need to develop quite a bit of profinite quandle theory. We continue on this course in the third part of the paper (\cref{sec_inner_aut} and \cref{sec_criterion}). In~\cref{sec_inner_aut}, we develop the theory of inner automorphisms for topological and profinite quandles. There, we introduce and study of the closure $\overline\Inn(Q)$ of the usual inner automorphism group of a topological quandle $Q$ in $\Homeo(Q)$. In particular, we show that it is a profinite group if $Q$ is profinite. In fact, we can make this into a criterion for profiniteness of quandles, which is the topic of~\cref{sec_criterion}. There, we show that a topological quandle is profinite if and only if it is compact and totally discontinuous and $\Inn(Q)$ is relatively compact in $\Homeo(Q)$. Thus, contrary to the case of groups, profiniteness is not a topological property of topological quandles, but we still have a nice criterion for it.

\paragraph{Disclaimer} Although IA has been used punctually for side questions and bibliographical research, it has not been used for the main research done here, and not a single word of this paper has been written by it.

\clearpage

\setcounter{tocdepth}{2}
\tableofcontents
\clearpage

\section{Algebras and topological algebras}\label{sec_algebras}

Our first main goal is the equivalence, as invariants of links, between the numbers of colorings by finite quandles and the profinite completion of the fundamental quandle. This will be a consequence (Cor.~\ref{cor_col_and_prof}) of a general result about (topological) algebras over a (discrete) Lawvere theory. Namely, we show that the profinite completion of a topologically finitely generated algebra encodes exactly the numbers of morphisms to finite discrete algebras (Thm.~\ref{theo_col_and_prof}). This is a generalization of a group-theoretic result from~\cite{DFPR}.

We begin by a general introduction to topological algebras and their quotients, leading up to a crucial ingredient in the proof of Theorem~\ref{theo_col_and_prof}: for every integer $n \geq 1$, the number of open congruences of index $n$ on a topologically finitely generated algebra is finite (Prop.~\ref{congruences_of_fixed_index}). 

\subsection{Algebras}

Let $\mathbb T$ be a Lawvere theory. Recall that it consists, for every $n$, of a set $\mathbb T_n$ of $n$-ary operation, and of a set of equations between such operations. A \emph{$\mathbb T$-algebra} is a set $X$ with a map $\omega_X : X^n \rightarrow X$ for each $\omega \in \mathbb T_n$, such that the equations are satisfied. A \emph{morphism} $f: X \rightarrow Y$ between $\mathbb T$-algebras $X$ and $Y$ is a map from $X$ to $Y$ such that for every $\omega \in \mathbb T_n$, we have $f \circ \omega_X = \omega_Y \circ f^{\times n}$.

A theory can be described by a \emph{presentation by generators and relations}, by taking a \emph{generating set} of operations (such that every operation is a composite of such generating operations) and a \emph{generating set} of equations (such that every equation true in every algebra is a consequence of these generating equations). This point of view is useful to describe theories in a compact way, but it can also be useful to consider all the operations and all the true equations at once.

\begin{defi}
A theory is called \emph{finitely generated} if it admits a finite set of generating operations.
\end{defi}

\begin{ex}\label{sets_as_T-alg}
Sets are algebras over the trivial theory, described by the empty presentation.
\end{ex}

\begin{ex}\label{groups_as_T-alg}
Groups are algebras over the theory having three generating operations: a binary one (for the product), a unary one (for the inverse) and a constant (a $0$-ary operation, for the neutral element). The usual group axioms are generating equations for the theory of groups.
\end{ex}

\begin{ex}\label{qdls_as_T-alg}
In order to see quandles as algebras over a theory, we need to write the second axiom of Definition~\ref{def_qdl} as a set of equations. This is possible only if we include $\qdn{-1}$ as a generating operation. Then quandles are algebras over the theory having two generating operations (for $\qdn{\pm 1}$) and the quandle axioms as generating equations.
\end{ex}

\begin{defi}\label{def_top_alg}
 A \emph{topological $\mathbb T$-algebra} is a $\mathbb T$-algebra in the category of topological spaces. In other words, it is a topological space endowed with a $\mathbb T$-algebra structure where all the $\omega_X$ are continuous. Morphisms between topological $\mathbb T$-algebra are $\mathbb T$-algebra morphisms that are continuous.
 \end{defi}

 \begin{ex}
For the three examples above (Ex.~\ref{sets_as_T-alg}, \ref{groups_as_T-alg} and \ref{qdls_as_T-alg}), topological $\mathbb T$-algebras are respectively topological spaces, topological groups and topological quandles.
\end{ex}

\subsection{Congruences}\label{sec_congruences}

Quotients of algebras are dealt with using the general notion of congruence:

\begin{defi}\label{def_congruence}
A \emph{congruence} $\mathcal R$ on a $\mathbb T$-algebra $X$ is an equivalence relation on $X$ which is compatible with the operations: for all $n \geq 0$ and all $\omega \in \mathbb T_n$, if $(x_i), (y_i) \in X^n$ satisfy $x_i \mathcal R y_i$ for all $i$, then $\omega_X(x_1, \ldots, x_n)\mathcal R \hspace{2pt}\omega_X(y_1, \ldots, y_n)$.
\end{defi}

An equivalence relation $\mathcal R$ on a $\mathbb T$-algebra $X$ is a congruence if and only if there is a (clearly unique) $\mathbb T$-algebra structure on the set $X/\mathcal R$ making the projection $X \twoheadrightarrow X/\mathcal R$ into a morphism of $\mathbb T$-algebras.

\begin{rmq}
An equivalence relation $\mathcal R$ on a $\mathbb T$-algebra $X$ is a congruence if and only if $\mathcal R$ is a sub-$\mathbb T$-algebra of $X^2$. In other words, congruences are exactly internal equivalence relations in the category of $\mathbb T$-algebra.
\end{rmq}

\begin{defi}\label{def_ker}
Let $f: X \rightarrow Y$ be a morphism of $\mathbb T$-algebras. Its \emph{kernel} is the congruence $\mathcal R_f$ on $X$ defined by: $x \mathcal R_f y$ if and only if $f(x) = f(y)$.
\end{defi}

By construction, we have a canonical isomorphism $\ima(f) \cong X/\mathcal R_f$. In particular, every surjective morphism can be identified canonically with the quotient by a congruence relation.

\begin{defi}\label{def_index}
A congruence $\mathcal R$ on a $\mathbb T$-algebra $X$ is \emph{of finite index} if $X/\mathcal R$ is finite. Its \emph{index}, denoted by $[X: \mathcal R]$, is then the cardinal of $X/\mathcal R$.
\end{defi}

\paragraph{About the topological setting}

Let $X$ be a topological $\mathbb T$-algebra, and let $\mathcal R$ be a congruence on $X$. Then $X/\mathcal R$ is canonically a $\mathbb T$-algebra, and a topological space (with the quotient topology). But it is not always a topological $\mathbb T$-algebra. The problem is that $(X/\mathcal R)^n$, endowed with the product topology, is not in general a quotient of $X^n$, so the operations $(X/\mathcal R)^n \rightarrow X/\mathcal R$ need not be continuous. We have, however, the following:
\begin{lem}\label{quotient_of_top_algebras}
Let $X$ be a topological $\mathbb T$-algebra endowed with a congruence $\mathcal R$. If $X/\mathcal R$ is discrete, or if $X$ and $X/\mathcal R$ are locally compact (for instance if $X$ is compact), then $X/\mathcal R$ is canonically a topological $\mathbb T$-algebra.
\end{lem}

\begin{proof}
If $X/\mathcal R$ is discrete, then its algebra operations are automatically continuous. In the locally compact case, Corollary~\ref{prod_of_quotients} implies that $(X/\mathcal R)^n$ is a quotient of $X^n$, whence the conclusion.
\end{proof}

\begin{rmq}
In this paper, we consider mostly discrete quotients, so we do not need to worry about the aforementioned problems. In fact, the only non-discrete quotients of the paper will appear in~\cref{sec_proof_criterion}, and they will be quotients of compact algebras.
\end{rmq}

More subtleties appear when looking at the correspondence between congruences and surjective morphisms in the topological setting. Namely, the kernel of a morphism $f$ is defined as in the discrete case but, even when $X/\mathcal R_f$ is a topological $\mathbb T$-algebra, there is only a continuous bijection from  $X/\mathcal R_f$ to $\ima(f)$, which may not be a homeomorphism. In other words, in the topological setting, not every surjective morphism can be realized as the quotient by a congruence relation.

\subsection{Cofinite congruences and residual finiteness}\label{par_res_finite_alg}

We are interested in the relation between a topological algebra and its finite discrete quotients. In this context, an important role is played by topological algebras whose elements can be distinguished in such quotients: 

\begin{defi}\label{def_res_finite}
A topological $\mathbb T$-algebra $X$ is called \emph{residually finite} if morphisms to finite discrete $\mathbb T$-algebras separate points, that is, if for any $x, y \in X$, if $x \neq y$, then there exists a finite discrete $\mathbb T$-algebra $F$ and a continuous morphism $f : X \rightarrow F$ such that $f(x) \neq f(y)$. 
\end{defi}

Recall that a topological space $X$ is called \emph{totally disconnected} if its connected components are points; equivalently, for all $x, y \in X$, if $x \neq y$, then there is a partition $X = U \sqcup V$ into a union of two disjoint open subsets such that $x \in U$ and $y \in V$.

\begin{fact}\label{res_finite_are_tot_disconnected}
A residually finite topological $\mathbb T$-algebra $X$ is totally disconnected.
\end{fact}

\begin{proof}
Let $x, y \in X$, and let $f : X \rightarrow F$ be a continuous morphism to a finite discrete $\mathbb T$-algebra such that $f(x) \neq f(y)$. Then $\{x\}$ and its complement are open in the discrete space $F$, so $X = f^{-1}(x) \sqcup f^{-1}(F - \{x\})$ is a partition of $X$ into two open sets, with $x \in f^{-1}(x)$ and $y \in f^{-1}(F - \{x\})$. 
\end{proof}

A congruence $\mathcal R$ on a topological $\mathbb T$-algebra $X$ is open in $X^2$ if and only if its classes are open, which is equivalent to the quotient topology on $X/\mathcal R$ being discrete. Then the $\mathbb T$-algebra operations on $X/\mathcal R$ are automatically continuous, and the projection $X \twoheadrightarrow X/\mathcal R$ is a morphism of topological $\mathbb T$-algebra. Notice that the classes of an open congruence are also closed: they are \emph{clopen}. 

\begin{defi}
A congruence $\mathcal R$ on a topological $\mathbb T$-algebra $X$ is \emph{cofinite} if $X/\mathcal R$ is finite and discrete.
\end{defi}

In other words, a cofinite congruence is an open congruence of finite index. In particular, for congruences on discrete $\mathbb T$-algebras, being cofinite is just being \emph{of finite index} (see Definition~\ref{def_index}).

\begin{rmq}
A topological $\mathbb T$-algebra $X$ is residually finite if and only if the intersection of the cofinite congruences in $X^2$ is reduced to the diagonal.
\end{rmq}

\begin{lem}\label{intersection_finite_index}
If $\mathcal R$ and $\mathcal S$ are cofinite congruences, then so is $\mathcal R \cap \mathcal S$.
\end{lem}

\begin{proof}
Let $\mathcal R$ and $\mathcal S$ be cofinite congruences on the topological $\mathbb T$-algebra $X$. The intersection $\mathcal R \cap \mathcal S$ is clearly a congruence. In order to show that it is cofinite, consider the continuous map $\pi : x \mapsto (\overline x, \overline x)$ from $X$ to $X/\mathcal R \times X/\mathcal S$. We have $\pi(x) = \pi(y)$ if and only if $x \mathcal R y$ and $x \mathcal S y$, that is, if and only if $x (\mathcal R \cap \mathcal S) y$, so $\pi$ induces a continuous injection of $X/(\mathcal R \cap \mathcal S)$ into the finite discrete $\mathbb T$-algebra $X/\mathcal R \times X/\mathcal S$, whence the conclusion.
\end{proof}

\subsection{A finiteness result}\label{subsec_finite_nb_of_cong}

There is an obvious notion of a sub-$\mathbb T$-algebra of a $\mathbb T$-algebra $X$ \emph{generated} by a subset $\Gamma$ of $X$. It is both the intersection of sub-$\mathbb T$-algebras of $X$ containing $\Gamma$, and the subset of element obtained by applying operations to elements of $\Gamma$. 

\begin{defi}\label{def_topo_fg}
A $\mathbb T$-algebra is called \emph{finitely generated} if it has a finite set of generators. A topological $\mathbb T$-algebra is called \emph{topologically finitely generated} if it has a dense finitely generated sub-$\mathbb T$-algebra.
\end{defi}

\begin{rmq}
A discrete $\mathbb T$-algebra is topologically finitely generated if and only if it is finitely generated.
\end{rmq}

The following observation is a straightforward generalization of the following fact: the number of group structures on $\{1, ..., n\}$ is finite, bounded by the number of possible Cayley tables of size~$n \times n$.

\begin{prop}\label{algebras_of_fixed_cardinality}
Let $\mathbb T$ be a finitely generated Lawvere theory. For every integer $n \in \mathbb N$, there are finitely many isomorphism classes of $\mathbb T$-algebras of cardinality $n$. 
\end{prop}

\begin{proof}
Let $\omega_1, ..., \omega_N$ be generating operations for $\mathbb T$, with $\omega_i \in \mathbb T_{k_i}$ (for $i = 1, ..., N$). Every $\mathbb T$-algebra of cardinality $n$ is isomorphic to $ X=\{1, ..., n\}$ endowed with some $\mathbb T$-algebra structure. But a $\mathbb T$-algebra structure on $X$ is uniquely determined by the finitely many maps $\omega_{i,X} : X^{k_i} \rightarrow X$, for which there are finitely many possible choices.
\end{proof}

\begin{prop}\label{congruences_of_fixed_index}
Let $\mathbb T$ be a finitely generated Lawvere theory. Let $X$ be a topologically finitely generated $\mathbb T$-algebra and $n \geq 1$ be an integer. There are finitely many open congruences of index $n$ on $X$.
\end{prop}

\begin{proof}
Let $Y_1, ..., Y_r$ be representatives of the isomorphism classes of $\mathbb T$-algebras of cardinality $n$, whose number is finite by Proposition~\ref{algebras_of_fixed_cardinality}. If $\mathcal R$ is an open congruence of index $n$ on $X$, then the finite discrete $\mathbb T$-algebra $X/\mathcal R$ is isomorphic to some $Y_i$, and $\mathcal R$ can be recovered as the kernel of the corresponding surjection $f: X \twoheadrightarrow X/\mathcal R \cong Y_i$, which is continuous for the discrete topology on $Y_i$. Let $x_1, ..., x_s$ be a finite set of topological generators for $X$. Then a continuous morphism $f: X \rightarrow  Y_i$ is determined by the element $(f(x_1), ..., f(x_s))$ of the finite set $Y_i^s$. Hence, the number of cofinite congruences of index $n$ on $X$ is bounded by $|Y_1|^s + \cdots + |Y_r|^s$, so it is finite.
\end{proof}

\begin{rmq}\label{rk_Singh_finite_nb_of_cofinite}
For quandles, this is also~\cite[Cor.~5.2]{Singh}.
\end{rmq}

\section{Profinite structures}\label{sec_prof_alg}

This section is devoted to showing the equivalence, for a topologically finitely generated algebra $X$ over a Lawvere theory, between the data of the number of morphisms from $X$ to every finite discrete algebras and the data of the profinite completion of $X$, up to continuous isomorphism (Theorem~\ref{theo_col_and_prof_gen}).

\subsection{Profinite algebras}\label{par_prof_alg}

We recall here some classical material about profinite spaces, groups and, more generally, algebras over a Lawvere theory (which is the framework of~\cite{Schneider}).

\subsubsection{Definition and first properties}

Let $\mathbb T$ be a Lawvere theory.

\begin{defi}\label{def_prof_alg}
A topological $\mathbb T$-algebra $X$ is called \emph{profinite} if it is a projective limit of finite discrete $\mathbb T$-algebras.
\end{defi}

Let $X$ be the projective limit of a directed system $(X_i, \psi_{ij})$ of finite (discrete) $\mathbb T$-algebras, indexed by a poset $I$, where $\psi_{ij}$ is morphism from $X_i$ to $X_j$ if $i \geq j$. Recall that:
\[X := \varprojlim X_i = \{(x_i)\ |\ \psi_{ij}(x_i) = x_j\ \text{if}\ i \geq j\} \subseteq \prod\limits_{i \in I} X_i.\]
It is a sub-$\mathbb T$-algebra of the product (whose operations are defined pointwise). This structure is the unique one on the set $X$ making all the canonical maps $\psi_i: X \rightarrow X_i$ into $\mathbb T$-algebra morphisms. 

The topology of the projective limit is defined as the restriction of the product topology; equivalently, it is the initial topology with respect to the canonical maps $\psi_i: X \rightarrow X_i$, which is generated by the $\psi_i^{-1}(t)$, for $i \in I$ and $t \in X_i$. Notice that the operations on are continuous for this topology (as restrictions of the pointwise operations on the product), so $X$ is a \emph{topological $\mathbb T$-algebra}. 

The product is compact by Tychonoff's theorem, and the equations $\psi_{ij}(x_i) = x_j$ define closed subsets of the product, whose intersection is a closed subset, so $X$ is compact. Moreover, the canonical maps $\psi_i: X \rightarrow X_i$ are continuous morphisms to finite discrete $\mathbb T$-algebras separating the points of $X$, so $X$ is residually finite (and, in particular, totally disconnected -- see Fact~\ref{res_finite_are_tot_disconnected}).

\begin{prop}\label{profinite_iff_cpct_and_res_finite}
A topological $\mathbb T$-algebra is profinite if and only if it is compact and residually finite.
\end{prop}

\begin{proof}
We have already proved one direction; let us prove the converse. Let $X$ be a topological $\mathbb T$-algebra that is compact and residually finite. Consider the canonical morphism $\iota$ from $X$ to $\widehat X := \varprojlim X/\mathcal R$, where the limit is taken over all cofinite congruences on $X$ (that is, over all finite discrete quotients of $X$). It is continuous (by construction), and injective, because of residual finiteness. Since $X$ is compact, and the limit is Hausdorff, $\iota$ is a homeomorphism onto its image, which is then compact, hence closed. But the maps $X \rightarrow X/\mathcal R$ are surjective. Given the above description of the topology of the projective limit, this means that $\iota$ has dense image. Hence it is a homeomorphism.
\end{proof}

For quandles, this is essentially~\cite[Th.~4.8]{Singh}. As an immediate corollary, we get a generalization to algebras of~\cite[Prop.~4.7]{Singh} (see also~\cite[Prop.~3.23]{Yetter_et_al}):

\begin{cor}
Let $Y$ be a $\mathbb T$-subalgebra of a topological $\mathbb T$-algebra $X$. Then $Y$ is profinite if and only if it is closed in $X$.
\end{cor}

\subsubsection{Examples: profinite groups and profinite spaces}\label{par_profinite_groups_and_spaces}

\paragraph{Profinite spaces}

We can apply the theory of profinite $\mathbb T$-algebras to the theory having the empty presentation. Then topological $\mathbb T$-algebras are just topological spaces. 

\begin{fact}\label{res_finite_spaces}
A topological space is residually finite if and only if it is totally disconnected.
\end{fact}

\begin{proof}
One implication is a particular case of Fact~\ref{res_finite_are_tot_disconnected}; let us prove the converse. Let $X$ be a totally disconnected topological space, and let $x, y \in X$ with $x \neq y$. There is a partition $X = U \sqcup V$ into a union of two disjoint open subsets such that $x \in U$ and $y \in V$. Then the application $f : X \rightarrow \{0, 1\}$ defined by $f(U) = \{1\}$ and $f(V) = \{0\}$ is a continuous map to a finite discrete space, and $f(x) \neq f(y)$. This shows that $X$ is residually finite.
\end{proof}

The general definition of profinite algebras (Definition~\ref{def_prof_alg}) becomes, in this context:
\begin{defi}\label{def_prof_space}
A \emph{profinite space} is a topological space which is a projective limit of finite discrete spaces.
\end{defi}

Then, Proposition~\ref{profinite_iff_cpct_and_res_finite}, together with Fact~\ref{res_finite_spaces}, gives the following characterisation: 
\begin{cor}[to Proposition~\ref{profinite_iff_cpct_and_res_finite}]\label{char_profinite_spaces}
A topological space is profinite if and only if it is compact and totally disconnected.
\end{cor}

Such spaces are also called \emph{Stone spaces} (which is the terminology used in~\cite{Schneider}).

\paragraph{Profinite groups}

One important feature of groups is their homogeneity, that allows one to describe a congruence $\mathcal R$ by the class $N$ of the neutral element, which is a normal subgroup (then $\mathcal R$ is congruence modulo $N$). In particular, on a topological group, open congruences correspond to open normal subgroups.

\begin{prop}[{\cite[III, \S 4, n°6, Cor.~1 to Prop.~13]{Bourbaki}}]\label{prop_Bourbaki}
In a compact totally disconnected group, open normal subgroups form a basis of neighbourhoods of $1$.
\end{prop}

\begin{rmq}
In \cite{Bourbaki}, this is stated for open subgroups, but we can easily show that every open subgroup contains an open normal subgroups. Namely, if $N$ is an open subgroup of a compact totally disconnected group $G$, then $G/N$ is discrete and compact, so it is finite. Moreover, the conjugate $gNg^{-1}$ depends only on $gN \in G/N$, so their intersection is a finite intersection of open subgroups, which is then an open normal subgroup contained in $N$.
\end{rmq}

In other words, compact totally disconnected groups are residually finite: for every $x,y \in G$, if $x \neq y$, then there exists an open normal subgroup $N$ not containing $x^{-1}y$, corresponding to a finite discrete quotient $G/N$ where the class $xN$ of $x$ is different from the class $yN$ of $y$. This gives a converse to Fact~\ref{res_finite_are_tot_disconnected} in this context. Combined with Proposition~\ref{profinite_iff_cpct_and_res_finite}, this entails:
\begin{cor}[to Proposition~\ref{profinite_iff_cpct_and_res_finite}]\label{profiniteness_criterion_groups}
A topological group is profinite if and only if it is compact and totally disconnected.
\end{cor}

\subsubsection{Profinite completions}\label{sec_prof_comp}

The projective limit appearing in the proof of Proposition~\ref{profinite_iff_cpct_and_res_finite} can be considered for any topological $\mathbb T$-algebra:

\begin{defi}\label{def_prof(alg)}
The \emph{profinite completion} of a topological $\mathbb T$-algebra $X$ is the projective limit of its finite discrete quotients:
\[\widehat X = \varprojlim(X/\mathcal R),\]
where the projective limit is taken over the \emph{cofinite} congruences $\mathcal R$ on $X$ (that is, open congruences of finite index), ordered by inclusion.
\end{defi}

\begin{rmq}
It is a filtered limit, because of Lemma~\ref{intersection_finite_index}. 
\end{rmq}

\begin{rmq}
Recall that on a discrete $\mathbb T$-algebra, a cofinite congruence is just a finite-index one.
\end{rmq}

The projections $\pi_{\mathcal R} : X \twoheadrightarrow X/\mathcal R$ factor through the projective limit, giving a canonical morphism $\iota : X \rightarrow \widehat X$. Since the $\pi_{\mathcal R}$ are surjective, $\iota$ has dense image.

\begin{rmq}\label{rk_entourages}
The set of cofinite congruences on a topological $\mathbb T$-algebra $X$ is a \emph{fundamental system of entourages} for a \emph{uniform structure} on $X$. With this language, $\widehat X$ is indeed the \emph{completion} of $X$, with respect to this uniform structure \cite[II, \S 3, n°9, Prop.~18]{Bourbaki}.
\end{rmq}

The profinite completion has the expected universal property:
\begin{prop}\label{univ_prop_completion_finite}
Let $F$ be a finite $\mathbb T$-algebra, endowed with the discrete topology. Then restriction along the canonical morphism $\iota : X \rightarrow \widehat X$ induces a natural bijection:
\[\Hom(X, F) \cong \Hom( \widehat X, F)\] 
between sets of continuous morphisms.
\end{prop}

\begin{proof}
The map  $(-) \circ \iota$ is injective because of the density of $\iota(X)$ in $\widehat X$. Let us show that it is surjective. Given a continuous morphism $f : X \rightarrow F$, we can factor it into a projection $p_f: X \twoheadrightarrow X/\mathcal R_f$ followed by an injection $i_f: X/\mathcal R_f \hookrightarrow F$. Note that the continuity of $i_f$ implies that $X/\mathcal R_f$ is discrete, hence it is canonically a topological $\mathbb T$-algebra, and this is a factorisation into continuous algebra morphisms (see the discussion in~\cref{sec_congruences}); in particular, $\mathcal R_f$ is a cofinite congruence. Then define $\hat f : \widehat X \rightarrow F$ by the canonical projection $\pi_{\mathcal R_f}: \widehat X \twoheadrightarrow X/\mathcal R_f$, followed by $i_f$. It is a continuous morphism ($i_f$ being a continuous map between discrete spaces), and $\hat f \iota = i_f \pi_{\mathcal R_f} \iota = i_f p_f = f$, whence the result.
\end{proof}

\begin{rmq}
In more topological terms, this proof can be seen as a proof that any morphism from $X$ to a finite discrete $F$ is uniformly continuous for the profinite topology; then $\hat f$ is the unique factorisation of $f$ through the completion.
\end{rmq}

\begin{cor}\label{univ_prop_completion}
For every profinite $\mathbb T$-algebra $Y$ and every topological $\mathbb T$-algebra $X$, restriction along the canonical morphism $\iota : X \rightarrow \widehat X$ induces a natural bijection:
\[\Hom(X, Y) \cong \Hom( \widehat X, Y).\]
In other words, the functor $X \mapsto \widehat X$ from topological $\mathbb T$-algebras to profinite $\mathbb T$-algebras is left adjoint to the forgetful functor.
\end{cor}

\begin{proof}
Let us write $Y$ as the projective limit of finite (discrete) $\mathbb T$-algebras $Y_i$. Notice that the projective limit in the category of topological $\mathbb T$-algebras is the same as the projective limits in $\mathbb T$-algebras, endowed with some topology. Hence, we can write:
\[\Hom(X, \varprojlim Y_i) =  \varprojlim \Hom(X, Y_i) \cong \varprojlim \Hom(\widehat X, Y_i) = \Hom( \widehat X, \varprojlim Y_i),\]
where the middle isomorphism is the one from Proposition~\ref{univ_prop_completion_finite}.
\end{proof}

\begin{rmq}
If $X$ is a discrete algebra, then for any topological algebra $Y$, the set $\Hom(X, Y)$ is just the set of all algebra morphisms from $X$ to $Y$. In particular, if $Y$ is profinite (in particular if it is finite discrete), any algebra morphism $X \rightarrow Y$ corresponds to a unique continuous morphism $\widehat X \rightarrow Y$.
\end{rmq}

\subsection{Profinite completion of finitely generated algebras}
\label{subsec_prof_comp_fg}

We now generalize the main theorem of~\cite{DFPR} to the context of topological algebras over a Lawvere theory, which includes the case of quandles. The proof is in fact mostly the same as theirs, replacing finite index normal subgroups of discrete groups by cofinite congruences on topological algebras. In particular, Proposition~\ref{congruences_of_fixed_index} plays the role of the finiteness of the number of normal subgroups of fixed index in a finitely generated group.

\begin{rmq}\label{Uniform_qdls}
The language of uniform structures from Remark~\ref{rk_entourages} is useful when dealing with general $\mathbb T$-algebras: contrary to what happens for groups, in general, a uniform structures compatible with a $\mathbb T$-algebra structure cannot be described by neighbourhoods of one point. Although we do not use explicitly the language of uniform structures in this paper, the reader who knows about them is invited to keep this point of view in mind as a useful guide for their intuition.
\end{rmq} 

\subsubsection{Profinite completion in the finitely generated case}\label{par_prof_comp_fg}

If $X$ is a topologically finitely generated topological algebra over a finitely generated theory $\mathbb T$, one can see the projective limit defining its profinite completion (see Def.~\ref{def_prof(alg)}) as a projective limit indexed by integers, using the following trick:
\[\mathcal R_n(X) := \bigcap\limits_{\mathcal R\ \mathrm{open},\ [X: \mathcal R] \leq n} \mathcal R.\]
Proposition~\ref{congruences_of_fixed_index} ensures that this intersection is finite; then Lemma~\ref{intersection_finite_index} ensures that it is cofinite. By construction, $\mathcal R_{n+1} \subseteq \mathcal R_n$, and this sequence is cofinal in cofinite congruences on $X$. Hence:
\[\widehat X = \varprojlim\limits_n(X/\mathcal R_n).\]

\begin{rmq}
The $\mathcal R_n$ are a countable fundamental system of entourages for the uniform structure from Remark~\ref{Uniform_qdls}, so this structure comes from a generalized distance on $X$, given by: 
\[d(x,y) = 2^{- \sup \{n | x \mathcal R_n y\}}.\]
The reader unfamiliar with uniform structures may think of $\widehat X$ as the usual completion of the (generalized) metric space $X$.
\end{rmq}

\subsubsection{Number of morphisms and finite quotients}\label{par_nb_of_morphisms}

If a topological $\mathbb T$-algebra $X$ is (topologically) finitely generated and $F$ is a finite discrete $\mathbb T$-algebra, then the number $\# \Hom(X, F)$ of continuous morphisms from $X$ to $F$ is finite, since a continuous morphism is determined by the images of the generators.

\begin{nota}
We denote by $\Phi(X)$ the set of isomorphism classes of finite discrete quotients of $X$. 
\end{nota}

\begin{lem}\label{C(X)_from_colourings}
Let $X$ and $Y$ be finitely generated topological $\mathbb T$-algebras. Suppose that for every finite discrete $\mathbb T$-algebra $F$, we have $\# \Hom(X, F) = \# \Hom(Y, F)$. Then $\Phi(X) = \Phi(Y)$.
\end{lem}

\begin{proof}
Let $F$ be a finite discrete $\mathbb T$-algebra. Let $F_1, ..., F_N$ be a list of all its strict subalgebras. Then the number of surjective morphisms from $X$ to $F$ is given by the inclusion-exclusion formula:
\[\# \Hom(X, F) - \sum\limits_i \# \Hom(X, F_i) + \sum\limits_{i < j} \# \Hom(X, F_i \cap F_j) - \ldots,\]
so it is entirely determined by the numbers of morphisms from $X$ to finite $\mathbb T$-algebras. Since it is non-zero if and only if $F$ belongs to $\Phi(X)$, we conclude that $\Phi(X)$ is completely determined by these numbers, whence the result.
\end{proof}

\subsubsection{Profinite completion and morphisms to finite algebras}\label{par_prof_comp_and_morphisms}

We now prove the following theorem, which is a generalization of the main theorem of~\cite{DFPR} to algebras over a Lawvere theory. The proof is in fact exactly the same as theirs, apart from the general argument in the proof of Proposition~\ref{congruences_of_fixed_index}, replacing their group-theoretic argument.

\begin{theo}\label{theo_col_and_prof_gen}
Let $\mathbb T$ be a finitely generated Lawvere theory. Let $X$ and $Y$ be topologicaly finitely generated topological $\mathbb T$-algebras. The following are equivalent:
\begin{enumerate}[itemsep=-0.3em, topsep = 4pt]
\item For every finite discrete $\mathbb T$-algebra $F$, $\# \Hom(X, F) = \# \Hom(Y, F)$;
\item For every finite discrete $\mathbb T$-algebra $F$, $X$ surjects onto $F$ if and only if $Y$ does.
\item $\widehat X$ and $\widehat Y$ are (continuously) isomorphic.
\end{enumerate}
\end{theo}

\begin{proof}
We have already seen that the first statement implies the second one (this is Lemma~\ref{C(X)_from_colourings}). The fact that the third one implies the first one follows directly from Proposition~\ref{univ_prop_completion_finite}. Let us now suppose that the second statement holds, that is, that $\Phi(X) = \Phi(Y)$, and let us show that $\widehat X$ and $\widehat Y$ are continuously isomorphic. Let us consider the congruences $\mathcal R_n(X)$ and $\mathcal R_n(Y)$ from~\cref{subsec_prof_comp_fg}, that we denote only by $\mathcal R_n$ when the underlying algebra is clear. We have $X/\mathcal R_n \in \Phi(X) = \Phi(Y)$, so there is a congruence $\mathcal R$ on $Y$ such that $Y/\mathcal R \cong X/\mathcal R_n$. By definition of $\mathcal R_n$, $X/\mathcal R_n$ injects into a product of discrete $\mathbb T$-algebras having at most $n$ elements. Hence so does $Y/\mathcal R$, which means that $\mathcal R$ is an intersection of open congruences of index at most $n$ on $Y$. As a consequence, $\mathcal R \supseteq \mathcal R_n(Y)$. This implies that $X/\mathcal R_n$ is isomorphic to a quotient of $Y/\mathcal R_n$. By symmetry, $Y/\mathcal R_n$ is also a quotient of $X/\mathcal R_n$. Since these are finite discrete $\mathbb T$-algebras, we conclude that $X/\mathcal R_n \cong Y/\mathcal R_n$. In other words, the set $I_n$ of isomorphisms between $X/\mathcal R_n$ and $Y/\mathcal R_n$ is nonempty. By construction, we have a canonical isomorphism $(X/\mathcal R_n)/\mathcal R_{n-1} \cong X/\mathcal R_{n-1}$. Since any isomorphism between $X/\mathcal R_n$ and $Y/\mathcal R_n$ must send $\mathcal R_{n-1}(X/\mathcal R_n)$ to $\mathcal R_{n-1}(Y/\mathcal R_n)$, we obtain a canonical map $I_n \rightarrow I_{n-1}$. Let us consider the projective limit $I_\infty$ of the $I_n$. These sets are finite and nonempty, so $I_\infty$ is not empty (one can see it as a decreasing intersection of nonempty closed subsets of the compact $\prod I_n$, by considering subsets defined by more and more coherence conditions). An element $\varphi \in I_\infty$ is a coherent sequence of isomorphisms $\varphi_n : X/\mathcal R_n \cong Y/\mathcal R_n$ (which are continuous, seen as maps between discrete spaces), so it induces a continuous isomorphism between  $\widehat X = \varprojlim (X/\mathcal R_n)$ and $\widehat Y = \varprojlim (Y/\mathcal R_n)$.
\end{proof}

\subsection{Colourings and profinite completions of quandles}
\label{subsec_colourings_vs_completion}

We now apply Theorem~\ref{theo_col_and_prof_gen} to (discrete) quandles to show the equivalence, as link invariants, between the numbers of colourings by finite quandles, and the profinite completion of the fundamental quandle. 

Let us first spell out the application of Theorem~\ref{theo_col_and_prof_gen} to (discrete) quandles:

\begin{theo}\label{theo_col_and_prof}
For finitely generated quandles $Q$ and $R$, the following are equivalent:
\begin{enumerate}[itemsep=-0.3em, topsep = 4pt]
\item For every finite quandle $X$, $\# \Hom(Q, X) = \# \Hom(R, X)$;
\item $\Phi(Q) = \Phi(R)$;
\item $\widehat Q$ and $\widehat R$ are (continuously) isomorphic.
\end{enumerate}
\end{theo}

Recall that $\Phi(Q)$ denotes the set of isomorphism classes of finite quotients of a quandle $Q$, and that $\widehat Q$ is the profinite completion of $Q$, obtained by applying Definition~\ref{def_prof(alg)} to (discrete) quandles:
\[\widehat Q = \varprojlim(Q/\mathcal R),\]
where the projective limit is taken over the finite-index congruences $\mathcal R$ on $Q$, ordered by inclusion. In other words, it is the projective limit of the finite quotients of $Q$.

\paragraph{Colourings and quandle morphisms} The application of Theorem~\ref{theo_col_and_prof} to link theory uses the following classical fact: there is a one-to-one correspondence between colourings of a diagram of a link $L$ by a quandle $X$ and morphisms from the fundamental quandle $Q_L$ to $X$. One way to look at this is to see the fundamental quandle as given by the Wirtinger presentation (for any choice of diagram of $L$), seen as a quandle presentation (where conjugation by a generator $x_i$ is replaced with $x_i \qd (-)$). Recall that this presentation has the arcs of the diagrams as generators, and relations corresponding to crossings (see Figure~\ref{fig_colourings}). As a consequence, a morphism from $Q_L$ to $X$ is exactly the data of one element of $X$ for each generator (one colour for each arc), such that these elements satisfy the relations (see the example depicted in Figure~\ref{fig_colourings}). This data is exactly a colouring of (the chosen diagram of) $L$ by the quandle $X$, as claimed.

\begin{figure}[!htbp]
\centering    
\subfigure{\label{fig:a}\includegraphics[width=5em]{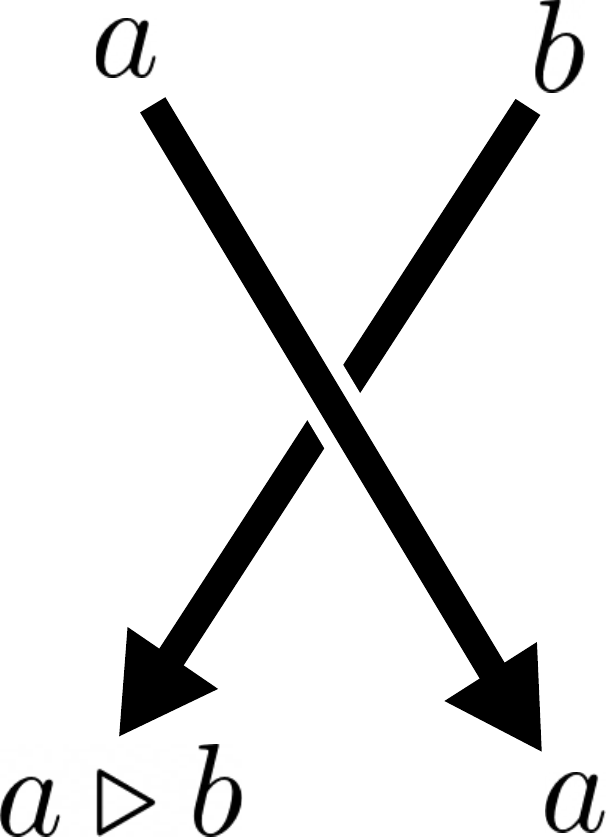}}
\hspace{3.5em}
\subfigure{\label{fig:b}\includegraphics[width=5.2em]{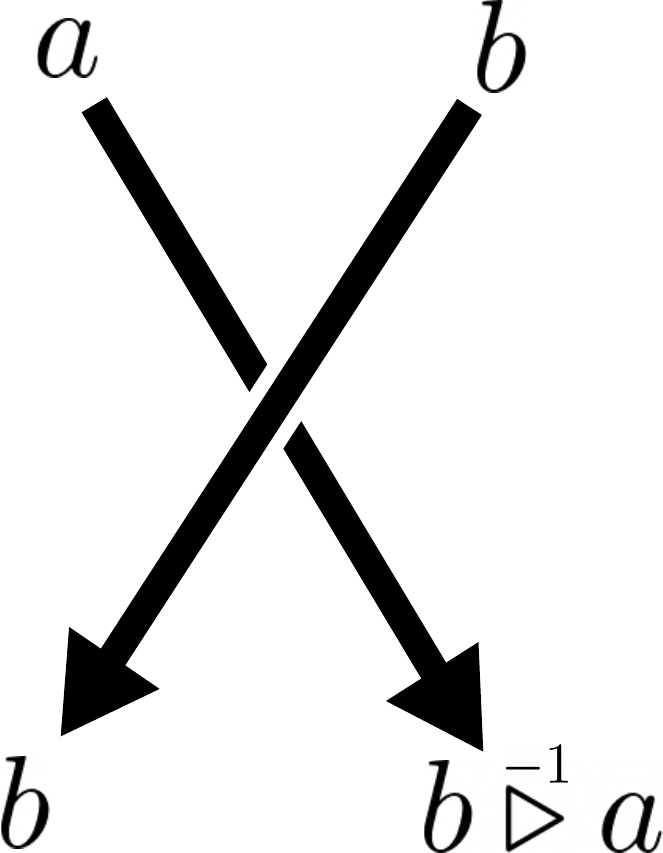}}
\hspace{2.5em}
\raisebox{0.8em}{\subfigure{\label{fig:c}\includegraphics[width=6em]{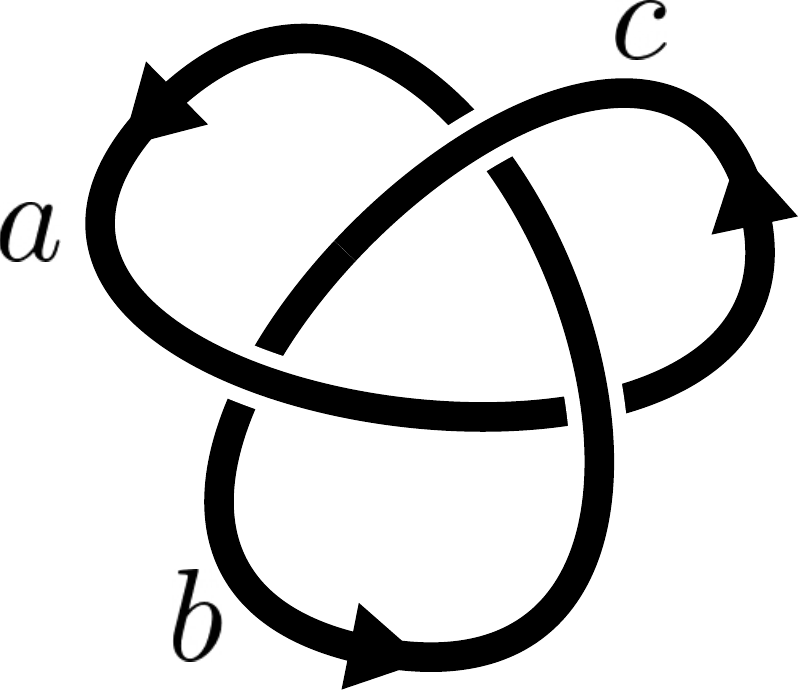}}}
\hspace{2.5em}
\raisebox{0.8em}{\subfigure{\label{fig:d}\includegraphics[width=6em]{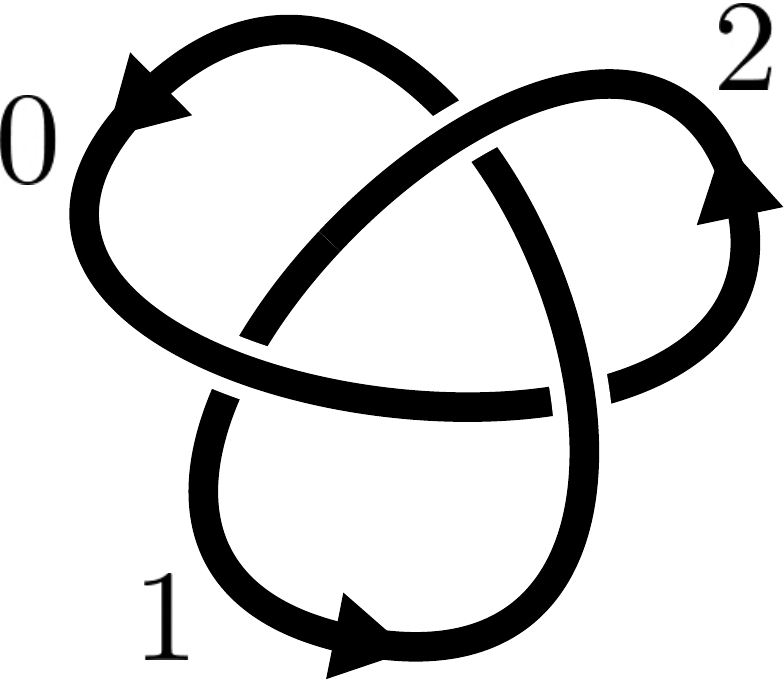}}}
\caption{If $L$ is the trefoil pictured here, with the convention pictured on the left for the Wirtinger relations, we have the presentation $Q_L = \langle a,b,c \ |\ c = b \qd a, b = a \qd c,\ a = c \qd b  \rangle$. If $X$ is the quandle $\Z/3$, endowed with the Alexander law $i \qd j = 2j -i$, then the colouring of $L$ by elements of $X$ shown on the right corresponds to the morphism $Q_L$ to $X$ defined by $a \mapsto 0$, $b \mapsto 1$ and $c \mapsto 2$ (the usual tricolouring of the trefoil knot).}
\label{fig_colourings}
\end{figure}

\begin{rmq}
Fundamental quandles of links are always finitely generated, since the number of arcs of a diagram is always finite. They are in fact always finitely presented.
\end{rmq}

As a consequence, we get:

\begin{cor}\label{cor_col_and_prof}
Two links $K$ and $L$  have the same number of colourings by every finite quandle if and only if $\widehat Q_K$ and $\widehat Q_L$ are (continuously) isomorphic.
\end{cor}

\section{Topological quandles}\label{sec_quandles}

Our second main goal is to establish a precise link between the profinite completion of a quandle and the profinite completion of its associated peripheral system. This will be achieved at the end of~\cref{sec_ProPph} with Corollary~\ref{equiv_btw_prof_comp_of_periph} which, combined with the results already obtained, will entail our main theorem (Thm.~\ref{main_thm}). 

In order to reach that goal, we need to develop some profinite quandle theory. We will do so in~\cref{sec_ProQdl}, but we first need to investigate topological quandle theory, which is the subject of the current section. Namely, we extend to topological quandles the basic notions of quandle theory, including the enveloping group, the adjoint action, and the trivialization of a quandle. The reader can keep in mind that usual quandle theory can be recovered as the theory of \emph{discrete} quandles, for which all this is classical material, that can be found for instance in \cite{Joyce, Andruskiewitsch-Grana, Eisermann}.

\subsection{Basic definitions}

\begin{defi}\label{def_qdl}
A \emph{quandle} is a set $Q$ endowed with a binary operation $\qd$ satisfying the following axioms:
\begin{enumerate}[itemsep=-0.3em, topsep = 4pt]
\item $\forall x, y, z \in Q,\ x \qd (y \qd z) = (x \qd y) \qd (x \qd z),$
\item $\forall x \in Q,\ x \qd (-) : Q \rightarrow Q$ is a bijection,
\item $\forall x \in Q,\ x \qd x = x$.
\end{enumerate}
Quandles and their morphisms form a category, denoted by $\mathcal Qdl$.
\end{defi}

The first axiom (called \emph{self-distributivity}) means that the maps $\sigma_x := x \qd(-)$ are endomorphisms of $(Q, \qd)$. The second axiom says that they are in fact automorphisms.

\begin{nota}
If $x$ is an element of the quandle $Q$, we denote by $x \qdn{-1} (-)$ the inverse of $x \qd (-)$. More generally, for any $n \in \Z$, we denote by $x \qdn{n} (-)$ the $n$-th iterate of $x \qd (-)$, for any $n \in \Z$. Notice that for any $n \in \Z$, $(Q, \qdn{n})$ is again a quandle.
\end{nota}

The following classical definition is a specialization of the general notion of a topological algebra  (Definition~\ref{def_top_alg}) to quandles (see Example~\ref{qdls_as_T-alg}):
\begin{defi}
A \emph{topological quandle} is a quandle endowed with a topology such that the two maps $\qdn{\pm 1} : Q^2 \rightarrow Q$ are continuous. A \emph{morphism} of topological quandles is a quandle morphism that is also continuous. This defines a category $\mathcal{TQ}dl$ of topological quandles.
\end{defi} 

Notice that for a quandle $Q$ endowed with a topology, these requirement are in general not equivalent to $x \mapsto x \qdn{\pm 1} (-)$ being continuous maps $Q \rightarrow \Homeo(Q) \subseteq \mathcal C(Q,Q)$, unless $Q$ is locally compact. This equivalence would be true more generally for $Q$ compactly generated, provided that the product $Q^2$ in the above definition were replaced with the compactly generated product, weakening the continuity requirement (see~\cref{sec_topo_spaces}).

\begin{ex}[A counter-example]
Let $X$ be a topological space, and let $f : X \rightarrow X$ be any bijection such that $f^{-1}$ is not continuous. Let $Q $ be $X \sqcup \{t\}$, endowed with the quandle structure defined by $t \qd x = f(x)$ for all $x \in X$, and $a \qd b = b$ unless $a = t$ and $b \in X$. Then $\qd : Q^2 \rightarrow Q$ is continuous, but $\qdn{-1}$ is not, so $Q$ is not a topological quandle.
\end{ex}

\begin{ex}[Trivial topological quandles]\label{triv_topo_qdl}
Any topological space $X$ is a topological quandle, for the law defined by $x \qd y = y$ for all $x,y \in X$. Such topological quandles are called \emph{trivial}. A morphism between two trivial topological quandles is just a continuous map between the underlying topological spaces.
\end{ex}

\begin{ex}[Conjugation quandles]\label{ex_c(G)}
Any topological group $G$ is a topological quandle for the conjugation law, which is defined by $g \qd h = {}^g\!h = ghg^{-1}$ for all $g, h \in G$. The quandle $(G, \qd)$ is denoted by $c(G)$, and called the \emph{conjugation quandle} associated with $G$. Notice that a (continuous) group morphism $G \rightarrow H$ is also a (continuous) quandle morphism $c(G) \rightarrow c(H)$.
\end{ex}

\begin{conv}
Conjugation is the default quandle structure on a group $G$: it is the one implied when no other quandle structure is specified. In other words, we often denote $c(G)$ just by $G$. For instance, we may speak of morphisms from a quandle $Q$ to a group $G$, meaning quandle morphisms $Q \rightarrow c(G)$.
\end{conv}

\subsection{Quotients and congruences}

In order to deal efficiently with quotients, we use the specialization to quandles of the general notion of \emph{congruence} (see Definition~\ref{def_congruence}):
\begin{defi}
A \emph{congruence} $\mathcal R$ on a quandle $Q$ is an equivalence relation on $Q$ which is compatible with the quandle law: if $x \mathcal R x'$ and $y \mathcal R y'$, then $(x \qd  y) \mathcal R (x' \qd y')$ and $(x \qdn{- 1}  y) \mathcal R (x' \qdn{- 1} y')$
\end{defi}

If $\mathcal R$ is a congruence on $Q$, we will always consider $Q/\mathcal R$ as a quandle for the unique quandle structure making the projection $Q \twoheadrightarrow Q/\mathcal R$ into a morphism, that is, via $\overline x  \qd \overline y = \overline{x \qd y}$ (where $\overline x$ is the class of $x$ modulo $\mathcal R$).

\begin{rmq}
It is not enough to ask that $(x \qd  y) \mathcal R (x' \qd y')$ if $x \mathcal R x'$ and $y \mathcal R y'$ to get a quandle law on the quotient: the second quandle axiom may no longer be true there. For instance, take $Q_{0,1} = \Z \sqcup \{x\}$, whose law is given by $x \qd n = n + 1$ if $n \in \Z$, and $a \qd b = b$ if $a \neq x$ or $b = x$. This is a quandle, and the relation $\mathcal R$ identifying all the non-negative integers satisfies the above condition, but the quotient is no longer a quandle, since $x \qd(-)$ is no longer bijective there. 
\end{rmq}

A number of useful congruences come from group actions:
\begin{lem}\label{quotient_by_group}
Let $Q$ be a quandle and let $\alpha : G \rightarrow \Aut(Q)$ be a group action whose image is normalized by the $x \qd (-)$, for $x \in Q$. Then the equivalence relation whose classes are $G$-orbits is a congruence. The quandle of $G$-orbits will be called the \emph{quotient of $Q$ by $G$}, denoted by $Q/G$.
\end{lem}

\begin{proof}
Let $x, y \in Q$, and $g,h \in G$. We need to prove that $(gx) \qd (hy)$ belongs to the $G$-orbit of $x \qd y$. Since $g$ acts via an automorphism, we have $(gx) \qd (hy) = g (x \qd (g^{-1}hy))$, so it is enough to prove that $x \qd Gy \subseteq G (x \qd y)$. By taking $z = x \qd y$, this is equivalent to showing that $x \qd G (x \qdn{-1} z) \subseteq G z$. But $\alpha(G)$ is normalized by $x \qd(-)$, whence the result, noticing that the same is true if we replace $\qd$ with $\qdn{-1}$, and $x \qd (-)$ by its inverse $x \qdn{-1}(-)$.
\end{proof}

\paragraph*{About the topological setting}

The discussion from~\cref{sec_congruences} applies directly to quandles. In particular, if  $\mathcal R$ is a congruence on a topological quandle $Q$, the quotient $Q/\mathcal R$ is not automatically a topological quandle, unless it is discrete. However, it is a topological quandle when $Q$ is compact (Lemma~\ref{quotient_of_top_algebras}), which is the only context in which we will deal with non-discrete quotients (see~\cref{sec_proof_criterion}).

\subsection{Topological enveloping groups}\label{subsec_G(Q)}

We now describe an adjunction between topological quandles and topological groups that restricts to the usual adjunction between quandles and groups when restricted to discrete ones. We begin with:

\begin{defi}
Let $Q$ be a topological quandle. Its \emph{enveloping group} $G(Q)$ is the group defined by the following presentation:
\[G(Q) = \langle Q \mid \forall x, y \in Q,\ x \qd y = xyx^{-1} \rangle.\]
We put on $G(Q)$ the finest group topology for which the canonical map $\eta_Q : Q \rightarrow G(Q)$ is continuous.
\end{defi}

The family of group topologies on $G(Q)$ rendering $\eta_Q$ continuous is clearly non-empty (it contains the undiscrete topology), so we can take its supremum. Recall that the supremum $\tau$ of a family $(\tau_i)_{i \in I}$ of topologies is the one for which a basis of neighbourhoods of a point $x$ is given by finite intersections $V_{i_1} \cap \cdots \cap V_{i_n}$, where $V_i$ is a neighbourhood of $x$ for $\tau_i$. It is easy to check that the supremum of a family of group topologies is again a group topology ($(x,y) \mapsto xy^{-1}$ is continuous for $\tau$ if it is for all the $\tau_i$), so our definition makes sense.

\begin{ex}\label{G(discrete)}
If $Q$ is discrete, $\eta_Q$ is continuous for the discrete topology, so $G(Q)$ is discrete.
\end{ex}

\begin{ex}\label{G(triv)}
If $Q$ is a trivial topological quandle, then the relations defining $G(Q)$ are saying exactly that the generators commute with each other. Thus, in this case, $G(Q) = \Z[Q]$ is the free topological abelian group on the topological space $Q$.
\end{ex}

The canonical map from $Q$ to $G(Q)$ is in fact (by definition of $G(Q)$) a continuous quandle morphism $\eta_Q: Q \rightarrow c(G(Q))$ (whose target is defined in Example~\ref{ex_c(G)}), which is the unit of the following adjunction:
\begin{prop}\label{adj_c_G}
The constructions $c(-)$ and $G(-)$ define functors between the category of topological quandle and the category of topological groups, which are  adjoint to each other:
\[ G : \begin{tikzcd}\mathcal{TQ}dl \ar[r, shift left] & \mathcal{TG}rp \ar[l, shift left] \end{tikzcd} : c.\]
\end{prop}

\begin{proof}
Let $Q$ be a topological quandle and $G$ be a topological group. We need to show that for every continuous quandle morphism $f: Q \rightarrow c(G)$, there is a unique continuous group morphism $\tilde f: G(Q) \rightarrow G$ such that $\tilde f \circ \eta_Q = f$. We already know (from the classical case, which is just an application of the universal property of group presentations) that there is a unique group morphism $\tilde f$ satisfying this condition. We need to show that it is continuous. For $X \in \{ Q, G(Q), G\}$, let us denote by $\tau_X$ the topology of $X$. The inverse image $\tilde f^{-1}(\tau_G)$ is a topology on $G(Q)$, compatible with the group structure, since $\tilde f$ is a group morphism. Moreover, the continuity of $f$ means that $f^{-1}(\tau_G) \subseteq \tau_Q$. Since $\tilde f \eta_Q = f$, we have $\eta_Q^{-1}(\tilde f^{-1}(\tau_G))\subseteq \tau_Q$, hence $\tilde f^{-1}(\tau_G)$ renders $\eta_Q$ continuous. By definition of $\tau_{G(Q)}$, this implies $\tilde f^{-1}(\tau_G) \subseteq \tau_{G(Q)}$, which means that $\tilde f$ is continuous.
\end{proof}

Notice that $\eta_Q$ does not need to be injective in general. In fact, it is often not injective (see for instance Example~\ref{Q12} below). And even if it is, it does not need to be a homeomorphism with its image. 

\begin{ex}\label{Q12}
Let $Q_{1,2}$ be the quandle consisting of $3$ elements $a,b,c$, with $a \qd(-)$ exchanging $b$ and $c$, and $b \qd(-) = c \qd(-) = id$ (this is the one from~\cite[\S~6]{Joyce}). Then $\eta$ is not injective. Indeed, since $b \qd a = a$, $\eta(a)$ commutes with $\eta(b)$. However, $a \qd b = c$, so $\eta(a)\eta(b)\eta(a)^{-1} = \eta(c)$, whence $\eta(a)$ and $\eta(b) = \eta(c)$ in $G(Q_{1,2})$.
\end{ex}

\begin{ex}
The topology of $G(Q)$ is a group topology, so it always comes from a uniform structure: it is \emph{uniformizable}. As a consequence, the topology of its subspace $\eta(Q)$ is also uniformizable. Let $X$ be a topological space which is not uniformizable (which is in fact equivalent to asking that $X$ be not completely regular). If we see it as a trivial quandle, then $G(X)$ identifies (algebraically) with the free abelian group $\Z[X]$ on $X$, so $\eta : X \rightarrow G(X)$ is injective. But $\eta$ cannot be a homeomorphism onto its image $\eta(X)$, since $X$ is not uniformizable, whereas $\eta(X)$ is.
\end{ex}

\begin{rmq}
The injectivity of $\eta_Q$ is easily shown to be equivalent to the existence of an injective morphism of $Q$ into $c(G)$ for some group~$G$. This carries over to topological quandles: $\eta_Q$ is a homeomorphism with its image if and only if there is a topological group $G$ and a morphism from $Q$ to $c(G)$ which is a homeomorphism with its image.
\end{rmq}

\subsection{The canonical action on the quandle}

Let us recall how the set-theoretic version of the adjoint action works, before discussing it continuity.

\begin{defi}\label{def_action}
An \emph{action} of a quandle $Q$ on another quandle $L$ is a morphism from $Q$ to $c(\Aut(L))$. Using the above adjunction, it can also be seen as an action of the group $G(Q)$ on $L$ by automorphisms (that is, a morphism from $G(Q)$ to $\Aut(L)$).
\end{defi}

\begin{ex}[Adjoint action]\label{adjoint_action}
Any quandle $Q$ acts on itself by $x \mapsto x \qd (-)$. Indeed, this defines a morphism from $Q$ to the conjugation quandle of $\Aut(Q)$, since:
\[(x \qd y) \qd z = (x \qd y) \qd (x \qd( x \overset{-1}{\qd}  z)) = x \qd (y \qd ( x \overset{-1}{\qd}  z)).\]
This action (or the corresponding action of $G(Q)$ on $Q$) is called the \emph{adjoint action}. In the sequel, the adjoint action of $g \in G(Q)$ on $Q$ is simply denoted by $q \mapsto g \cdot q$.
\end{ex}

\begin{lem}\label{equivariance}
Let $f: Q \rightarrow R$ be a quandle morphism. Then $f$ is equivariant with respect to adjoint actions. Precisely, if $f_*: G(Q) \rightarrow G(R)$ is the morphism induced by~$f$, then:
\[\forall g \in G,\ \forall x \in Q,\  f(g \cdot x) = f_*(g) \cdot f(x). \]
\end{lem}

\begin{proof}
Recall that we denote by $\eta : Q \rightarrow G(Q)$ the canonical map. It is enough to check equivariance with respect to the generators $g = \eta(y)$ ($y \in Q$) of $G(Q)$. By definition of the adjoint action,
$f(\eta(y) \cdot x) = f(y \qd x) = f(y) \qd f(x) = f_*(\eta(y)) \cdot f(x),$ whence our claim.
\end{proof}

Let us reformulate this lemma in terms of the canonical action map $a_Q : G(Q) \times Q \rightarrow Q$ (where $Q$ is a quandle). If $f : Q \rightarrow R$ is a quandle morphism, equivariance says that the following square is commutative:
\begin{equation}\label{square_of_actions}
\begin{tikzcd}
G(Q) \times Q \ar[r,"a_Q"] \ar[d, swap, shift left = 9pt, "G(f) \times f"] 
&Q \ar[d, "f"]\\
G(R) \times R \ar[r,"a_R"] &R.
\end{tikzcd}
\end{equation}

For topological quandles, the adjoint action is continuous, in the following sense:

\begin{fact}\label{continuous_action_of_G(Q)}
If $Q$ is a topological quandle, the canonical map $G(Q) \rightarrow \mathcal \Aut_{\mathcal C}(Q)$ is continuous. Hence so is its adjoint $G(Q) \times Q \rightarrow Q$ if $Q$ is locally compact.
\end{fact}

\begin{proof}
By the universal property of $G(Q)$ (Prop.~\ref{adj_c_G}), the continuity of the canonical map $G(Q) \rightarrow \mathcal \Homeo(Q) \subseteq \mathcal C(Q,Q)$ is equivalent to the continuity of $Q \rightarrow \mathcal C(Q,Q)$. The latter is deduced from the continuity of its adjoint $\qd : Q^2 \rightarrow Q$, using Proposition~\ref{prop_adj_direct}. The rest of the statement follows from Proposition~\ref{prop_adj_converse}.
\end{proof}

\subsection{Trivialisation of topological quandles}\label{par_trivialisation}

The category $\mathcal T$ of topological spaces is a full subcategory of the category $\mathcal{TQ}dl$ of topological quandle, via the functor endowing every topological space with the trivial law (Example~\ref{triv_topo_qdl}). It is in fact a \emph{reflexive} subcategory: if $Q$ is a topological quandle, then its trivialisation $\pi_0(Q) = Q/G(Q)$, endowed with the quotient topology, is the universal trivial quotient of $Q$. Indeed, let us first note that the canonical projection $\pi : Q \twoheadrightarrow \pi_0(Q)$ is a morphism of topological quandles, for the trivial law on the target: if $x,y \in Q$, then $\pi(x \qd y) = \pi(\eta(x) \cdot y) = \pi(y)$, which is equal to $\pi(x) \qd \pi(y)$ if the target is endowed with the trivial law. Then:
\begin{lem}\label{UP_pi_0(top)}
For every morphism $f: Q \rightarrow X$ from a topological quandle to a trivial quandle $X$, there exists a unique continuous map $\overline f : \pi_0(Q) \rightarrow X$ such that $\overline f \circ \pi = f$.
\end{lem}

\begin{proof}
Since $G(X)$ acts trivially on $X$, Lemma~\ref{equivariance} says that  for every $w \in G(Q)$ and every $q \in Q$, $f(w \cdot q) = G(f)(w) \cdot f(q) = f(q)$. Thus $f$ factors through the quotient $Q/G(Q)$, and the resulting  $\overline f : \pi_0(Q) \rightarrow X$ is continuous for the quotient topology.
\end{proof}

The abelianisation $G^{\mathrm{ab}}$ of a topological group $G$ is endowed with the quotient topology; this makes $(-)^{\mathrm{ab}}$ into a left adjoint to the inclusion  $\mathcal{TA}b \hookrightarrow \mathcal{TG}rp$ of topological abelian groups into topological groups. Similarly, we get a linearisation functor $\Z[-]$ adjoint to the forgetful functor $\mathcal{TA}b \rightarrow \mathcal T$ from topological abelian groups to topological spaces, by endowing the linearisation $\Z[X]$ of a topological space $X$ with the finest group topology rendering $X \rightarrow \Z[X]$ continuous. We thus get a commutative square of adjunctions:
\[\begin{tikzcd}[column sep=40, row sep=40]
\mathcal{TG}rp   \ar[r, shift right, hook'] \ar[d, shift right, swap, "{(-)^{\mathrm{ab}}}"]
& \mathcal{TQ}dl \ar[l, shift right, swap, "G(-)"] \ar[d, shift right, swap, "\pi_0(-)"] \\
\mathcal{TA}b   \ar[r, shift right, hook'] \ar[u, shift right, hook']  
& \mathcal T. \ar[l, shift right, swap, "{\Z[-]}"] \ar[u, shift right, hook']  
\end{tikzcd}\]
The square of right adjoints obviously commutes, so the square of left adjoints does too. Hence we have proved:
\begin{prop}\label{G(Q)^ab}
There is a canonical isomorphism $G(Q)^{\mathrm{ab}} \cong \Z[\pi_0 Q]$  of topological abelian groups, which is natural in the  topological quandle $Q$.
\end{prop}

\section{Profinite quandles}
\label{sec_ProQdl}

This section is devoted to the study of profinite quandle. In particular, we introduce the tools necessary to the proof of our main theorem in~\cref{sec_ProPph}. As was done for topological quandles in~\cref{sec_quandles}, we notably investigate the enveloping group, the adjoint action, and the trivialization of a quandle, this time in the profinite context.

\subsection{Definition and first properties}

We can apply the terminology and results of~\cref{par_res_finite_alg} and~\cref{par_prof_alg} to quandles. In particular, Definition~\ref{def_prof_alg} becomes:
\begin{defi}
A \emph{profinite quandle} is a topological quandle which is continuously isomorphic to a projective limit of finite  discrete quandles. 
\end{defi} 

\begin{ex}\label{triv_prof}
Trivial topological quandles are projective limits of finite discrete trivial quandles, hence just profinite spaces, that is, compact and totally continuous spaces (see~\cref{par_profinite_groups_and_spaces}), endowed with the trivial quandle law.
\end{ex}

\begin{ex}\label{ex_c(prof)}
If $G$ is a profinite group, then the associated conjugation quandle $c(G)$ is a profinite quandle. Indeed, we always have $c(\varprojlim G_i) = \varprojlim c(G_i)$.
\end{ex}

We can also apply Proposition~\ref{profinite_iff_cpct_and_res_finite} to quandles:
\begin{prop}\label{qdl_prof_iff_cpct_and_res_finite}
A topological quandle is profinite if and only if it is compact and residually finite.
\end{prop}

By Fact~\ref{res_finite_are_tot_disconnected} (and Corollary~\ref{char_profinite_spaces}), this entails:
\begin{cor}\label{qdl_prof_cpct_and_discontinuous}
A profinite topological quandle is compact and totally disconnected. In other words, it is profinite as a topological space.
\end{cor}

\begin{rmq}
The object of~\cref{sec_criterion} will be to investigate the converse of the latter statement. We will see that, contrary to what happens for groups (Corollary~\ref{profiniteness_criterion_groups}), it does not hold, which means that profinite quandles do not have a characterisation purely in terms of the underlying topological space. However, we will still be able to show a criterion for profiniteness which looks like a partial converse to Corollary~\ref{qdl_prof_cpct_and_discontinuous} (see~Theorem~\ref{profiniteness_criterion}).  
\end{rmq}

The theory of profinite completions (\cref{sec_prof_comp}) can be applied readily to quandles. In this context, Corollary~\ref{univ_prop_completion} becomes:
\begin{cor}\label{univ_prop_completion_qdl}
For every profinite quandle $X$ and every topological quandle $Q$, restriction along the canonical morphism $\iota : Q \rightarrow \widehat Q$ induces a natural bijection:
\[\Hom(Q, X) \cong \Hom( \widehat Q, X).\]
In other words, the functor $Q \mapsto \widehat Q$ from topological quandles to profinite quandles is left adjoint to the forgetful functor.
\end{cor}

\subsection{Profinite enveloping groups}\label{par_prof_G(Q)}

If $Q$ is a nonempty topological quandle, the constant map $q \mapsto 1$ from $Q$ to the discrete group $\Z$ factors through a continuous surjection $G(Q) \twoheadrightarrow \Z$. In particular, the topological group $G(Q)$ is never compact, hence never a profinite group. So in order to stay in the profinite world, we pose:
\begin{defi}
Let $Q$ be a topological quandle. Its \emph{profinite enveloping group} $\widehat G(Q)$ is the profinite completion of the topological group $G(Q)$.
\end{defi}

Recall that taking the associated conjugation quandle defines a functor $c(-)$ from profinite groups to profinite quandles (see Example~\ref{ex_c(prof)}). The following result is an adaptation to the profinite world of Proposition~\ref{adj_c_G}:
\begin{prop}\label{G(-)_adj_to_c(-)}
The functor $\widehat G$ is left adjoint to the functor $c(-)$ from profinite groups to profinite quandles.
\end{prop}

\begin{proof}
Using the usual universal property of the profinite completion of groups (Proposition~\ref{univ_prop_completion}) and Proposition~\ref{adj_c_G}, we can write, the following sequence of canonical isomorphisms, natural in the profinite quandle $Q$ and the profinite group $G$:
\[\Hom(\widehat{G(Q)}, G) \cong \Hom(G(Q), G) \cong \Hom(Q, c(G)).\]
Moreover, $\Hom(Q, c(G))$ is the set of morphisms from $Q$ to $c(G)$ in the category of profinite quandles, considered as a full subcategory of topological quandles. Whence the conclusion.
\end{proof}

This implies that $\widehat G(-)$, applied to the profinite completion of any topological quandle~$Q$, gives the profinite completion of the topological group $G(Q)$:
\begin{cor}\label{profinite_completion_of_env_groups}
For every topological quandle $Q$, we have $\widehat G (\widehat Q) = \widehat{G(Q)}$.
\end{cor}

\begin{proof}
Both sides have the same universal properties. Namely, as functors on topological quandles, they are both left adjoint to $G \mapsto c(G)$, seen as a functor from profinite groups to topological quandles. Precisely, for any profinite group $G$, we have, on the one hand:
\[\Hom(Q, c(G)) = \Hom(\widehat Q, c(G)) = \Hom(G(\widehat Q), G) = \Hom(\widehat G(\widehat Q), G),\]
where the middle isomorphism comes from Proposition~\ref{G(-)_adj_to_c(-)}, and the other ones from Proposition~\ref{univ_prop_completion} (using Example~\ref{ex_c(prof)} for the first one).
On the other hand:
\[\Hom(Q, c(G)) = \Hom(G(Q), G) = \Hom(\widehat{G(Q)}, G).\]
These isomorphisms are easily seen to be natural in $G$ and $Q$, whence the canonical isomorphism of the statement.
\end{proof}

Recall that $\widehat G(Q)$ it the projective limit of finite discrete quotients of $G(Q)$. Such a quotient can be obtained in two steps: first choose a cofinite congruence $\mathcal R$ on $Q$, to get the discrete quotient $G(Q) \twoheadrightarrow G(Q/\mathcal R)$. Notice that $G(Q/\mathcal R)$ is discrete (since $Q/\mathcal R$ is), but not finite: its abelianisation is a non-trivial free abelian group. However, a finite discrete quotient $G_{\mathcal R, N}$ of $G(Q)$ can be obtained by quotienting further by any finite-index normal subgroup $N$ of $G(Q/\mathcal R)$. In fact:

\begin{lem}\label{explicit_Prof_G(-)}
Every finite discrete quotient of $G(Q)$ is of the form $G_{\mathcal R, N} = G(Q/\mathcal R)/N$, where $\mathcal R$ is a cofinite congruence on the topological quandle $Q$, and $N$ is a finite-index normal subgroup of $G(Q/\mathcal R)$.
\end{lem}

\begin{proof}
Let $p : G(Q) \twoheadrightarrow G$ be a continuous surjection onto a finite discrete group $G$. Let $\mathcal R$ be the kernel of the quandle morphism $p \eta$, where $\eta : Q \rightarrow G(Q)$ is the canonical map. Then $p \eta$ induces a continuous injection $p_0: Q/\mathcal R \hookrightarrow G$, so $\mathcal R$ is cofinite ($Q/\mathcal R$ is finite discrete). Furthermore, $p_0$ extends uniquely to a group morphism $\tilde p_0 : G(Q/\mathcal R) \twoheadrightarrow G$. By taking $N := \ker(\tilde p_0)$, we finally get an isomorphism $G(Q/\mathcal R)/N \cong G$.
\end{proof}

It is not hard to see that $G_{\mathcal R, N}$ is a larger quotient than  $G_{\mathcal R', N'}$ if and only if $\mathcal R \subseteq \mathcal R'$ and the image of $N$ by the canonical surjection $G(Q/\mathcal R) \twoheadrightarrow G(Q/\mathcal R')$ is contained in $N'$. This gives a quite useful description of the projective system of which $\widehat G(Q)$ is the projective limit. We denote by $I_Q$ the corresponding poset.

\subsection{The canonical action on the quandle}

The goal of the present section is the following:
\begin{prop}\label{profinite_adj_action}
Let $Q$ be a profinite quandle. The canonical action of $G(Q)$ on $Q$ extends uniquely to a continuous action of $\widehat G(Q)$ on $Q$.
\end{prop}

\begin{proof}
We use the description of $\widehat G(Q)$ as the projective limit of the $G_{\mathcal R,N} = G(Q/\mathcal R)/N$, for $(\mathcal R,N) \in I_Q$ (see~\cref{par_prof_G(Q)}). For a cofinite congruence $\mathcal R$ on $Q$, let us consider the canonical action of $G(Q/\mathcal R)$ on $Q/\mathcal R$. Its kernel $K_\mathcal R$ is the kernel of the corresponding morphism from $G(Q/\mathcal R)$ to the finite group $\Aut(Q/\mathcal R)$, so it is of finite index in $G(Q/\mathcal R)$. Let us consider the (full) sub-poset $I$ of $I_Q$ consisting of the $(\mathcal R, N)$ for which $N \subseteq K_\mathcal R$. It is clearly cofinal in $I_Q$, so $\widehat G(Q)$ is also the projective limit of the $G_{\mathcal R,N}$ for $(\mathcal R,N) \in I$. For $(\mathcal R,N) \in I$, the action of $G(Q/\mathcal R)$ on $Q/\mathcal R$ factors through quotienting by $N$, giving well-defined maps (between discrete spaces) :
\[a_{\mathcal R,N} : G_{\mathcal R,N} \times Q/\mathcal R \longrightarrow Q/\mathcal R,\ \ \text{for } (\mathcal R,N) \in I.\]
From the diagram~\eqref{square_of_actions} applied to canonical projections between quotients of $Q$, we get commutative squares saying that these maps define a morphism of projective systems indexed by $I$, which is also a projective system in the arrow category of (discrete) topological spaces. Taking the projective limit yields a continuous map:
\[\widehat{a_Q} : \widehat G(Q) \times Q = \varprojlim \left(G_{\mathcal R,N} \times Q/\mathcal R\right) \longrightarrow \varprojlim (Q/\mathcal R) = Q.\]
The fact that $\widehat a$ is a group action comes from the fact that the $a_{\mathcal R,N}$ are. 

The diagram~\eqref{square_of_actions} also gives morphisms in the arrow category (that is, commutative squares) from $a_Q$ (seen as a map between discrete spaces) to the $a_{\mathcal R, N}$, which are compatible with the transition maps. These induce a morphism $a_Q \rightarrow \widehat{a_Q}$ to the projective limit, that is, a commutative square:
\[
\begin{tikzcd}
G(Q) \times Q \ar[r,"a_Q"] \ar[d, swap, shift left = 9pt, "\iota \times id"] 
&Q \ar[d, "id"]\\
\widehat G(Q) \times Q \ar[r,"\widehat{a_Q}"] &Q,
\end{tikzcd}
\]
where $\iota : G(Q) \rightarrow \widehat G(Q)$ is the canonical morphism. This says that the action $\widehat{a_Q}$ extends $a_Q$. Unicity follows immediately from the density of the image of $\iota$, which entails the density of the image of $\iota \times id$.
\end{proof}

\begin{rmq}\label{adj(profinite_adj_action)}
The continuity of the canonical action $\widehat G(Q) \times Q \rightarrow Q$ is equivalent to the continuity of its adjoint $\widehat G(Q) \rightarrow \mathcal C(Q,Q)$, since $Q$ is profinite, hence compact (see Prop.~\ref{prop_adj_direct} and~\ref{prop_adj_converse})
\end{rmq}

\begin{rmq}
We will see later on another proof of this result using inner automorphisms instead of the enveloping group (see Corollary~\ref{action_of_prof_G(Q)}).
\end{rmq}

\begin{lem}\label{prof_equivariance}
Let $f: Q \rightarrow R$ be a continuous quandle morphism between profinite quandles. Then $f$ is equivariant with respect to completed adjoint actions. Precisely, if $\widehat G(f) : \widehat G(Q) \rightarrow \widehat G(R)$ is the morphism induced by~$f$, then:
\[\forall w \in \widehat G(Q),\ \forall x \in Q,\  f(w \cdot x) = \widehat G(f)(w) \cdot f(x). \]
\end{lem}

\begin{proof}
Both sides of the equality are continuous with respect to $w \in \widehat G(Q)$, so the conclusion follows from Lemma~\ref{equivariance}, by density of the image of $\iota : G(Q) \rightarrow \widehat G(Q)$, using the fact that $\widehat G(f) \circ \iota = \iota \circ G(f)$, and that the action of $\widehat G(Q)$ extends that of $G(Q)$.
\end{proof}

\subsection{Trivialisation of profinite quandles}\label{par_trivialisation_profinite}

If $Q$ is a profinite quandle, then $\pi_0(Q)$ (defined in \cref{par_trivialisation}) is not, in general, a profinite space.

\begin{ex}\label{ex_comp_of_Z}
Let $Q$ be the profinite quandle $\{q_0\} \cup \widehat \Z$, where $x \qd y = y$ except if $x = q_0$ and $y \in \widehat \Z$, in which case $x \qd y = y+1$. Then the only generator of $G(Q)$ acting non-trivially is $\eta(q_0)$, acting via $(-) + 1$ on $\widehat \Z$ and fixing $q_0$. Hence $\pi_0(Q) = Q/G(Q) \cong \{q_0\} \cup \widehat \Z/\Z$ is a very ill-behaved space, whose topology is clearly not Hausdorff (every point of $\widehat \Z/\Z$ is dense).
\end{ex}

Thus, if we want to define a trivialization that stays in the profinite realm, we have to modify this construction. One way to do that is to take the quotient of $Q$ not by $G(Q)$, but by the profinite group $\widehat G(Q)$. 
\begin{defi}
Let $Q$ be a profinite quandle. Its \emph{profinite trivialization} is $\widehat \pi_0(Q) := Q/\widehat G(Q)$, endowed with the quotient topology.
\end{defi}

\begin{lem}
Let $Q$ be a profinite quandle. Then $\widehat \pi_0(Q)$ is a profinite space.
\end{lem}

\begin{proof}
It is a quotient of the profinite space $Q$ by a continuous action of the compact group $\widehat G(Q)$, so we can apply Proposition~\ref{Quotient_of_Stone_space}. 
\end{proof}

Recall that a quandle is called~\emph{topologically finitely generated} if its admits a dense finitely generated subquandles (Definition~\ref{def_topo_fg}).

\begin{cor}\label{finite_triv_of_fg_prof}
Let $Q$ be a profinite quandle. If $Q$ is topologically finitely generated, then $\widehat \pi_0(Q)$ is finite (hence discrete).
\end{cor}

\begin{proof}
Let $q_1, ..., q_n$ be topological generators of $Q$. This means that $Q = \overline{\langle q_1, ..., q_n\rangle}$, the closure of the subquandle generated by the $q_i$. Let $\pi : Q \twoheadrightarrow \widehat \pi_0(Q) = Q / \widehat G(Q)$ be the canonical projection. It is a continuous quandle morphism (to a trivial topological quandle), so:
\[\pi\left(\overline{\langle q_1, ..., q_n\rangle}\right) \subseteq \overline{\pi\left(\langle q_1, ..., q_n\rangle\right)} =
\overline{\langle \pi(q_1), ..., \pi(q_n)\rangle} = 
\overline{\{ \pi(q_1), ..., \pi(q_n)\}}.\]
Since $\widehat \pi_0(Q)$ is Hausdorff, its finite subsets are closed. Moreover, $\pi$ is surjective, so $\widehat \pi_0(Q) = \pi(Q) =  \{ \pi(q_1), ..., \pi(q_n)\}$ is finite, hence also discrete, since it is Hausdorff.
\end{proof}

\begin{rmq}\label{finite_triv_of_fg_completion}
This applies in particular to profinite completions of finitely generated (discrete) quandles.
\end{rmq}

\begin{ex}\label{ex_comp_of_Z_2}
Let $Q$ be as is Example~\ref{ex_comp_of_Z}. Then $\Homeo(Q)$ has $\widehat \Z$ as a subgroup, acting via translations of $\widehat \Z$ (fixing $q_0$), which is topologically generated by $q_0 \qd (-)$. The map $\widehat G(Q) \rightarrow \Homeo(Q)$ is continuous (Prop~\ref{profinite_adj_action} and Remark~\ref{adj(profinite_adj_action)}), so its image is compact, hence closed. Since it contains $q_0 \qd (-)$, it must contains $\widehat \Z$. Then one easily sees that $\widehat \pi_0(Q) \cong \{q_0\} \cup \widehat \Z/\widehat \Z$ is simply a discrete space with two points.
\end{ex}

Recall that topological spaces identify with the full subcategory of topological quandles given by trivial ones. The profinite trivialization satisfies the following universal property:
\begin{prop}\label{pi0_adj_to_incl}
The functor $\widehat \pi_0$ is left adjoint to the inclusion of profinite spaces into profinite quandles.
\end{prop}

\begin{proof}
For any profinite quandle $Q$, let us denote by $\pi : Q \twoheadrightarrow \widehat \pi_0(Q)$ the canonical projection. Let $R$ be a profinite space, seen as a trivial profinite quandle. Let $f : Q \rightarrow R$ be a continuous quandle morphism. We want to show that there exists a unique continuous map $\overline f : \widehat \pi_0(Q) \rightarrow R$ such that $\overline f \circ \pi = f$. Lemma~\ref{prof_equivariance} ensures that $f$ is equivariant with respect to $\widehat G(f) : \widehat G(Q) \rightarrow \widehat G(R)$. But the image of $G(R)$ in $\Homeo(R)$ is  $\{id\}$, which is closed in $\Homeo(R)$ (which is totally disconnected by Proposition~\ref{Homeo(X)_tot_disc}), so $\widehat G(R)$ acts trivially on $R$. Hence for every $w \in \widehat G(R)$ and every $q \in Q$, we have $f(w \cdot q) = \widehat G(f)(w) \cdot f(q) = f(q)$. The latter means that $f$ factors (obviously uniquely) through the quotient $\widehat \pi_0(Q) = Q/\widehat G(Q)$, as claimed.
\end{proof}

Combining this with the universal properties of the topological trivialisation (see Lemma~\ref{UP_pi_0(top)}) and of the profinite completion (Corollary~\ref{univ_prop_completion_qdl}), we can then adapt the proof of
Corollary~\ref{profinite_completion_of_env_groups} in a straightforward way to show the following:
\begin{cor}\label{prof_triv_of_prof_completion}
For any topological quandle $Q$, we have $\widehat\pi_0(\widehat Q) \cong \widehat{\pi_0(Q)}$.
\end{cor}

In particular, when $Q$ is profinite (that is, when $Q \cong \widehat Q$), its profinite trivialization $\widehat\pi_0(Q) = Q/\widehat G(Q)$ is really the profinite completion of it topological trivialization $\pi_0(Q)$.

\section{Profinite peripheral systems}
\label{sec_ProPph}

In \cite[\S~4.2]{Darne7}, Joyce's construction of a quandle from a group $G$ together with some group-theoretic data \cite[\S 7]{Joyce} was reinterpreted categorically, as a functorial correspondence between quandles and (abstract) peripheral systems. The present section is devoted to an adaptation of this machinery to the profinite realm, in order to turn it into a correspondence between the profinite completion of a finitely generated quandle and the profinite completion of the associated peripheral system. 

We first recall the main pieces of the machinery in its usual (discrete) version in~\cref{discrete_pph}. Then we adapt it to the profinite world in~\cref{par_prof_Pph}, where we show that the machinery reconstructs any profinite quandle from the canonically associated profinite peripheral system, in a functorial way (Prop.~\ref{QP=1_prof}). Finally, in~\cref{prof_compl_of_Pph}, we discuss the identification between the profinite peripheral system associated with $\widehat Q$ and the profinite completion of the peripheral system associated with $Q$, for a finitely generated (discrete) quandle $Q$. By applying this to link quandles, we are then able to prove our main theorem (Th.~\ref{main_thm}).

\subsection{Discrete quandles and peripheral systems}\label{discrete_pph}

Let us recall briefly how the construction from \cite[\S~4.2]{Darne7} works. The only thing new here is the notion of a quotient of a peripheral system by a normal subgroup (Ex.~\ref{Quotients_of_Pph} and~\ref{Q(Quotients_of_Pph)}).

\begin{defi}\label{def_Ob(Pph)}
A \emph{peripheral system} $\left(G, (H_i, z_i)_{i \in I}\right)$ is the data of a group $G$, a set $I$ and, for each $i \in I$, a subgroup $H_i$ of $G$ and an element $z_i$ in the center of~$H_i$.
\end{defi}

\begin{defi}\label{def_Mor(Pph)}
A \emph{morphism} $(\varphi, u, [g]) : \left(G, (H_i, z_i)_{i \in I}\right) \rightarrow \left(G', (H_j', z_j')_{j \in J}\right)$ of peripheral system is the data of a group morphism $\varphi : G \rightarrow G'$, a map $u: I \rightarrow J$, and the class $[g]$ of a map $g: I \rightarrow G'$ denoted by $i \mapsto g_i$, where $g_i$ is defined only up to right multiplication by $H_{u(i)}'$. This data has to satisfy, for all $i \in I$:
\begin{equation}\label{condition_morphisms_Pph}
\begin{cases}
g_i^{-1}\varphi(H_i)g_i \subseteq H_{u(i)}', \\ 
g_i^{-1}\varphi(z_i)g_i = z_{u(i)}'.
\end{cases}   
\end{equation}
\end{defi}

\begin{conv}
We often write only $g$ for $[g]$, remembering that what is meant by $g_i$ is really $g_iH_{u(i)}'$.
\end{conv}

Composition of morphisms is given by the formula:
\begin{equation}\label{composition_morphisms_Pph}
(\varphi', u', g') \circ (\varphi, u, g) = (\varphi' \circ \varphi, u' \circ u, i \mapsto \varphi'(g_i) g'_{u(i)}).
\end{equation}

\begin{defi}\label{def_Pph}
All the above defines a category $\mathcal Pph$ of peripheral systems.
\end{defi}

\begin{ex}\label{Quotients_of_Pph}
Let $P := \left(G, (H_i, z_i)_{i \in I}\right)$ be a peripheral system, and $N$ be a normal subgroup of $G$. Then $(G/N, (H_iN/N, z_iN)_{i \in I}$ is also a peripheral system, called the \emph{quotient of $P$ by $N$}, and denoted by $P/N$. It is endowed with a canonical morphism $P \twoheadrightarrow P/N$. 
\end{ex}

Given a peripheral system $\left(G, (H_i,z_i)_{i \in I} \right)$, we construct a quandle $\mathcal Q \left(G, (H_i,z_i)_{i \in I} \right)$ by endowing the set $\bigsqcup_{i \in I} G/H_i$ with the law:
\[\forall i,j \in I,\ \forall x,y \in G,\ \ xH_i\ \qd\  yH_j := x z_i x^{-1} yH_j.\]
One can check easily that this is a well-defined quandle law. Moreover, this construction is functorial: if $(\varphi, u, g)$ is a morphism between $\left(G, (H_i, z_i)_{i \in I}\right)$ and $\left(G', (H_j', z_j')_{j \in J}\right)$, then the formula
\[xH_i \mapsto \varphi(x) g_i H'_{u(i)}\] 
is easily seen to give a well-defined quandle morphism, and this construction preserves the composition of morphisms. 

\begin{ex}\label{Q(Quotients_of_Pph)}
Let $P := \left(G, (H_i, z_i)_{i \in I}\right)$ be a peripheral system. Notice that the action by left multiplication of $G$ on $\bigsqcup G/H_i$ is an action of $G$ on $\mathcal Q(P)$ by automorphisms:
\[gxH_i\ \qd\  gyH_j := gx z_i (gx)^{-1} gyH_j = gx z_i x^{-1} yH_j.\]  
Moreover, this action is normalized by inner automorphisms $xH_i\ \qd\ (-)$:
\[xH_i\ \qd\  gyH_j := x z_i x^{-1} gyH_j = x z_i x^{-1} g x z_i^{-1}x^{-1}(xH_i\ \qd\  gyH_j).\]  
As a consequence, if $N$ is a normal subgroup of $G$, we can apply Proposition~\ref{quotient_by_group} to define the quandle $Q/N$. One can check easily that it relates to the peripheral system $P/N$ from Example~\ref{Quotients_of_Pph} through a canonical isomorphism $\mathcal Q(P/N) \cong \mathcal Q(P)/N$. 
\end{ex}

Conversely, if $Q$ is a quandle, endowed with a choice of representatives $(q_i)_{i \in \pi_0(Q)}$ of its orbits, then we can construct a peripheral system, by looking at the canonical action of the enveloping group $G(Q)$ on $Q$, and taking:
\[\mathcal P (Q,(q_i)_{i \in \pi_0(Q)}) := (G(Q), (\Stab(q_i), \eta(q_i))_{i \in \pi_0(Q)}).\]
Now, if we have a quandle morphism $f : Q \rightarrow R$, and choices of orbit representatives $(q_i)_{i \in \pi_0(Q)}$ and $(r_j)_{j \in \pi_0(R)}$ (which need not be preserved in any way by $f$), then we can associate to it the morphism:
\[\mathcal P (f) := (G(f), \pi_0(f), [g]),\]
where $g_i$ is an element of $G(R)$ such that $f(q_i) = g_i \cdot r_j$, for some $j$. Notice that such an element exists, by construction of the $r_j$, and that by definition of $\pi_0(f)$, we must have $j = \pi_0(f)(i)$. Moreover, the class $[g]$ is independent of the choices made: if $g_i'$ satisfies the same equation as $g_i$, then $g_i^{-1}g_i'$ fixes $r_j$, so $g_i' \in g_i \Stab(r_j)$, and $[g'] = [g]$.
If we consider the category $\mathcal Qdl_s$ of quandles $Q$ endowed with a section of $Q \twoheadrightarrow \pi_0(Q)$, where morphisms are just quandle morphisms (not necessarily preserving the chosen representatives of orbits in any way), this construction defines a functor from $\mathcal Qdl_s$ to $\mathcal Pph$. 

Notice that for a peripheral system $\left(G, (H_i,z_i)_{i \in I} \right)$, the quandle $\mathcal Q \left(G, (H_i,z_i)_{i \in I} \right)$ comes equipped with a canonical choice of orbit representatives, given by the class of~$1$ in each $G/H_i$. This makes $\mathcal Q$ into a functor from $\mathcal Pph$ to $\mathcal Qdl_s$. In fact, $\mathcal P$ is left adjoint to $\mathcal Q$ (see \cite[\S~4.2]{Darne7}).

\medskip

Let  $\left(G, (H_i,z_i)_{i \in I} \right)$ be a peripheral system, and let $G$ act on a quandle $Q$ by quandle automorphisms. Suppose that there are elements $q_i \in Q$ such that $H_i \subseteq \Stab(q_i)$, and $z_i \cdot (-) = q_i \qd (-)$. Then the canonical map
\[f:\left\{\renewcommand{\arraystretch}{1.5}
\begin{array}{clc}
\bigsqcup\limits_{i \in I} G/H_i &\longrightarrow &Q \\
gH_i &\longmapsto & g \cdot q_i
\end{array}
\right.\]
turns out to be a quandle morphism. It is bijective if and only if the $q_i$ are representatives of $G$-orbits in $Q$, and $H_i = \Stab(q_i)$ (for every $i \in I$). These conditions are satisfied for the peripheral system $\mathcal P(Q)$, with respect to the canonical action of $G(Q)$ on $Q$. This gives back Joyce's result, in its version from~\cite[Prop.~4.30]{Darne7}:
\begin{prop}\label{QP=1}
There is a canonical isomorphism $\mathcal Q \circ \mathcal P \cong id_{\mathcal Qdl_s}$.
\end{prop}

This is the main tool for proving that the data of the fundamental quandle $Q_L$ of a link $Q$ up to isomorphism is equivalent to the data of its (weak) peripheral system up to equivalence. In fact, the relevant peripheral system is $\mathcal P(Q_L)$, and our isomorphisms of peripheral systems correspond to the classical notion of equivalence (see for instance~\cite[\S~7.6]{Darne7}).

\subsection{Adaptation to the profinite world}\label{par_prof_Pph}

We now adapt the previous correspondence between quandles and peripheral systems to the case of profinite quandles. We limit ourselves to profinite quandles having a (necessarily finite) discrete set of orbits, a hypothesis that is satisfied for profinite completions of fundamental quandles of links. We will explain briefly in an appendix (see~\cref{sec_top_Pph}) what a more general setting can look like.

\begin{defi}
A \emph{profinite peripheral system} is a peripheral system $\left(G, (H_i, [z_i])_{i \in I}\right)$, where $I$ is finite, $G$ is endowed with a topology making it into a profinite group, and the $H_i$ are closed in $G$. A \emph{morphism} of profinite peripheral systems is a morphism $(\varphi, u, [g])$ of peripheral systems, where $\varphi$ is continuous. With the composition of morphisms defined as in \cref{discrete_pph}, this defines a category $pr\mathcal Pph$ of profinite peripheral systems.
\end{defi}

The construction $\mathcal Q$ from above can be generalized in a straightforward manner: if each $G/H_i$ is endowed with the quotient topology and $\bigsqcup_{i \in I} G/H_i$ is a topological disjoint union, then the quandle $\mathcal Q\left(G, (H_i, z_i)_{i \in I}\right)$ becomes a topological quandle, and morphisms of profinite peripheral systems induce continuous quandle morphisms. 
\begin{lem}
If $ P =\left(G, (H_i, z_i)_{i \in I}\right)$ is a profinite peripheral system, then $\mathcal Q(P)$ is a profinite quandle. Precisely, $\mathcal Q(P) \cong \varprojlim \mathcal Q(P/N)$, where the limit is taken over open normal subgroups $N$ of $G$ (with notations from Example~\ref{Quotients_of_Pph}).
\end{lem}

\begin{proof}
Since $G$ is compact, the $G/H_i$ are compact (they are in fact profinite spaces by Proposition~\ref{Quotient_of_Stone_space}, and in particular Hausdorff, but we do not need this fact for the proof), so $\mathcal Q(P)$ is too, as a finite disjoint union of compact spaces. The canonical morphisms $\mathcal Q(P) \rightarrow \mathcal Q(P/N)$ are surjective, so the induced morphism from $\mathcal Q(P)$ to $\varprojlim \mathcal Q(P/N)$ has dense image. It is also continuous, because the $\mathcal Q(P/N)$, seen as finite discrete spaces, are topological quotients of $\mathcal Q(P)$. As a consequence, it is surjective, and it will be a homeomorphism if we show that it is injective. In other words, we want to show that the quotient maps $\mathcal Q(P) \rightarrow \mathcal Q(P/N)$ separate the points of $\mathcal Q(P)$.

Let $xH_i$ and $yH_j$ be distinct points of $\mathcal Q(P)$. If $i \neq j$, then they are separated by the canonical map $\mathcal Q(P) \twoheadrightarrow \mathcal Q(P/G) \cong I$ onto the finite discrete trivial quandle $I$, so we can assume that $i = j$. Recall that open normal subgroups form a basis of neighbourhoods of $1$ in the profinite group $G$ (Proposition~\ref{prop_Bourbaki}). Since $yH_i$ is closed and $x \notin yH_i$, there is such a subgroup $N$ of $G$ such that $xN \cap yH_i = \varnothing$, which means that $x \notin yH_iN$. But the latter means exactly that the image $xH_iN$ of $xH_i$ in  $\mathcal Q(P/N)$ is distinct from the image $yH_iN$ of $y H_i$, whence the conclusion.
\end{proof}

The construction of $\mathcal P$ is also easily adapted, by replacing the enveloping group with the profinite enveloping group $\widehat G(Q)$ of a profinite quandle $Q$, which acts canonically on $Q$ (see Proposition~\ref{profinite_adj_action}). Notice that since $Q$ is Hausdorff and the action is continuous, stabilizers of points of $Q$ are closed. This modified version of $\mathcal P$, denoted by $\widehat{\mathcal P}$, is defined on the category $pr\mathcal Qdl_s$, whose objects are profinite quandles having a discrete (hence finite) $\widehat \pi_0$, endowed with a choice of representatives of \emph{profinite orbits} (that is, orbits under $\widehat G(Q)$), and whose morphisms are all quandle morphisms. 

Finally, the arguments leading to Proposition~\ref{QP=1} are easily generalized to this setting, giving:
\begin{prop}\label{QP=1_prof}
There is a canonical isomorphism $\mathcal Q \circ \widehat{\mathcal P} \cong id_{pr\mathcal Qdl_s}$.
\end{prop}

\begin{proof}
Let $P = \left(G, (H_i,z_i)_{i \in I} \right)$ be a profinite peripheral system, let $G$ act continuously on a quandle $Q$ by quandle automorphisms, and let us assume that there are elements $q_i \in Q$ such that $H_i \subseteq \Stab(q_i)$, and $z_i \cdot (-) = q_i \qd (-)$. Then we already know (from the discrete case, applied to the underlying set-theoretic structures) that the associated canonical map $f : \mathcal Q(P) \rightarrow Q$ is a quandle morphism, which is bijective if and only if the $q_i$ are representatives of $G$-orbits in $Q$ and for every $i \in I$, $H_i = \Stab(q_i)$. Moreover, it is always continuous if the action of $G$ on $Q$ is. In the profinite setting, the $G/H_i$ are compact Hausdorff (since the $H_i$ are closed, we can apply Proposition~\ref{Quotient_of_Stone_space}), and since $I$ is assumed finite, $\bigsqcup G/H_i$ is compact too, so $f$ is a homeomorphism when it is a continuous bijection.

We need to check that these conditions hold for $P = \widehat{\mathcal P}(Q)$, if $Q \in pr\mathcal Qdl_s$, for any choice of representatives of its profinite orbits. But if $P = \widehat{\mathcal P}(Q)$, then $P$ is profinite: $G = \widehat G(Q)$ is a profinite, it acts continuously on $Q$ (Proposition~\ref{profinite_adj_action}), which is Hausdorff since it is profinite, so the $H_i$, which are by definition the $\Stab(q_i)$, are closed subgroups. Moreover, by definition, the $q_i$ are representatives of $\widehat G(Q)$-orbits, so all the conditions are fulfilled, hence the canonical map $\mathcal Q\bigl(\widehat{\mathcal P}(Q)\bigr) = \mathcal Q(P) \rightarrow Q$ is an isomorphism of topological quandles.
\end{proof}

\begin{rmq}\label{rk_Yetter_Pph}
An attempt at such a construction of profinite quandle from profinite peripheral systems may be found at the end of~\cite{Yetter_et_al}, although they limit themselves to \emph{profinitely connected} profinite quandles (that is, in our language, profinite quandles $Q$ such that $\widehat \pi_0(Q) = *$), and they do not get very far.  In particular, they fall short of proving that (with notations introduced in~\cref{sec_inner_aut}) if $\widehat \pi_0(Q) = *$ and $q \in Q$, then $Q \cong \mathcal Q\left(\overline{\Inn}(Q), (\Stab(q), \sigma_q)\right)$, which is readily establish with a variation on the proof of Proposition~\ref{QP=1_prof}
\end{rmq}

\begin{cor}\label{equiv_btw_prof_periph}
Two topologically finitely generated profinite quandles $Q$ and $R$ are (continuously) isomorphic if and only if for any (hence every) choice of representatives of profinite orbits, the associated peripheral systems $\widehat{\mathcal P}(Q)$ and $\widehat{\mathcal P}(R)$ are isomorphic.
\end{cor}

\begin{proof}
Since $Q$ and $R$ are finitely generated, their profinite trivialization is finite (Corollary~\ref{finite_triv_of_fg_prof}), so we can apply the construction $\widehat{\mathcal P}$, once representatives of profinite orbits are chosen. An isomorphism between $Q$ and $R$ gives an isomorphism in $pr\mathcal Qdl_s$ for any such choice of representatives; then $\widehat{\mathcal P}$ sends it to an isomorphism between $\widehat{\mathcal P}(Q)$ and $\widehat{\mathcal P}(R)$. Conversely, an isomorphism between $\widehat{\mathcal P}(Q)$ and $\widehat{\mathcal P}(R)$ is sent by $\mathcal Q$ to an isomorphism between $\mathcal Q(\widehat{\mathcal P}(Q)) \cong Q$ and $\mathcal Q(\widehat{\mathcal P}(R)) \cong R$ in $pr\mathcal Qdl_s$, which is just an isomorphism of topological quandles. 
\end{proof}

\begin{rmq}
Here again, an isomorphism of peripheral systems is what is classically called an equivalence. 
\end{rmq}

Notice that given an isomorphism $Q \cong R$ between two quandles, the construction of an isomorphism $\mathcal P(Q) \cong \mathcal P(R)$ depends on the choice of representatives of orbits for $Q$ and $R$, that is, on how we lift $Q$ and $R$ to objects of $pr\mathcal Qdl_s$ (a category equivalent to $pr\mathcal Qdl$, but there is no canonical choice of quasi-inverse for this equivalence).

\subsection{Profinite completion of peripheral systems}\label{prof_compl_of_Pph}

We define the profinite completion of a peripheral system in the obvious way:
\begin{defi}
Let $P = (G, (H_i, z_i))$ be a peripheral system. Its \emph{profinite completion} $\widehat P$ is $(\widehat G, (\overline H_i, z_i))$, where $\widehat G$ is the profinite completion of $G$, $\overline H_i$ is the closure of the image of $H_i$ in $\widehat G$, and $z_i$ denotes the image of $z_i$ in $\widehat G$.
\end{defi}

Notice that $z_i$ is central in $\overline H_i$: since $\widehat G$ is Hausdorff, the centralizer of $z_i$ is closed in $\widehat G$, and it contains (the image of) $H_i$, so it contains $\overline H_i$. It is then easy to check that this construction defines a functor from $\mathcal Pph$ to $pr\mathcal Pph$. The final piece in our construction is the following:
\begin{prop}\label{Stab_density}
For any finitely generated $Q \in \mathcal Qdl_s$, we have $\widehat{\mathcal P}(\widehat Q) \cong \widehat{\mathcal P(Q)}$, canonically. 
\end{prop}

The proof uses the following easy
\begin{lem}\label{Stab(Nx)}
Let a group $G$ act on a set $X$, and let $N$ a be a normal subgroup of $G$. Then $G/N$ acts canonically on $X/N$, and the stabilizer of $Nx$ is the image in $G/N$ of the stabilizer of $x$ in $G$.
\end{lem}

\begin{proof}
Since $gN = Ng$, letting $g \cdot Nx := Ngx$ clearly defines an action of $G$ on $X/N$, and it factors into an action of $G/N$. Moreover, if $gNx = Nx$, then $gx = nx$ for some $n \in N$, hence $n^{-1}g \in \Stab(x)$, which means that $g$ is equivalent to some element of $\Stab(x)$ modulo $N$. The converse inclusion is trivial.
\end{proof}

\begin{proof}[Proof of Proposition~\ref{Stab_density}]
We need to show that, given a quandle $Q$ with a choice of representatives of orbits $(q_i)_{i \in \pi_0(Q)}$, if we take the profinite completion of the peripheral system $\mathcal P(Q) = (G(Q), (\Stab(q_i), \eta(q_i))_i)$, we recover $\widehat{\mathcal P}(\widehat Q) = (\widehat G(\widehat Q), (\Stab(q_i), \eta(q_i))_i)$, where $q_i$ also denotes the image of $q_i$ in $\widehat Q$. We already know that the canonical map from $G(Q)$ to $\widehat G(\widehat Q)$ identifies the later with the profinite completion of $G(Q)$ (Corollary~\ref{profinite_completion_of_env_groups}). 
We also know that the $q_i$ are representatives of profinite orbits for $\widehat Q$. Indeed, we have  $\widehat \pi_0(\widehat Q) \cong \widehat{\pi_0(Q)}$ canonically (Corollary~\ref{prof_triv_of_prof_completion}). Moreover, $\pi_0(Q)$ is finite discrete, since the discrete quandle $Q$ is finitely generated (then $\pi_0(Q)$ is a finitely generated trivial quandle, hence finite). Thus $\widehat \pi_0(\widehat Q) \cong \widehat{\pi_0(Q)} = \pi_0(Q)$. 

We are left with proving that the stabilizer $S_i'$ of $q_i \in \widehat Q$ in $\widehat G(\widehat Q)$ is the closure of the image of the stabilizer $S_i \subseteq G(Q)$ of $q_i \in Q$. In other words, we need to show that the morphism from $S_i$ to $S_i'$ induced by $\iota: G(Q) \rightarrow \widehat{G(Q)} = \varprojlim(G(Q)/N)$ has dense image. This is equivalent to $S_i'$ and $S_i$ having the same image in every finite quotient $G(Q)/N$ of $G(Q)$. One inclusion is clear, so let us show that the image of $S_i'$ is contained in that of $S_i$.

Let us fix a finite-index normal subgroup $N$ of $G(Q)$. On the one hand, the image of $S_i$ in $G(Q)/N$ is the stabilizer of $Nq_i \in Q/N$ (Lemma~\ref{Stab(Nx)}). On the other hand, by construction of the action of $\widehat G(\widehat Q)$ on $\widehat Q$ (see Proposition~\ref{profinite_adj_action}), an element $g \in \widehat{G(Q)} \cong \widehat G(\widehat Q)$ fixes $q_i \in \widehat Q$ if and only if its image in each $\Aut(Q/\mathcal R)$ fixes the class of $q_i$ modulo $\mathcal R$ (for $\mathcal R$ a cofinite congruence on $Q$). In particular, its image in $\Aut(Q/N)$ fixes $Nq_i$, which means that its image in $G(Q)/N$ is in $\Stab (Nq_i)$, whence the conclusion.
\end{proof}

\begin{cor}\label{equiv_btw_prof_comp_of_periph}
Let $Q$ and $R$ be finitely generated (discrete) quandles. Their profinite completions $\widehat Q$ and $\widehat R$ are (continuously) isomorphic if and only if the profinite completion of their associated peripheral systems are isomorphic.
\end{cor}

\begin{proof}
Notice that profinite completions of finitely generated quandles are clearly topologically finitely generated. Then Corollary~\ref{equiv_btw_prof_periph} says that $\widehat Q$ and $\widehat R$ are isomorphic if and only $\widehat{\mathcal P}(\widehat Q)$ and $\widehat{\mathcal P}(\widehat R)$ are isomorphic (for any choices of representatives of profinite orbits). By Proposition~\ref{Stab_density}, these are canonically isomorphic to the profinite completions of $\mathcal P(Q)$ and $\mathcal P(R)$ (for any choices of representatives of orbits, which also give representatives of profinite orbits for the profinite completions, by Corollary~\ref{finite_triv_of_fg_prof} and Corollary~\ref{prof_triv_of_prof_completion}), whence the conclusion.
\end{proof}

Finally, combining Corollary~\ref{equiv_btw_prof_comp_of_periph} with Corollary~\ref{cor_col_and_prof}, we get our main theorem:

\begin{theo}\label{main_thm}
Two links $K$ and $L$  have the same number of colouring by every finite quandle if and only if the profinite completion of their associated peripheral systems are equivalent.
\end{theo}

\begin{rmq}
The term ``equivalent" is the one usually used in the literature, to mean a kind of isomorphism \emph{up to conjugation}, which is exactly our notion of an isomorphism between peripheral systems. 
\end{rmq}

\begin{rmq}
Our notion of the peripheral system associated to a link is a weak one, since it does not include a choice of \emph{preferred longitudes}. See the discussion in~\cite[\S~7.6]{Darne7} for more details on this point.
\end{rmq}

\section{Inner automorphisms}\label{sec_inner_aut}

The rest of the paper is devoted to some additional profinite quandle theory, culminating in the proof of our criterion for a topological quandle to be profinite (Theorem~\ref{profiniteness_criterion}). The main character in this story is the group $\Inn(Q)$ of inner automorphisms of a quandle $Q$, and its closure $\overline{\Inn}(Q)$ in $\Homeo(Q)$ when $Q$ is a topological quandle. 

This section is devoted to the introduction and study of the inner automorphism group and its closure. We notably establish its functoriality properties with respect to quotients of quandles. We do this first in the (classical) discrete setting, then in the more general topological setting, and finally in the profinite setting.

\subsection{The discrete case}

Let us first recall the definition and basic properties of inner automorphisms in (discrete) quandle theory.

\begin{defi}
The \emph{group of inner automorphisms} of a quandle $Q$, denoted by $\Inn(Q)$, is the subgroup of $\Aut(Q)$ generated by the $\sigma_x = x \qd(-)$ (for $x \in Q$).
\end{defi}

\begin{fact}\label{sigma_morphisme}
The canonical map $\sigma: x \mapsto \sigma_x$ from $Q$ to $\Inn(Q)$ is a quandle morphism.
\end{fact}

\begin{proof}
Self-distributivity gives $x \qd (y \qd (x^{-1} \qd (-))) = (x \qd y) \qd (-)$, which means exactly that $\sigma_x \sigma_y \sigma_x^{-1} = \sigma_{x \qd y}$, i.e.~$\sigma(x) \qd \sigma(y) = \sigma(x \qd y)$.
\end{proof}

This gives an alternative point of view on inner automorphisms: the above quandle morphism induces a unique group morphism $G(Q) \rightarrow \Inn(Q)$, sending $\eta(x)$ to $x \qd(-)$. It equals the adjoint action (because it does so when restricted to a set of generators), and it is surjective (because its image contains generators of $\Inn(Q)$). As a consequence, $\Inn(Q)$ is also the image of the adjoint action $G(Q) \rightarrow \Aut(Q)$.

\begin{rmq}
In Lemma~\ref{quotient_by_group}, the condition means that $\alpha(G)$ is normalized by $\Inn(Q)$. This holds for instance if $Q$ is a normal subgroup of $Q$. In particular, for $G = \Inn(Q)$, we recover the trivialization $\pi_0(Q) = Q/\Inn(Q) = Q/G(Q)$.
\end{rmq}

The properties of a quandle are closely linked to the properties of its inner automorphism group. Here is an easy example of this phenomenon:
\begin{lem}\label{finite_Inn}
Let $Q$ be a finitely generated quandle whose inner automorphism group is finite. Then $Q$ is finite.
\end{lem}

\begin{proof}
Let $S$ be a finite set of generators for $Q$. Let $Q'$ be the subset $\bigcup_{\varphi \in \Inn(Q)} \varphi(S)$ of $Q$, which is finite, by hypothesis. If $x \in Q$ and $y = \varphi(s) \in Q'$ (with $\varphi \in \Inn(Q)$ and $s \in S$), then $x \qd y = x \qd \varphi(s) = (\sigma_x \circ \varphi)(s) \in Q'$. In particular, $Q'$ is a subquandle of~$Q$; since it contains $S$, we have $Q' = Q$, so $Q$ is finite. 
\end{proof}

Let $\mathcal R$ be a congruence on a quandle $Q$, and let $x \in Q$. Then, by definition of the law on $Q/\mathcal R$, the inner automorphism $\sigma_{\overline x} = \overline x \qd (-)$ is induced by $\sigma_x = x \qd (-)$. The fact that this law is well-defined implies that every $\sigma_x$ factors through the quotient by $\mathcal R$. By composing such maps, we see that this is true of every inner automorphism of $Q$. In other words, every $\varphi \in \Inn(Q)$ factors uniquely through the quotient:
\[\begin{tikzcd}[column sep=45]
Q \ar[r, "\varphi"] \ar[d, two heads, "p"] 
&Q \ar[d, two heads, "p"] \\
Q/\mathcal R \ar[r, dashed, "\exists !\ p_*(\varphi)"]
&Q/\mathcal R.
\end{tikzcd}\]
Since every surjection can be seen as a quotient by a congruence, we recover: 
\begin{prop}\textup{\cite[Prop.~2.31]{Eisermann}}\label{Inn(surj)}
Let $p: Q \twoheadrightarrow R$ be a surjective quandle morphism. Then $p$ induces a unique group morphism $p_*: \Inn(Q) \rightarrow \Inn(Q')$ sending $\sigma_x$ to $\sigma_{p(x)}$, and $p_*$ is surjective.
\end{prop}

\begin{proof}
The map $p_*$ defined before the statement is easily seen to be a morphism: it preserves composition. The $\sigma_x$ generate $\Inn(Q)$, hence $p_*$ is uniquely determined by $p_*(\sigma_x) = \sigma_{p(x)}$. Finally, the $\sigma_{p(x)}$ generate $\Inn(R)$, so $p_*$ is surjective.
\end{proof}

\begin{rmq}
This construction defines a functor from the category of surjections between quandles  to the category of surjections between groups.
\end{rmq}

It is often useful to see the following counterpart of Lemma~\ref{equivariance} as an \emph{equivariance property} (of $p$ with respect to actions of inner automorphisms).
\begin{lem}\label{equivariance_Inn(Q)}
Let $p: Q \twoheadrightarrow R$ be a surjective quandle morphism. We have:
\[\forall \varphi \in \Inn(Q),\ \forall x \in Q,\  p(\varphi (x)) = p_*(\varphi)(p(x)).\]
\end{lem}

\begin{proof}
This holds by construction of $p_*(\varphi)$ as the factorisation of $\varphi$ through the quotient.
\end{proof}

\subsection{The topological case}

If $Q$ is a topological quandle, automorphisms of $Q$  are self-homeomorphisms of $Q$ that are quandle automorphisms. They form a subgroup $\Aut_{\mathcal C}(Q)$ of $\Homeo(Q)$. The latter is endowed with the compact-open topology (see~\cref{Homeo(X)}), which is not in general a group topology. Recall that this topology makes the evaluation maps $\ev_q : \Homeo(Q) \rightarrow Q$ sending $f$ to $f(q)$ continuous, for all $q \in Q$, since $\ev_q^{-1}(U) = \mathcal O_{x,U}$, for every open subset $U$ of $Q$.

\begin{lem}
If $Q$ is Hausdorff, then $\Aut_{\mathcal C}(Q)$ is closed in $\Homeo(Q)$.
\end{lem}

\begin{proof}
If $Q$ is Hausdorff, then the diagonal $\Delta_Q$ is closed in $Q^2$. Hence $\Aut_{\mathcal C}(Q)$ is closed in $\Homeo(Q)$: it
is the intersection of the $\Phi_{x,y}^{-1}(\Delta_Q)$ (for $x,y \in Q$), where $\Phi_{x,y} : \Homeo(Q) \rightarrow Q^2$ is the continuous map given by $f \mapsto (f(x \qd y), f(x) \qd f(y))$.  
\end{proof}

\begin{rmq}\label{rk_Singh_main_mistake}
Then same argument may be used to show that $\End_{\mathcal C}(Q)$ is closed in $\mathcal C(Q,Q)$, but $\Aut_{\mathcal C}(Q)$ is not closed in $\mathcal C(Q,Q)$ in general, since $\Homeo(Q)$ is usually not closed in $\mathcal C(Q,Q)$. Overlooking this essential distinction was the main source of errors in~\cite[\S~5]{Singh} (see Rk.~\ref{rk_Singh_counter-ex})
\end{rmq}

If $Q$ is a topological quandle, then the $\sigma_x = x \qd (-)$ are self-homeomorphisms of $Q$, so $\Inn(Q)$ is a subgroup of $\Homeo(Q)$, which is endowed with the compact-open topology. We can consider the closure $\overline\Inn(Q)$ of $\Inn(Q)$ in $\Homeo(Q)$ (hence in $\Aut_{\mathcal C}(Q)$ if $Q$ is Hausdorff). When $\Homeo(Q)$ is a topological group (in particular when $Q$ is compact Hausdorff -- see Thm.~\ref{Arens}), $\overline\Inn(Q)$ is a subgroup thereof.

\begin{prop}\label{closed_Inn(surj)}
Let $p : Q \twoheadrightarrow R$ be a surjective continuous morphism between compact Hausdorff topological quandle. Then it induces a continuous group morphism $\overline p_*: \overline\Inn(Q) \rightarrow \overline \Inn(R)$. It is the unique continuous group morphism sending $\sigma_x$ to $\sigma_{p(x)}$ (for all $x \in Q$), and its image is dense.
\end{prop}

\begin{proof}
Let $\mathcal R_p$ be the kernel of $p$ (Definition~\ref{def_ker}). The map from $Q/\mathcal R_p$ to $R$ induced by $p$ is a continuous bijection from a compact space to a Hausdorff one, so it is a homeomorphism $Q/\mathcal R_p \cong R$. Let us consider the group $\Homeo(Q, \mathcal R_p)$ from~\cref{Homeo(X)}. Since $\mathcal R_p$ is a congruence, it is preserved by any inner automorphism (see the argument preceding Proposition~\ref{Inn(surj)}); in other words, $\Inn(Q) \subseteq \Homeo(Q, \mathcal R_p)$. Notice that $\mathcal R_p = (p \times p)^{-1}(\Delta_R)$ is closed since $R$ is Hausdorff. Hence $\overline\Inn(Q) \subseteq \Homeo(Q, \mathcal R_p)$ by Lemma~\ref{Homeo(X,R)_closed}. The canonical map $\Homeo(Q, \mathcal R_p) \rightarrow \Homeo(R)$ is continuous (Lemma~\ref{Homeo(X,R)_to_Homeo(X/R)}), and it sends $\Inn(Q)$ to $\Inn(R)$ (Proposition~\ref{Inn(surj)}), so it sends $\overline\Inn(Q)$  to $\overline\Inn(R)$. This shows that $\overline p_*$ is well-defined; moreover, the unicity statement follows from the unicity of its restriction $p_*$ to the dense subset $\Inn(Q)$ of $\overline\Inn(Q)$ (see Proposition~\ref{Inn(surj)}). Finally, $p_*$ is surjective, so $\ima(\overline p_*)$ contains $\ima(p_*) = \Inn(R)$, which is dense in its closure.  
\end{proof}

\begin{rmq}
The compactness condition can be relaxed, if we \emph{assume} that $p$ is the quotient by a congruence; see Remarks~\ref{rk_gen1} and~\ref{rk_gen2}. Notice that such a generalization can lead to the homeomorphism spaces not being topological groups; then $\overline\Inn(Q)$ need not be a group, and $\overline p_*$ is not a morphism, but the unique map restricting to $p_*$ (from Proposition~\ref{Inn(surj)}) on $\Inn(Q)$.
\end{rmq}

If $Q$ is any topological quandle (not necessarily compact), one can replace the use of Lemma~\ref{Homeo(X,R)_to_Homeo(X/R)} in the previous proof by Lemma~\ref{p*_continuous} to get, with exactly the same reasoning (noticing that $\overline \Inn(R) = \Inn(R)$ in the discrete space $\Homeo(R)$):
\begin{prop}\label{closed_Inn(surj_to_finite)}
Let $p : Q \twoheadrightarrow R$ be a surjective continuous morphism from a topological quandle $Q$ to a finite discrete quandle $R$. Then the induced morphism $p_*: \Inn(Q) \twoheadrightarrow \Inn(R)$ from Proposition~\ref{Inn(surj)} is continuous, and it extends (uniquely) to a continuous map $\overline p_*: \overline\Inn(Q) \twoheadrightarrow \Inn(R)$.
\end{prop}

\begin{rmq}
Without further hypothesis, $\overline\Inn(Q)$ need not be a group. If $\Homeo(X)$ is a topological group for the compact-open topology, then it is, and $\overline p_*$ is a morphism.
\end{rmq}

We conclude this investigation of topological inner automorphisms with the following adaptation of the equivariance Lemma~\ref{equivariance}:
\begin{lem}\label{top_equivariance}
Let $p: Q \twoheadrightarrow R$ be a surjective continuous quandle morphism. Assume either that $Q$ and $R$ are compact Hausdorff or that $R$ is finite discrete. Then we have:
\[\forall \varphi \in \overline\Inn(Q),\ \forall x \in Q,\  p(\varphi (x)) = \overline p_*(\varphi)(p(x)).\]
\end{lem}

\begin{proof}
For $\varphi \in \Inn(Q)$, we have $\overline p_*(\varphi) = p_*(\varphi)$, so the equality holds by Lemma~\ref{equivariance}. But this equality is between the continuous functions $p \circ \ev_x$ and $\ev_{p(x)} \circ \overline p_*$ of $\varphi \in \overline\Inn(Q)$, which take values in the Hausdorff space $R$. Since they coincide on the dense subset $\Inn(Q)$, we conclude that they coincide everywhere.
\end{proof} 

So far in the current section, we have not used explicitly the following crucial fact, which is worth mentioning:
\begin{fact}\label{sigma_morphisme_continu}
If $Q$ is a topological quandle, then the morphism $\sigma: Q \rightarrow \Inn(Q)$ from Fact~\ref{sigma_morphisme} is continuous.
\end{fact}

\begin{proof}
The map $\sigma: Q \rightarrow \mathcal C(Q,Q)$ is adjoint to $\qd : Q^2 \rightarrow Q$, which is continuous, so the result follows from Proposition~\ref{prop_adj_direct}.
\end{proof}

\subsection{The profinite case}

We now show that the group of inner automorphisms of a profinite quandle is relatively compact in $\Homeo(X)$. We will show later that this property characterizes profinite quandles among those whose underlying space is profinite (Theorem~\ref{profiniteness_criterion}). 
\begin{prop}\label{Inn(prof)_is_prof}
If a topological quandle $Q$ is profinite, then $\overline{\Inn}(Q)$ is a profinite group. Precisely:
\[\overline{\Inn}(Q) \cong \varprojlim \Inn(Q/\mathcal R),\]
where the limit is taken over all cofinite congruences on $Q$ (that is, over all finite discrete quotients of $Q$).
\end{prop}

\begin{proof}
For every inclusion $\mathcal R \subseteq \mathcal R'$ of cofinite congruences on $Q$, we get a canonical projection $p^{\mathcal R, \mathcal R'}: Q/\mathcal R \twoheadrightarrow Q/\mathcal R'$, whence a surjective morphism $p^{\mathcal R, \mathcal R'}_*$ between the corresponding inner automorphisms groups, using Proposition~\ref{Inn(surj)}. It then makes sense to consider the projective limit of the statement. We can apply Proposition~\ref{closed_Inn(surj)} to get canonical continuous group morphisms $p^{\mathcal R}_* : \overline{\Inn}(Q) \twoheadrightarrow \Inn(Q/\mathcal R)$. Notice that $\Inn(Q/\mathcal R)$ is equal to its closure in the discrete space $\Homeo(Q/\mathcal R)$, and that the density of the image of $p^{\mathcal R}_*$ implies its surjectivity. If $\mathcal R \subseteq \mathcal R'$, the morphism $p^{\mathcal R, \mathcal R'}_* \circ p^{\mathcal R}_*$ is continuous and sends $\sigma_q$ to $\sigma_{q\!\mod \mathcal R'}$, so it equals $p^{\mathcal R'}$, by unicity in Proposition~\ref{closed_Inn(surj)}. As a consequence, the $p^{\mathcal R}$ induce a canonical continuous morphism $\rho$ from $\overline{\Inn}(Q)$ to $\mathcal I := \varprojlim \Inn(Q/\mathcal R)$. Moreover, since the $p^{\mathcal R}$ are surjective, the image of $\rho$ is dense in~$\mathcal I$.

An element $\varphi = (\varphi_{\mathcal R})_{\mathcal R} \in \mathcal I$ defines an automorphism of the projective system  $(Q/\mathcal R)_{\mathcal R}$.  If $Q$ is profinite, we have $Q = \varprojlim (Q/\mathcal R)$, so such a $\varphi$ defines a continuous automorphism of $Q$. This construction gives a morphism $\xi$ from $\mathcal I$ to $\Aut_{\mathcal C}(Q) \subseteq \Homeo(Q)$, which is continuous. Thanks to Proposition~\ref{prop_adj_direct}, one way to prove its continuity is to show the continuity of its adjoint $\mathcal I \times Q \rightarrow Q$, which is none other than the projective limit of the evaluation maps $\Inn(Q/\mathcal R) \times Q/\mathcal R \rightarrow Q/\mathcal R$. The latter are maps between discrete spaces, so they are continuous, whence the continuity of $\xi$. Moreover, $\xi$ is injective, because each $\varphi_{\mathcal R}$ can be recovered from $\xi(\varphi)$ as the induced automorphism of the quotient $Q/\mathcal R$. 

It is easy to check that $\xi \circ \rho$ is just the inclusion of $\overline\Inn(Q)$ into $\Homeo(Q)$. Now, the density of the image of $\rho$ implies that $\overline \Inn(Q)$ is dense in $\xi(\mathcal I)$. But $\overline \Inn(Q)$ is closed, hence $\overline \Inn(Q) = \xi(\mathcal I)$. As a consequence, $\xi$ (corestricted to its image) is a continuous inverse of $\rho$, which is thus a homeomorphism. 
\end{proof}

This can be used to recover the fact that for a profinite quandle $Q$, the adjoint action extends (necessarily uniquely) to a continuous action of $\widehat G(Q)$ (see Proposition~\ref{profinite_adj_action}):

\begin{cor}\label{action_of_prof_G(Q)}
If a topological quandle $Q$ is profinite, then the canonical continuous action of $G(Q)$ on $Q$ factors into an action of its profinite completion $\widehat G(Q)$.
\end{cor}

\begin{proof}
The map $G(Q) \rightarrow \Aut_{\mathcal C}(Q)$, which is continuous (Fact~\ref{continuous_action_of_G(Q)}), takes its values in the subgroup $\Inn(Q)$ of $\overline{\Inn}(Q)$, which is profinite by Proposition~\ref{Inn(prof)_is_prof}, so it induces a unique continuous morphism from the profinite completion $\widehat G(Q)$ of $G(Q)$ to $\overline{\Inn}(Q) \subseteq \Aut_{\mathcal C}(Q)$.
\end{proof}

\begin{rmq}\label{prof_G(Q)_vs_Inn(Q)}
Conversely, we could have used Proposition~\ref{profinite_adj_action} to recover Proposition~\ref{Inn(prof)_is_prof}, or at least the fact that $\overline{\Inn}(Q)$ is profinite when $Q$ is. Indeed, if $Q$ is profinite, the continuous map $\widehat G(Q) \rightarrow \Aut_{\mathcal C}(Q)$ given by Proposition~\ref{profinite_adj_action} has compact image, and its image contains $\Inn(Q)$, which is the image of the map $G(Q) \rightarrow \Aut_{\mathcal C}(Q)$. This shows that $\Inn(Q)$ is relatively compact in $\Aut_{\mathcal C}(Q) \subseteq \Homeo(Q)$, so $\overline{\Inn}(Q)$ is a compact subgroup of $\Homeo(Q)$ (which is a topological group by Theorem~\ref{Arens}). The latter is totally disconnected (Proposition~\ref{Homeo(X)_tot_disc}), hence $\overline{\Inn}(Q)$ is a profinite group, by Corollary~\ref{profiniteness_criterion_groups}.

In fact, the projective system that appears in the proof of Proposition~\ref{Inn(prof)_is_prof} is a quotient of the one appearing in the proof of Proposition~\ref{profinite_adj_action}.
\end{rmq}

\subsection{Relation to previous work}

\paragraph{Relation to Singh's work : $\Inn(Q)$ is not profinite}\label{rk_Singh_counter-ex}
There are some mistakes in~\cite[\S~5]{Singh}. Namely, the statements about endomorphisms are true, but the proofs do not carry over to automorphisms. This is because of the distinction pointed out in Remark~\ref{rk_Singh_main_mistake}. For automorphisms, the main claim is still true, but with a different proof. For inner automorphisms, the main claim is false~: if $Q$ is profinite, $\Inn(Q)$ is not, even if $Q$ is (topologically) finitely generated. 

An infinitely generated profinite quandle $Q$ such that $\Inn(Q)\neq \overline{\Inn}(Q)$ was already described in~\cite[Ex.~4.3]{Yetter_et_al}. We now describe a finitely generated one. 

\begin{ex}[Counter-example to {\cite[Th.~5.2]{Singh}} for inner automorphisms]
Let us consider the profinite quandle $Q = \{q_0\} \sqcup \widehat \Z$ from Examples~\ref{ex_comp_of_Z} and~\ref{ex_comp_of_Z_2}. It is topologically generated by the two elements $q_0$ and $0 \in \widehat \Z$, since $q_0 \qdn{n} 0 = n$, so $\langle q_0, 0 \rangle$ contains the dense subset $\{q_0\} \sqcup \Z$.  Moreover, $\Inn(Q)$ is generated by $q_0 \qd (-)$, which is the translation by $1$ on $\widehat \Z$, of infinite order, so $\Inn(Q) \cong \Z$. But it is the copy of $\Z$ generated by $1$ in the group of translations of $\widehat \Z$, which canonically identifies with $\widehat \Z$. The latter is a compact subgroup of $\Homeo(Q)$ in which $\Inn(Q)$ is dense, hence equal to $\overline{\Inn}(Q)$. Thus, $\overline{\Inn}(Q) \cong \widehat \Z$ is much large than $\Inn(Q) \cong \Z$. In particular, $\Inn(Q)$ is not closed, hence not compact, hence not profinite.
\end{ex}

As for profinite quandles $Q$ such that $\Aut_{\mathcal C}(Q)$ is not closed in $\mathcal C(Q,Q)$ (and hence not compact), trivial ones give infinitely generated examples: then $\Aut_{\mathcal C}(Q) = \Homeo(Q)$ is not closed in $\mathcal C(Q,Q)$ in general. For instance, if $Q = \{0,1\}^{\mathbb N}$, and $f_n$ permutes cyclically the $n$ first coordinates, then $(f_n)$ converges to the shift $(x_1, x_2, ...) \mapsto (x_2, x_3, ...)$, which is clearly not injective. This shows that there is no equivalent of~\cite[Prop.~5.6]{Singh} for automorphisms, which prevents the reasoning made in~\cite[\S~5]{Singh} for endomorphisms to carry over to automorphisms. However, the part of~\cite[Th.~5.12]{Singh} about automorphism groups does hold. We give a simple proof which by-passes any consideration of $\mathcal C(Q,Q)$ (except for the definition of the topology on its subspaces), and also recovers the case of endomorphisms.

Let us first recall that a set $I$ of cofinite congruences on a topological quandle $Q$ is a \emph{fundamental system of cofinite congruences} if, for every cofinite congruence $\beta$, then there is a congruence $\alpha \in I$ which is finer than $\beta$. Which means exactly that $I$ is cofinal in the directed set of cofinite congruences on $Q$. For instance, if $Q$ is profinite and topologically finitely generated, then the $\mathcal R_n$ from~\cref{par_prof_comp_fg} form a fundamental system of cofinite congruences, and these are \emph{totally invariant} (that is, invariant under all endomorphisms of $Q$).
\begin{prop}
Let $Q$ be a profinite quandle endowed with a fundamental system $I$ of cofinite congruences which are totally invariant. Then $\End_{\mathcal C}(Q)$ and $\Aut_{\mathcal C}(Q)$ are profinite. Precisely:
\[\End_{\mathcal C}(Q) \cong \varprojlim_{\alpha \in I} \left( \End(Q/\alpha) \right)\ \ \text{ \rm{and} }\ \ \Aut_{\mathcal C}(Q) \cong \varprojlim_{\alpha \in I} \left( \Aut(Q/\alpha) \right).\]
\end{prop}

\begin{proof}
The total invariance of the $\alpha \in I$ imply that the canonical maps $\Phi_\alpha : \End_{\mathcal C}(Q) \rightarrow \End(Q/\alpha)$ are well-defined and induce a map $\Phi$ from $\End_{\mathcal C}(Q)$ to $\mathcal E := \varprojlim_{\alpha \in I} \left( \End(Q/\alpha) \right)$. Since $Q$ is profinite and $I$ is a cofinal in the directed set of cofinite congruences on $Q$, we have $Q \cong \varprojlim_{\alpha \in I} (Q/\alpha)$, canonically; we identify them in the sequel. Thus, if $u = (u_\alpha)_{\alpha \in I} \in \mathcal E$, then $(x_\alpha)_\alpha \mapsto (u_\alpha(x_\alpha))_\alpha$ defines a continuous endomorphism $\Psi(u)$ of $\varprojlim_{\alpha \in I} (Q/\alpha) = Q$. The map $\Psi : \mathcal E \rightarrow \End_{\mathcal C}(Q)$ so defined is clearly a set-theoretic inverse to $\Psi$. Moreover, it takes values in $\mathcal C(Q,Q)$, hence its continuity is equivalent to the continuity of its adjoint $\mathcal E \times Q \rightarrow Q$. But the latter is none other than the projective limit of the canonical maps $\End(Q/\alpha) \times Q/\alpha \rightarrow Q/\alpha$, which are continuous, since they are maps between (finite) discrete sets. Thus, we conclude that $\Psi$ and $\Phi$ are mutually inverse homeomorphisms. 

Moreover, it is clear that $\Phi$ and $\Psi$ are isomorphisms of topological monoids, so they induce mutually inverse isomorphisms between the topological groups of units of these monoids, which are respectively $\Aut_{\mathcal C}(Q)$ and $\varprojlim_{\alpha \in I} \left( \Aut(Q/\alpha) \right)$.
\end{proof}

\begin{rmq}
Our proof says in particular that the distinction between $\varprojlim\left( \Aut(Q/\alpha) \right)$ and $\varprojlim \Phi_\alpha\! \left( \Aut(Q/\alpha) \right)$ from~\cite{Singh} (where $\Phi_\alpha$ is denoted by $\Omega_\alpha$) is irrelevant. Indeed, any element $u = (u_\alpha)_\alpha$ of $\varprojlim\left( \Aut(Q/\alpha) \right)$ equals $\Phi(\Psi(u))$, which means that $u_\alpha = \Phi_\alpha(\Psi(u))$ for every $\alpha$, so it is in fact an element of $\varprojlim \Phi_\alpha\! \left( \Aut(Q/\alpha) \right)$.
\end{rmq}

\begin{rmq}
This argument does not carry over to inner automorphisms, for which, as we have seen, the statement fails. This comes from the fact that in an element $u = (u_\alpha)_\alpha$, $\Psi(u)$ may not be an inner automorphism even if every $u_\alpha$ is. Precisely, the existence of a $x_\alpha$ such that $u_\alpha = x_\alpha \qd(-)$ for each individual $\alpha$ does not imply the existence of a coherent family of such $x_\alpha$'s (which would give $x \in Q$ such that $\Phi(u) = x \qd(-)$).
\end{rmq}

\paragraph{Relation to \cite{Yetter_et_al}: completions of $\Inn(Q)$}\label{rk_Yetter_insight}
One of the main insights of \cite{Yetter_et_al} was the idea that in profinite quandle theory, the relevant inner automorphism group was not $\Inn(Q)$, but some completion of it. They only failed to recognize that their $\widehat{\Inn}_D(Q)$ was actually independent of the choice of diagram $D$, and that it identified with the closure $\overline{\Inn}(Q)$ of $\Inn(Q)$ in $\Homeo(Q)$. The fact that it does not depend of $D$ is actually rather easy to see: up to removing isomorphic copies of quotients of $Q$, their $D$ must be a subsystem of the projective system of all finite quotients of $Q$, necessarily cofinal (else the projective limit would not be $\widehat Q$). But then $\Inn(D)$ must be a cofinal subsystem of the projective system defining $\mathcal I$ in the above proof of Proposition~\ref{Inn(prof)_is_prof}. Hence $\widehat{\Inn}_D(Q) \cong \mathcal I$, canonically. As for the identification with $\overline{\Inn}(Q)$, it correspond to the construction of the canonical action of $\widehat{\Inn}_D(Q)$ on $Q$ (that we can see as a map to $\overline{\Inn}(Q) \subseteq \Homeo(Q)$), which they did, and it is essentially the same as the construction of the isomorphism $\mathcal I \cong \overline{\Inn}(Q)$ in our proof.

They also consider the profinite completion $\widehat{\Inn(Q)}$ of $\Inn(Q)$, which is less relevant to the theory than $\overline{\Inn}(Q)$, since it bears less relation to $Q$. 

\begin{ex}
let $Q$ be $\{q_0\} \sqcup \Z_p$, defined as in Examples~\ref{ex_comp_of_Z} but with $\Z_p$ (or any profinite group generated by one element) instead of $\widehat \Z$. Then $\overline{\Inn}(Q) \cong \Z_p$ (as in Remark~\ref{rk_Singh_counter-ex}), whereas $\Inn(Q) \cong \Z$ always has $\widehat \Z$ as its profinite completion.
\end{ex}

Notice that $\overline{\Inn}(Q)$ is always the image of $\widehat{\Inn(Q)}$ by the canonical map $\widehat{\Inn(Q)} \rightarrow \Homeo(Q)$ corresponding to the action of $Q$ that they constructed. The latter which can also be seen as induced by the usual map $\Inn(Q) \hookrightarrow \Homeo(Q)$, once we know that its image is contained in a profinite group (namely, $\overline{\Inn}(Q)$). In the previous example, the map $\widehat{\Inn(Q)} \twoheadrightarrow \overline{\Inn}(Q)$ identifies with the canonical projection $\widehat \Z \twoheadrightarrow \Z_p$.

\section{A criterion for profiniteness}\label{sec_criterion}

The goal of the present section is to show a criterion for a topological quandle to be profinite. Namely, a topological quandle $Q$ is profinite if and only if it is compact and totally disconnected, and $\Inn(Q)$ is relatively compact in $\Homeo(Q)$ (Theorem~\ref{profiniteness_criterion}).

\subsection{Introduction and statement of the criterion}

For groups, profiniteness is a purely topological property, since profinite groups are the ones that are \emph{profinite spaces}, that is, compact and totally disconnected (see~\cref{par_profinite_groups_and_spaces}). For quandles, this is no longer true, as the following example shows: 
\begin{ex}
Let $Q$ be, as a topological space, the disjoint union of a singleton $\{t\}$ and the Cantor set $K = \{0,1\}^{\mathbb N}$ (endowed with the product topology), whose elements are to be written as infinite words. Let $f$ be the self-homeomorphism of $K$ defined by $f(0w) = 00w$, $f(10w) = 01w$ and $f(11w) = 1w$ (for any $w \in K$). Let the quandle law be given by $t \qd w = f^{-1}(w)$ if $w \in K$, and $x \qd y = y$ if $x \neq t$ or $y \notin K$. This defines a topological quandle, that is clearly a profinite space. However, if $\mathcal R$ is a cofinite congruence on $K$, then the class $C_0$ of the word $w_0 = 000 \cdots$, which is open, must contain a neighbourhood $0^NK$ of $w_0$ (where $0^N$ means $N$ times the letter $0$). Since $\mathcal R$ is a congruence, $C_0$ is stable by $t \qd (-) = f^{-1}$. But by construction of $f$, for any word $w$ different from $w_1 = 111 \cdots$, $f^k(w)$ must land in $0^NK$ for $k$ big enough. In other words, $C_0 \supseteq \bigcup f^{-k}(0^NK)$ = $K - \{w_1\}$. Since $C_0$ is also closed, it contains $w_1$ too. Finally, $\mathcal R$ has at most two classes ($\{t\}$ and $K$), and $Q$ is clearly not residually finite.
\end{ex}

The previous example has been inspired by~\cite{Schneider}. A similar example is described in \cite[Ex.~3.17]{Yetter_et_al}. The problem is in fact that the family $\{f^k\}$ of inner automorphisms is not equicontinuous at $w_0$ (in \cite[Ex.~3.17]{Yetter_et_al}, equicontinuity at $\infty$ was at fault). Since inner automorphisms are translations (in the sense of \cite{Schneider}), and an algebra $X$ is profinite if and only if it is a profinite space and its translation monoid is equicontinuous \cite[Thm.~4.7]{Schneider}, such a quandle could not be profinite. Recall that the Arzel\`a–Ascoli theorem \cite[\S~2.5,~Cor.~3]{Bourbaki2} implies that the equicontinuity of the translation monoid of $X$ is equivalent to the compactness of its closure in the space $\mathcal C(X,X)$ of continuous functions from $X$ to itself. For quandles, we can get a different characterisation of profiniteness by considering only inner automorphisms, instead of all translations:
\begin{theo}\label{profiniteness_criterion}
A topological quandle $Q$ is profinite if and only if it is a profinite space and $\overline\Inn(Q)$ is compact.
\end{theo}

\begin{rmq}
This characterisation is not \emph{a priori} stronger than the one obtained by applying the result of~\cite{Schneider}: one cannot be deduced directly from the other. Indeed, we considers a smaller set of functions, but we ask for it to be relatively compact in $\Homeo(X)$, which is a stronger requirement than its relative compactness in $\mathcal C(X,X)$, since $\Homeo(X)$ is in general not closed in $\mathcal C(X,X)$. However, our statement is much more adapted to the theory of quandles, where the inner automorphism group is a fundamental object, contrary to the monoid of translations.
\end{rmq}

The easiest half of Theorem~\ref{profiniteness_criterion} has already been proved, as Corollary~\ref{qdl_prof_cpct_and_discontinuous} and Proposition~\ref{Inn(prof)_is_prof} (see also Remark~\ref{prof_G(Q)_vs_Inn(Q)}). The rest of the present section is devoted to a proof of the other half: if a topological quandle $Q$ is a profinite space and if $\overline\Inn(Q)$ is compact, we are going to show that $Q$ is profinite. This will require the construction of enough finite quotients of $Q$, using a construction of a quandle from a crossed set that we investigate first.

\subsection{Quandles from crossed sets}\label{par_crossed_sets}

\begin{defi}\label{def_crossed_sets}
Let $G$ be a group. A \emph{crossed $G$-set} $(X, \partial)$ is a set $X$ endowed with an action of $G$ and with a map $\partial : X \rightarrow G$ which is $G$-equivariant with respect to the conjugation action of $G$ on itself:
\[\forall g \in G,\ \forall x \in X,\ \partial(g \cdot x) = g \partial (x) g^{-1}.\]
A \emph{crossed set} $(G,X, \partial)$ is the data of a group $G$ and a crossed $G$-set $(X, \partial)$.
\end{defi}

Let $X$ be a crossed $G$-set. Let $G$ act on $X \sqcup G$ by the corresponding action on $X$, and by conjugation on itself (that is, $g \cdot h := g h g^{-1}$). Let us define a binary law $\qd$ on $X \sqcup G$ by posing:
\[\forall g, \in G,\ \forall x \in X,\ 
\begin{cases}
g \qd (-) = g \cdot (-) ; \\
x \qd (-) = \partial (x) \cdot (-).
\end{cases}\]

\begin{lem}\label{qd_from_crossed_set}
The above formulas defines a rack law on $X \sqcup G$, which is a quandle law if $\partial (x) \cdot x = x$, for all $x \in X$.
\end{lem}

\begin{proof}
This is just a matter of easy verification: one needs to write down the quandle axioms, distinguishing, for each variable, the cases where it belongs to $X$ and the cases where it belongs to $G$. 
\end{proof}

Crossed sets form a category, morphisms from $(G,X, \partial_X)$ to $(H,Y, \partial_Y)$ being the obvious ones: couples $(\varphi, f)$ where $\varphi: G \rightarrow H$ is a group morphism and $f : X \rightarrow Y$ is a $\varphi$-equivariant map (for all $g \in G$ and $x \in X$, $f(g \cdot x) = \varphi(g) \cdot f(x)$), such that $\varphi \circ \partial_X = \partial_Y \circ f$. The above construction $(G,X, \partial_X) \mapsto X \sqcup G$ is easily seen to be a functor from this category to the category of racks and rack morphisms. 

\medskip

Let $Q$ be a quandle. The canonical action of $\Inn(Q)$ on $Q$, together with the canonical map $\sigma: Q \rightarrow \Inn(Q)$, define a crossed set satisfying the condition of Lemma~\ref{qd_from_crossed_set}. Indeed, the equivariance condition can be checked for generators $\sigma(x) = \sigma_x = x \qd(-)$ of $\Inn(Q)$ (for $x \in Q$), and: 
\[\sigma(\sigma_x \cdot y) = \sigma(x \qd y) = \sigma(x) \qd \sigma(y) = \sigma_x \sigma(y) \sigma_x^{-1}.\]
The other condition means that $x$ is a fixed point of $\sigma(x)$, which it is. Thus, we can apply the Lemma to construct a quandle $Q \sqcup \Inn(Q)$.   

\begin{defi}\label{def_ideal}
A subset $I$ of a quandle $Q$ is an \emph{ideal} if it is stabilised by $\Inn(Q)$.
\end{defi}
In other words, an ideal is a subset stable under all the $q \qdn{\pm 1}(-)$. In particular, it is a subquandle. Notice that the inverse image of an ideal by a quandle morphism is an ideal.
\begin{ex}
In the above construction, the subset $X$ of $X \sqcup G$ is an ideal. In particular, $Q$ is an ideal of $Q \sqcup \Inn(Q)$. More generally, for any subset $Y$ of $X$ stable under the action of $G$ (that is, for any $G$-crossed subset of $X$) and any normal subgroup $H$ of $G$, $Y \sqcup H$ is an ideal of $X \sqcup G$.  
\end{ex}

Let $I$ be an ideal of $Q$. Let us consider the map $f_I: Q \rightarrow Q \sqcup G$ (where $G = \Inn(Q)$) defined by:
\[f_I(q) = \begin{cases}
q \in Q         &\text{if } q \in I; \\
\sigma(q) \in G &\text{else.}
\end{cases}\]

\begin{lem}\label{morphism_to_Q_sqcup_G}
The above map is a quandle morphism.
\end{lem}

\begin{proof}
Again, this is a matter of easy verification. Namely, that $f_I(x \qd y) = f_I(x) \qd f_I(y)$ is clear when $x$ and $y$ are both in $I$ or both not in $I$. Now, on the one hand, if $x \in I$ and $y \notin I$, then $x \qd y \notin I$ (else $y = x \qdn{-1} (x \qd y)$ would belong to $I$), hence:
\[f_I(x \qd y) = \sigma(x \qd y) = \sigma(x) \sigma(y) \sigma(x)^{-1} =  x \qd \sigma(y) = f_I(x) \qd f_I(y),\] by definition of $Q \sqcup G$. On the other hand, if $x \notin I$ and $y \in I$, then  $x \qd y \in I$, and $f_I(x \qd y) = x \qd y = \sigma_x(y) = f_I(x) \qd f(y)$. Whence the result.
\end{proof}

\begin{rmq}
The morphism $f_I$ takes values in the ideal $I \sqcup G$ of $Q \sqcup G$.
\end{rmq}

\subsection{Proof of the criterion}\label{sec_proof_criterion}

The proof of Theorem~\ref{profiniteness_criterion} will be a reduction to the following particular case:

\begin{prop}\label{profiniteness_criterion_for_finite_Inn}
Let $Q$ be a topological quandle which is a profinite space, and such that $\Inn(Q)$ is \emph{finite}. Then $Q$ is residually finite (hence profinite by Lemma~\ref{qdl_prof_iff_cpct_and_res_finite}).
\end{prop}

\begin{proof}
Notice that $\Homeo(Q)$ is totally disconnected (and in particular Hausdorff), so $G = \Inn(Q)$ must be discrete and closed in $\Homeo(X)$ (in particular, $\overline\Inn(Q) = \Inn(Q)$ in this case). 

Let $x, y \in Q$ such that $x \neq y$. We want to construct a finite discrete quotient of $Q$ where $x$ and $y$ are sent to distinct elements. The simplest case is the case where $Gx \neq Gy$. Then the quotient $Q/G$ is a trivial quandle (which is automatically a topological quandle, even without using Lemma~\ref{quotient_of_top_algebras}), and it is also a profinite space (Proposition~\ref{Quotient_of_Stone_space}). Thus, we can take a partition $Q/G = U \sqcup V$ of $Q/G$ into two clopen sets such that $Gx \in U$ and $Gy \in V$, and consider the equivalence relation whose classes are $U$ and $V$. Since $Q$ is a trivial quandle, it is a congruence, and the quotient $R$ is the discrete set with two elements, seen as a trivial quandles. Since $x$ and $y$ are sent to the two different elements of $R$, we have reached our goal in this case.

Let us now suppose that $x$ and $y$ are in the same orbit $O$ modulo the finite group $G = \Inn(Q)$. 
\begin{lem}
There exist a family $(V_s)_{s \in O}$ of subsets of $Q$ satisfying: 
\begin{itemize}[itemsep = -3pt, topsep= 3pt]
\item for every $s \in O$, $V_s$ is a clopen neighbourhood of $s$;
\item the $V_s$ are disjoint from each other;
\item $\forall q \in V_s,\ q \qd(-) = s \qd(-)$;
\item $\forall s \in O,\ \forall g \in G,\ g \cdot V_s = V_{gs}$. 
\end{itemize}
\end{lem}

\begin{proof}[Proof of the Lemma]
The canonical map $\sigma : Q \rightarrow \Inn(Q) = G$ is continuous, and $G$ is discrete, so the classes $\sigma^{-1}(\sigma(q))$ of elements acting in the same way as some fixed $q \in Q$ are clopen. Since $Q$ is a profinite space, we can choose, for every pair $(s,t) \in O^2$ with $s \neq t$, a clopen set $U_{s,t}$ containing $s$ and not $t$. Then the finite intersection $U_s := \sigma^{-1}(\sigma(s)) \cap \bigcap_{t \in O - \{s\}} U_{s,t}$ is a clopen set whose intersection with $O$ is only $\{s\}$. Now, let: 
\[V_s := \{q \in Q\ |\ \forall g \in G,\ \forall t \in O,\ g q \in U_t \Leftrightarrow t = gs \}.\]
In other words, $V_s$ is the intersection of the $g^{-1}U_{gs}$ with the  $(g^{-1}U_t)^c$ for $g \in G$ and $t \in O - \{gs\}$ (in particular, $V_s \subseteq U_s$). Since $G$ and $O = Gx$ are both finite, this is a finite intersection of clopen sets, so it is a clopen set. By construction, $V_s \cap O = \{s\}$, $V_s \cap V_t = \varnothing$ for $s \neq t$, and if $g \in G$, then $gV_s = V_{gs}$. Moreover, from $V_s \subseteq U_s \subseteq \sigma(\sigma^{-1}(s))$, we deduce that for all $q \in V_s$, we have $q \qd (-) = s \qd (-)$.
\end{proof}

We are going to use this Lemma to define a continuous morphism from $Q$ to a finite discrete quandle sending $x$ and $y$ to distinct elements. Precisely, notice first that Since $V := \bigsqcup V_s$ is stable under the action of $G = \Inn(Q)$ (that is, $V$ is an \emph{ideal} of $Q$), it is a subquandle of $Q$. Moreover, sending all of $V_s$ to $s$ (for each $s \in O$) defines a morphism $\pi$ from $V = \bigsqcup V_s$ to the finite quandle $O$, which is clearly a continuous surjection. Indeed, using $s \qd(-) \in G$, we see that $V_{s \qd t} = s \qd V_t$. As a consequence, $V_{s \qd t} = V_s \qd V_t$, and $\pi$ is a morphism, as claimed.

We now want to construct a finite quotient of $Q$ from this finite quotient of $V$. We deal with the complement $V^c$ of $V$ by using the construction from~\cref{par_crossed_sets}. Precisely, $V$ is an ideal of $Q$, so we can construct the morphism $f_V: Q \rightarrow V \sqcup G$ from Lemma~\ref{morphism_to_Q_sqcup_G}. Notice that $V \sqcup G$ is canonically a topological quandle, 
and that $f_V$ is continuous, being continuous on the two open sets $V$ and $V^c$.
Now, consider the quotient map $\pi: V \twoheadrightarrow O$. We have seen that it is continuous and $G$-equivariant with respect to the canonical action of $G$.
We have also seen that for $q \in V$, $\sigma(q) = q \qd (-)$ depends only on the image of $q$ in $O$; in other words, the canonical map $\sigma : V \rightarrow G$ factorizes through $\pi$ (into its restriction to $O$). All this means that $(id_G, \pi)$ defines a surjection of (topological) crossed sets from $(G, V, \sigma)$ to $(G, O, \sigma)$, which induces a (continuous) quandle epimorphism $\pi \sqcup id : V \sqcup G \twoheadrightarrow O \sqcup G$. 

Finally, $(\pi \sqcup id) \circ f_V$ is a continuous quandle morphism from $Q$ to the finite quandle $O \sqcup G$, sending $O$ bijectively onto itself, whence sending $x$ and $y$ to distinct elements. This finishes the proof of the residual finiteness of $Q$, and of its profiniteness. 
\end{proof}

\begin{proof}[Proof of Theorem~\ref{profiniteness_criterion}]
One implication follows directly from Corollary~\ref{qdl_prof_cpct_and_discontinuous} and Proposition~\ref{Inn(prof)_is_prof}. Let us show the converse. Let $Q$ be a topological quandle. Suppose that it is a profinite space and that $\overline\Inn(Q)$ is compact. Since $Q$ is compact, we want to show that it is residually finite, and then it will be profinite (Lemma~\ref{qdl_prof_iff_cpct_and_res_finite}). In other words, given $x,y \in Q$ such that $x \neq y$, we want to find a (discrete) finite quotient $R$ of $Q$ where the classes of $x$ and $y$ are different. 

We first look for an intermediate quotient whose inner automorphism group is finite. Let us first remark that $\mathcal I := \overline\Inn(Q)$ is totally disconnected, since $\Homeo(X)$ is (Lemma~\ref{Homeo(X)_tot_disc}). Since it is assumed compact, it is a profinite group.  Recall that open normal subgroups form a basis of neighbourhood of $1$ in profinite groups (Proposition~\ref{prop_Bourbaki}). Hence, by continuity of the action of $\mathcal I \subseteq \Homeo(Q)$ on $Q$, there exists an open subgroup $N$ of $\mathcal I$ such that $Nx \subseteq Q - \{y\}$ (notice that $Q - \{y\}$ is open since $Q$ is Hausdorff). This means that $x$ and $y$ are not in the same orbit modulo $N$. In other words, their classes in $Q/N$ are distinct. 

Let us examine the structure of $Q/N$. Notice first that it is a well-defined quandle, since $N$ is normalized by $\Inn(Q) \subseteq \mathcal I$ (Lemma~\ref{quotient_by_group}). Since $Q$ is compact, it is in fact a well-defined topological quandle, for the quotient topology (Lemma~\ref{quotient_of_top_algebras}).
By Proposition~\ref{Quotient_of_Stone_space}, it is a profinite space. Moreover, the quotient map $p : Q \twoheadrightarrow Q/N$ induces a surjection $p_* : \Inn(Q) \twoheadrightarrow \Inn(Q/N)$ (Proposition~\ref{Inn(surj)}), whose kernel contains $N \cap \Inn(Q)$, so  $\Inn(Q/N)$ is a quotient of the subgroup $\Inn(Q)/(N \cap \Inn(Q))$ of $\mathcal I/N$, which is finite. 

Thus, we can apply Proposition~\ref{profiniteness_criterion_for_finite_Inn} to $Q/N$, and conclude that it is residually finite. The images of $x$ and $y$ are distinct in $Q/N$, hence they are still distinct in some finite discrete quotient $R$ of $Q/N$, which is also a finite discrete quotient of $Q$, and the Theorem is proved.
\end{proof}

\subsection{About profinite crossed spaces}

The alert reader may have noticed that the previous proof is really about crossed sets, or about \emph{crossed spaces}, as we shall call them in a topological context. 

\begin{defi}
Let $G$ be a topological group. A \emph{crossed $G$-space} is a crossed set $(X, \partial)$ endowed with a topology making the action $G \times X \rightarrow X$ and the map $\delta : X \rightarrow G$. A \emph{crossed space} $(G,X, \partial)$ is the data of a topological group $G$ and a crossed $G$-space $(X, \partial)$. Crossed spaces form a category $\mathcal TX\mathcal Set$, whose morphisms are the obvious ones.
\end{defi}

The construction of $f_I$ can be adapted to crossed sets. Namely, given a crossed set $T = (G,X, \partial)$, if $I$ is a subset of $X$ which is stable under the action of $G$, then we have a crossed set $T_I := (G, I \sqcup G, \partial')$, where $G$ acts on $I \sqcup G$ in the obvious way, and $\partial'$ sends $x \in I$ to $\partial(x)$ and $g \in G$ to itself. This crossed set is endowed with a morphism $(id_G, f_I): T \rightarrow T_I$, where $f_I : X \rightarrow I \sqcup G$ is defined by $f_I(x) = x$ if $x \in I$, and $f_I(x) = \partial(x)$ if $x \notin I$. The reader can check that this morphism is well-defined. 

Crossed sets are algebras on a multi-sorted Lawvere theory, with two sorts. Let us explain briefly what this means. Such a theory $\mathbb T$ consists, for every $(m,n)$, of two sets $\mathbb T_{m,n}^1$ and $\mathbb T_{m,n}^2$ of $(m,n)$-ary operations, and of a set of equations between them. Then a $\mathbb T$-algebra $X$ is a couple $(X_1, X_2)$ of sets, endowed with a map $\omega_X : X_1^m \times X_2^n \rightarrow X_i$ for each $\omega \in \mathbb T_{m,n}^i$, such that the equations are satisfied. It is clear that crossed sets can be described in this way, with a finite set of generating operations and relations saying that an algebra $(X_1, X_2)$ should be a group $X_1$ with an action $X_1 \times X_2 \rightarrow X_2$ and a map $X_2 \rightarrow X_1$ having the right properties.

All the things that we have said for algebras over a Lawvere theory extend easily to multicoloured Lawvere theories. In particular, the notion of profiniteness of a topological $\mathbb T$-algebra extends easily, and a topological $\mathbb T$-algebra is profinite if and only if it is residually finite and all it components are compact and totally disconnected (that is, they are profinite spaces). 

With these tool, we can adapt the proof of Theorem~\ref{profiniteness_criterion} to show:

\begin{theo}\label{profiniteness_criterion_XSet}
A crossed set $(G,X)$ is profinite if and only if $G$ and $X$ are profinite spaces.
\end{theo}

\begin{proof}[Sketch of proof]
Given the above observations, we only need to show that $(G,X)$ is residually finite if $G$ and $X$ are profinite spaces. On the one hand, if $g, h \in G$ satisfy $h \neq g$, then, since $G$ is profinite, the images of $g$ and $h$ are distinct in a finite quotient $F$ of $G$, and $(F, *, 1)$ is a finite quotient of $(G,X, \partial)$ where the images of $g$ and $h$ are distinct. On the other hand, if $x,y \in X$ satisfy $x \neq y$, then, as in the proof of Theorem~\ref{profiniteness_criterion}. (with $G$ instead of $\mathcal I$), we find an open subgroup $N$ of $G$ such that their classes are distinct in $(G/N, X/N, \overline \partial)$, so we reduce to the case where $G$ is finite discrete. This case is treated exactly as in the quandle case, using the intermediate quotient $(1, X/G, 1)$ when $Gx \neq Gy$, and using the intermediate quotient $(G, V \sqcup G, \partial')$ given by the adaptation of the construction $f_V$ detailed above, in the case where $Gx = Gy$.
\end{proof}

The criterion for quandles (Theorem~\ref{profiniteness_criterion}), or at least the implication shown in this section, can actually be deduced from the one for (idempotent) crossed sets (Theorem~\ref{profiniteness_criterion_XSet}), using the usual construction of a rack from a crossed set, defined by:
\[\mathcal U(G,X, \partial) := (X, \qd) \text{ where for all }x,y \in X,\ x \qd y = \partial(x) \cdot y.\]
The law $\qd$ is a quandle law if and only if the crossed set is \emph{idempotent}, \emph{i.e.}~if $\partial(x) \cdot x = x$ for all $x \in X$. This is easily seen to behave well in the topological setting, defining a functor $\mathcal U$ from the category of crossed spaces to the category of topological racks, restricting to a functor from idempotent crossed spaces to topological quandles.

Projective limits of crossed spaces are computed component-wise, in the sense that:
\[\varprojlim (G_i, X_i,\partial_i) = (\varprojlim G_i, \varprojlim X_i, \varprojlim \partial_i).\]
This implies that $\mathcal U$ commutes with projective limits. Since it also sends finite discrete crossed sets to finite discrete racks, it sends profinite crossed sets to profinite racks.

\begin{proof}[Deduction of Theorem~\ref{profiniteness_criterion} from Theorem~\ref{profiniteness_criterion_XSet}]
If $Q$ is a quandle which is profinite as a space and such that $\mathcal I = \overline\Inn(Q)$ is compact (hence profinite by Lemma~\ref{Homeo(X)_tot_disc} and Corollary~\ref{profiniteness_criterion_groups}), then we have $Q = \mathcal U(\mathcal I, Q, \sigma)$. Notice that the canonical action $\mathcal I \times Q \rightarrow Q$ is continuous by Proposition~\ref{prop_adj_direct}, since $Q$ is compact and it is adjoint to the inclusion of $\mathcal I$ into $\mathcal C(Q,Q)$. Then Theorem~\ref{profiniteness_criterion_XSet} says that $(\mathcal I, Q, \sigma)$ is profinite as a crossed space, so $Q$ is a profinite quandle.
\end{proof}

\begin{rmq}
All this works verbatim without the idempotence hypothesis, showing that our criterion is in fact valid for \emph{racks}. 
\end{rmq}

\begin{rmq}\label{rk_loc_cpct_XSet}
The set-theoretic version of $\mathcal U$ has a left adjoint, given by $\mathcal F : (Q,\qd) \mapsto (G(Q), Q, \eta)$ (see for instance~\cite[\S~4.2]{Darne7}). This can be generalized to the topological setting, with some care. Namely, for $\mathcal F$ to be well-defined, we need the canonical action $G(Q) \times Q \rightarrow Q$ to be continuous, when the universal property of our topological enveloping group $G(Q)$ only gives a continuous map $G(Q) \rightarrow \mathcal C(Q,Q)$. By Proposition~\ref{prop_adj_converse}, this works for locally compact topological quandles (see also Fact~\ref{continuous_action_of_G(Q)}), so we get an adjunction between the category of locally compact topological quandles and the category of crossed spaces $(G,X, \partial)$ whose $X$ is locally compact. An extension of this adjunction to compactly generated quandles is discussed in~\cref{Alg_cartesian_closed}.
\end{rmq}

\appendix

\section{Topological spaces}\label{sec_topo_spaces}

In this appendix, we first recall some classical facts about locally compact topological spaces which play an important role throughout the paper. We then discuss the wider compactly generated setting, which is not relevant for the study of profinite spaces, (which are compact Hausdorff), but could be relevant to further work on topological quandles.

\subsection{Locally compact topological spaces}\label{par_lc_topo_spaces}

A lot of the algebraic structures appearing in our work involve maps of the form $X \times Y \rightarrow Z$, which should be continuous in the topological setting. The continuity of such maps is much easier to deal with when working with locally compact spaces, thanks to the usual adjunction that we now recall.

If $X$ and $Y$ are topological spaces, we endow the set $\mathcal C(X,Y)$ of continuous functions from $X$ to $Y$ with the compact-open topology, generated by the $\mathcal O_{K,U} = \{f \in \mathcal C(X,Y)\ |\ f(K) \subseteq U\}$, for $K$ compact in $X$ and $U$ open in $Y$. Recall the classical:

\begin{prop}\label{prop_adj_direct}
Let $X$, $Y$ and $Z$ be topological spaces. If $\varphi : X \times Y \rightarrow Z$ is continuous, then its \emph{adjoint} $\widetilde \varphi : x \mapsto \varphi(x,-)$ is a continuous map from $X$ to $\mathcal C(Y,Z)$.
\end{prop}

\begin{proof}
If $K$ is compact in $Y$ and $U$ is open in $Z$, we need to show that $\widetilde \varphi^{-1}(\mathcal O_{K,U})$ is open. Let $x \in \varphi^{-1}(\mathcal O_{K,U})$. For each $y \in K$, we have $\varphi(x,y) \in U$, so there is an open neighbourhood $V_y \times W_y$ of $(x,y)$ such that $\varphi(V_y \times W_y) \subseteq U$. By compacity of $K$, we can choose $y_1, ..., y_N \in K$ such that the $W_{y_i}$ cover $K$. Then $V := V_{y_1} \cap \cdots \cap V_{y_N}$ is an open neighbourhood of $x$, and if $W :=  W_{y_1} \cup \cdots \cup W_{y_N}$, we have $\varphi(V \times K) \subseteq \varphi(V \times W) \subseteq U$, since every element of $V \times W$ belongs to some $V_{y_i} \times W_{y_i}$. As a consequence, $V \subseteq \widetilde \varphi^{-1}(\mathcal O_{K,U})$, which is thus a neighbourhood of its points, hence open.
\end{proof}

It is well-known that the converse holds when $Y$ is locally compact. 

\begin{prop}\label{prop_adj_converse}
If $Y$ is locally compact, then the evaluation map $\ev_{Y,Z} :\mathcal C(Y,Z) \times Y \rightarrow Z$ is continuous, hence the converse of \ref{prop_adj_direct} holds.
\end{prop}

\begin{proof}
Let $U$ be open in $Z$, and let $(f,y) \in \ev_{Y,Z}^{-1}(U)$, which means that $f(y) \in U$. Since $f$ is continuous and $Y$ locally compact, there exists a compact neighbourhood $K$ of $y$ such that $f(K) \subseteq U$. Then $\mathcal O_{K,U} \times K$ is a neighbourhood of $(f,y)$ contained in $\ev_{Y,Z}^{-1}(U)$. This proved the continuity of $\ev_{Y,Z}$, and the rest follows from the equality $\varphi = \ev_{Y,Z} \circ (\widetilde \varphi \times id_Y)$.
\end{proof}

\begin{rmq}
For $y \in Y$, the map $\ev_y : \mathcal C(Y,Z) \rightarrow Z$ sending $f$ to $f(y)$ is always continuous, for any space $Y$. Indeed, $\ev_y^{-1}(U) = \mathcal O_{y,U}$ is open, for every open subset $U$ of $Z$.
\end{rmq}

\begin{rmq}
Let us mention that spaces for which this converse holds are called~\emph{exponentiable}. It is known that a topological space is exponentiable if and only if it is~\emph{core-compact}; this is a theorem of Day and Kelly \cite{Day-Kelly} (see also \cite{Isbell, Escardo} for a modern account). In particular, locally compact spaces are core-compact. Moreover, a Hausdorff space is core-compact if and only if it is locally compact.  
\end{rmq}

Recall that a continuous surjection $p : Y \twoheadrightarrow Z$ is called \emph{a quotient} is the induced map from $Y/\mathcal R_p$ to $Z$ is a homeomorphism (where $\mathcal R_p$ is the equivalence relation on $Y$ given by $y\mathcal R_p y'$ if and only if $p(y) = p(y')$).
\begin{prop}\label{prod_of_quotient_by_loc_cpct}
If $X$ is locally compact and $p : Y \twoheadrightarrow Z$ is a quotient, then the map $id_X \times p : X \times Y \twoheadrightarrow X \times Z$ is too.
\end{prop}

\begin{proof}
Let us show that $id \times p$ satisfies the universal property of the quotient of $X \times Y$ by $\mathcal R_{id \times p}$. Let $T$ be a topological space, and let $f : X \times Y \rightarrow T$ be a continuous map. Since $X$ is locally compact, $f$ factors uniquely through a map $\overline f : X \times Z \rightarrow T$ if and only if its adjoint $\widetilde f : Y \rightarrow \mathcal C(X, T)$ factors uniquely through a map $g: Z \rightarrow \mathcal C(X, T)$ (then $\overline f$ is the adjoint of $g$). But one easily checks that $\mathcal R_f$ contains $\mathcal R_{id \times p}$ if and only if $\widetilde f$ factors through $Y/\mathcal R_p \cong Z$, giving such a $g$, whence the conclusion.
\end{proof}

\begin{rmq}
In fact, the preservation of quotients by $X \times (-)$ is another characterisation of exponentiable spaces~\cite{Isbell}.
\end{rmq}

We are mostly interested in:

\begin{cor}\label{prod_of_quotients}
Let $X$ and $Y$ be locally compact spaces. Let $\mathcal R$ (resp.~$\mathcal S$) be an equivalence relation on $X$ (resp.~on $Y$). If $Y$ and $X/\mathcal R$ are locally compact (the latter being true for instance if $X$ is compact or if $X/\mathcal R$ is discrete), then the surjection $X \times Y \twoheadrightarrow X/\mathcal R \times Y/\mathcal S$ is a quotient.
\end{cor}

\begin{proof}
Notice that if $X$ is compact, then $X/\mathcal R$ is too, so it is locally compact. The result then follows from Proposition~\ref{prod_of_quotient_by_loc_cpct} and from the fact that the composition of quotients is a quotient, by writing $p \times q = (p \times id_Y) \circ (id_{X/\mathcal R} \times q)$, where $p : X \twoheadrightarrow X/\mathcal R$ and $q : Y \twoheadrightarrow Y/\mathcal S$ are the canonical projection.
\end{proof}

\subsection{Working in a convenient category of topological spaces}

Not every topological space is exponentiable: the converse of Proposition~\ref{prop_adj_direct} does not hold in general. In particular, if we want to show that a an action $G \times X \rightarrow X$ of a topological group $G$ on a space $X$ is continuous, it is not enough to show that the corresponding map $G \rightarrow \Homeo(X)$ is continuous. This is a problem for actions defined using some universal property of $G$, such as the action of its enveloping group on a quandle. There are two main ways of fixing this:
\begin{itemize}[itemsep = -3pt, topsep= 3pt]
\item Focus on locally compact spaces~;
\item Work in a convenient category of topological spaces, such as the category of compactly generated spaces.
\end{itemize}
Since our focus is on profinite spaces, we can even assume that all our spaces are compact (even compact Hausdorff), so we have chosen the first point of view. However, some tools developed in the paper pertain to the general theory of topological quandles, so we briefly discuss the second approach, which would certainly be preferable for developing further the theory of topological quandles.

\begin{rmq}
The subtleties discussed at the end of~\cref{sec_congruences} are partly solved by this change of perspective: in the category of compactly generated spaces, if $X$ is an internal $\mathbb T$-algebra and $\mathcal R$ is a congruence on it, $(X/\mathcal R)^n$ is always a quotient of $X^n$ (see for instance \cite[Prop.~2.1 and Prop.~2.20]{Strickland}), so a quotient of an internal $\mathbb T$-algebra is always an internal $\mathbb T$-algebra.
\end{rmq}

\subsubsection{Algebras in cartesian-closed categories}\label{Alg_cartesian_closed}

Some of the main constructions of quandle theory are instances of the construction of adjoints to algebraic functors. These constructions work well in the general context of \emph{cartesian-closed categories}. The latter do not include the category of topological spaces, but they include convenient categories of topological spaces, such as the category of \emph{compactly generated spaces}.

If $\mathcal C$ is a category with finite products, and $\mathbb T$ is a Lawvere theory, then we can consider $\mathcal C$-valued $\mathbb T$-algebras, which are also called \emph{internal $\mathbb T$-algebras} in $\mathcal C$. Then, if $\varphi : \mathbb S \rightarrow \mathbb T$ is a morphism of Lawvere theories (in the obvious sense), we get an induced functor $\varphi^*$ from internal $\mathbb T$-algebras to internal $\mathbb S$-algebras, endowing a $\mathbb T$-algebra $X$ with the $\mathbb S$-algebra structure given by $\omega_X = \varphi(\omega)_X$, for $\omega \in \mathbb S$. A functor of the form $\varphi^*$ is called an \emph{algebraic functor}.

If $\mathbb T$ is interpreted as a category in the usual way, internal $\mathbb T$-algebras in $\mathcal C$ identify with product-preserving functors from $\mathbb T$ to $\mathcal C$. Then a morphism of Lawvere theories $\varphi : \mathbb S \rightarrow \mathbb T$ is just a product-preserving functor, and $\varphi^*$ identifies with precomposition by $\varphi$.

Algebraic functors often have left adjoints. For instance, we have the:
\begin{theo}
If $\mathcal C$ is cartesian-closed and cocomplete, then for every morphism $\varphi : \mathbb S \rightarrow \mathbb T$ of Lawvere theories, $\varphi^*$ has a left adjoint, given by left Kan extension along $\varphi$.
\end{theo}

\begin{proof}[Sketch of proof]
This classical result in Lawvere theory can be deduced for instance from the general machinery of~\cite{Lack-Rosicky}. We sketch a rather direct argument from~\cite{Menni-Tholen}.
It is enough to show that left Kan extension along the product-preserving functor $\varphi$ sends product preserving functors to product preserving functors. Writing $\mathrm{Lan}_\varphi(x)$ as a colimit over the comma category $\varphi/x$, and using that $X \times (-)$, being a left adjoint, preserves colimits in $\mathcal C$, we see that it boils down to showing that the canonical functor $(\varphi/x) \times (\varphi/y) \rightarrow \varphi/(x \times y)$ is \emph{final} for every $x,y \in \mathbb T$. For a short explicit proof of this latter fact, see~\cite[Lem.~3.4]{Menni-Tholen}.
\end{proof}

\begin{ex}
Taking the associated conjugation quandle gives a functor $c$ from the category of internal groups in $\mathcal C$ to the category of internal quandles in $\mathcal C$, and this is an algebraic functor. It thus has a left adjoint $G(-)$, whenever $\mathcal C$ is cartesian-closed and cocomplete.
\end{ex}

\begin{ex}\label{generalized_pi0}
The functor $\mathrm{triv}$ from $\mathcal C$ to internal quandles which endows an object of $\mathcal C$ with the trivial (internal) quandle structure is algebraic, induced by the morphism from the theory of quandle to the trivial Lawvere theory sending $\qd$ to the projection on the second factor. It thus has a left adjoint $\pi_0(-)$, whenever $\mathcal C$ is cartesian-closed and cocomplete.
\end{ex}

\begin{rmq}\label{G(Q)^ab_generalized}
The inclusion $\iota$ of internal abelian groups into internal groups and the forgetful functor $\omega$ from internal abelian groups to $\mathcal C$ are both algebraic functors, whose left adjoints are denoted, respectively, by $(-)^{\mathrm{ab}}$ and $\Z[-]$. Taking the left adjoint of the relation $\mathrm{triv} \circ \omega = c \circ \iota$ (true on set-theoretic algebra, hence at the level of morphisms of Lawvere theories) gives a natural isomorphism $G(-)^{\mathrm{ab}} \cong \Z[\pi_0(-)]$.
\end{rmq}

In this context, we can also speak of the adjoint action of $G(Q)$ on $Q$, for an internal quandle $Q$, which is obtained by internalizing the usual construction. Namely, for an internal quandle $Q$, the adjoint of $\qd : Q^2 \rightarrow Q$ gives a map $\sigma : Q \rightarrow [Q,Q]$. Composition makes $[Q,Q]$ into an internal monoid, and $\sigma$ factors through its group $G$ of invertible elements, defined via the usual pullback construction (generalizing taking the first projection of the set of pairs $(x,y)$ such that $xy = yx = 1$). The corresponding map $Q \rightarrow G$ is an internal quandle morphism to $c(G)$, so its extends uniquely to an internal group morphism $G(Q) \rightarrow G \subseteq [Q,Q]$. The adjoint of the latter map gives an internal group action $G(Q) \times Q \rightarrow Q$.

\begin{rmq}
In particular, this extends Fact~\ref{continuous_action_of_G(Q)} to the compactly generated setting.
\end{rmq}

\begin{rmq}
If there is a good notion of the image of an internal group morphism in the category $\mathcal C$, then the group of inner automorphisms of $Q$ can be defined as the image of the above morphism $G(Q) \rightarrow G$. 
\end{rmq}

A closely related construction is the classical adjunction between racks (resp.~quandles) and crossed sets (resp.~idempotent crossed sets), which also extends to this general setting. In fact, the above construction of left adjoints to algebraic functors also works for multi-sorted Lawvere theories, so it applies to the forgetful functor $\mathcal U$ from internal (idempotent) crossed-sets in $\mathcal C$ to internal quandles. This shows that it has a left adjoint $\mathcal F$. With a bit more work, one can show that $\mathcal F(Q) = (G(Q), Q, \eta)$, where the (internal) action of $G(Q)$ on $Q$ is the one described above.

\begin{ex}
The category $\mathcal A(\mathcal C)$ of internal group actions (that is, of pairs $(G,X)$, where $G$ is an internal group, $X$ is an object of $\mathcal C$, with a given action of $G$ on $X$) is also a category of algebras on a $2$-sorted Lawvere theory. On the one hand, there is an algebraic functor from $\mathcal A(\mathcal C)$ to $\mathcal C$ sending $(G,X)$ to $G \times X$, endowed with the quandle law described by the formula $(g,x) \qd (h,y) = (ghg^{-1}, g \cdot y)$. One can show that its left adjoint is given by sending a quandle to its adjoint action. On the other hand, the functor $\mathcal C \rightarrow \mathcal A(\mathcal C)$ endowing an object with the trivial action of the trivial group is clearly algebraic, so it has a left adjoint. For an object $(G,X)$ of $\mathcal A(\mathcal C)$, we can \emph{define} $X/G$ as its image by this adjoint. Combining this with the adjunction from Example~\ref{generalized_pi0}, a reasoning similar to the one of Remark~\ref{G(Q)^ab_generalized} gives, for any internal quandle $Q$, a canonical isomorphism $\pi_0(Q) \cong Q/G(Q)$, natural in $Q$.
\end{ex}

\subsubsection{Topological quandles and topological peripheral systems}\label{sec_top_Pph}

We now examine how the set-theoretic constructions of~\cite[\S~4.2]{Darne7} fare in a topological context. These include the link between crossed sets and peripheral systems, which is somewhat different in nature than the previous adjunctions, since peripheral systems are not algebras over some Lawvere theory. We thus restrict our attention to a nice category of topological spaces, such as the category $\mathcal T$ of compactly generated spaces. In particular, $\mathcal T$ is cartesian-closed, so the constructions from \cref{Alg_cartesian_closed} apply.

In this topological context, we get more explicit descriptions of certain constructions. For instance, if $Q$ is an internal quandle in $\mathcal T$, then its enveloping group $G(Q)$ is obtained by endowing the set-theoretic $G(Q)$ by the topology coming from its presentation as a quotient of the free monoid on $Q \cup Q^{-1}$ (seen as a coproduct of finite products of copies of $Q$ and $Q^{-1}$, where $Q^{-1}$ is a copy of $Q$, in the category $CG$).

\begin{rmq}
This construction of the enveloping group is not available in the category of all topological spaces, since a quotient of a topological monoid is not necessarily a topological monoid. However, in $\mathcal T$, this does define an internal group $G(Q)$ having the required universal property.
\end{rmq}

The following definitions are obvious lifts of their set-theoretic counterparts to the topological context (here we call $\mathcal T$-space an object of $\mathcal T$, and $\mathcal T$-structure an internal structure in $\mathcal T$):
\begin{defi}\label{def_Ob(TPph)}
A \emph{$\mathcal T$-peripheral system} $\left(G, (H_i, z_i)_{i \in I}\right)$ is the data of a $\mathcal T$-group $G$, a $\mathcal T$-space $I$, a continuous map $i \mapsto z_i$ from $I$ to $G$ and, for each $i \in I$, a subgroup $H_i$ of $G$ containing $z_i$ in its center.
\end{defi}

\begin{defi}\label{def_Mor(TPph)}
A \emph{morphism} $(\varphi, u, [g]) : \left(G, (H_i, z_i)_{i \in I}\right) \rightarrow \left(G', (H_j', z_j')_{j \in J}\right)$ of $\mathcal T$-peripheral system is the data of a continuous group morphism $\varphi : G \rightarrow G'$, a continuous map $u: I \rightarrow J$, and the class $[g]$ of a continuous map $g: I \rightarrow G'$, under the following equivalence relation: two continuous maps $g,k : I \rightarrow G'$ are equivalent if $k_i^{-1}g_i \in H_{u(i)}'$ for every $i \in I$. This data has to satisfy conditions~\eqref{condition_morphisms_Pph}.

Composition of morphisms is still given by formula~\eqref{composition_morphisms_Pph}. This defines a category $\mathcal{TP}ph$ of $\mathcal T$-peripheral systems.
\end{defi}

Given a $\mathcal T$-peripheral system $P = \left(G, (H_i,z_i)_{i \in I} \right)$, we can construct a $\mathcal T$-crossed set (that is a crossed set internal to $\mathcal T$) $\mathcal X(P) = (G, X, \partial)$ by taking $X = G \times I/{\sim}$, where the equivalence relation is defined by:
\[(gh,i) \sim (g,i) \text{ if } h \in H_i.\]
The action of $g \in G$ on $X$ is given by the formula $g \cdot (g',i) = (gg',i)$, and the map $\partial$ is given by $\partial(g,i) = g z_i g^{-1}$. Crucially, this construction relies on the fact that a quotient of a compactly generated space is again compactly generated, and that $G \times (-)$ preserves such quotients (see for instance \cite[Prop.~2.1 and 2.20]{Strickland}), so that  $G \times X$ is a quotient of $G \times (G \times I)$, hence the above formula defines a continuous group action $G \times X \rightarrow X$.

The underlying set-theoretic construction is exactly the functor $\mathcal X$ from~\cite[\S~4.2]{Darne7}. Notice that, as in the set-theoretic version, the functor $\mathcal X$ produces $\mathcal T$-crossed sets having a continuous section, since $i \mapsto (1,i)$ is a continuous section of the projection $G \times I/{\sim} \twoheadrightarrow I$ given by quotienting by $G$. In the other way around, the construction $\mathcal S$ from~\cite[\S~4.2]{Darne7} is easily adapted to $\mathcal T$-crossed sets having a continuous section. 

The constructions $\mathcal X$ and $\mathcal S$ are functorial, as is easily deduced from the fact that the underlying set-theoretic construction are. Moreover, $\mathcal X$ is left adjoint to $\mathcal S$, since the set-theoretic unit and the counit morphisms are easily seen to be continuous. Furthermore, the unit is an isomorphism $\mathrm{Id} \cong \mathcal S \circ \mathcal X$ since, given a $\mathcal T$-peripheral system as above, each $H_i$ is clearly the stabilizer of $(1,i)$ in the action of $G$ on $G \times I/{\sim}$. However, the counit is not an isomorphism: given a $\mathcal T$-crossed set with a section $(G,X,\delta,\sigma)$, the canonical map $(G \times (X/G))/{\sim} \rightarrow X$ induced by the section $\sigma : X/G \rightarrow X$ is only a continuous bijection is general, not a homeomorphism.

Hence, even if the rest of the constructions from~\cite[\S~4.2]{Darne7} can be easily adapted to a topological context (restricting to quandles endowed with a continuous section of $Q \twoheadrightarrow Q/G(Q)$, and to crossed sets $(G, X, \partial)$ where $G$ is generated by $\partial(X)$), we cannot in general recover a $\mathcal T$-quandle with a continuous section from the associated $\mathcal T$-peripheral system.

However, the construction of quandles from peripheral systems is still possible in certain variants of this situation, as we have seen in the profinite case in~\cref{sec_ProPph}. There, the relevant continuous bijection was a homeomorphism because of a classical compactness argument (see the proof of Proposition~\ref{G(-)_adj_to_c(-)}). In fact, the constructions of~\cref{sec_ProPph} could certainly be fitted in a more categorical framework similar to the one from~\cite[\S~4.2]{Darne7}, using profinite crossed set as an intermediate category, endowed with the relevant profinite completion functors.

\section{Actions by homeomorphisms}

For the convenience of the reader, we gather here some classical facts about actions of topological groups on profinite spaces.

\subsection{Groups of homeomorphisms}\label{Homeo(X)}

Let $X$ be a compact space, and let $\Homeo(X)$ be the set of its homeomorphisms, endowed with the usual compact-open topology, which is the topology generated by the open sets 
\[\mathcal O_{K, U} := \{f \in \Homeo(X)\ |\ f(K) \subseteq U\},\]
for $U$ open and $K$ compact in $X$. If $K = \{x\}$, we denote $\mathcal O_{K, U}$ by $\mathcal O_{x, U}$.

\begin{theo}[{\cite[Thm.~3]{Arens}}]\label{Arens}
If $X$ is compact Hausdorff, then $\Homeo(X)$ is a topological group, and the action of $\Homeo(X)$ on $X$ is continuous. 
\end{theo}

\begin{proof}
Let us denote by $G$ the group $\Homeo(X)$ endowed with the compact-open topology. We first want to show that it is a topological group. Let us show that $\circ : G^2 \rightarrow G$ is continuous. It suffices to show that for $K$ compact and $U$ open in $X$, the subset $\circ^{-1}(\mathcal O_{K,U})$ is open in $G$. Let $g, f \in G$ such that $g \circ f \in \mathcal O_{K,U}$, which means that $f(K) \subseteq g^{-1}(U)$. In other words, the compact $f(K)$ is disjoint from the complement $g^{-1}(U)^c$ of $g^{-1}(U)$. Since $X$ is compact Hausdorff, the closed subset $g^{-1}(U)^c$ is compact, and the disjoint compact subsets $f(K)$ and $g^{-1}(U)^c$ can be separated by open subsets: there exist disjoint open sets $V$ and $W$ such that  $f(K) \subseteq V$ and $g^{-1}(U)^c \subseteq W$. Since $W^c$ is closed, $f(K) \subseteq V \subseteq \overline V \subseteq W^c \subseteq g^{-1}(U)$. But the closure $\overline V$ is closed in the compact Hausdorff space $X$, so it is compact. Finally, $\mathcal O_{\overline V, U} \times \mathcal O_{K, V}$ is a neighbourhood of $(g, f)$ contained in $\circ^{-1}(\mathcal O_{K,U})$, which shows that the latter is open, and that the composition is continuous. 

As for the continuity of the inverse map $(-)^{-1} : G \rightarrow G$, it follows directly from $\mathcal O_{K,U}^{-1} = \mathcal O_{U^c, K^c}$, noticing that $U^c$ is closed in $X$ compact, so it is compact, and $K^c$ is open.

Finally, let us show that the action $\mu: (f,x) \mapsto f(x)$ from $\Homeo(X) \times X$ to $X$ is continuous. Let $U$ be an open set of $X$, and $(f,x)$ such that $\mu(f,x) \in U$. The set $f^{-1}(U)$ is an open neighbourhood of $x$, that has to contain a compact neighbourhood $K$ of $x$. Then $ \mathcal O_{K,U} \times K$ is a neighbourhood of $(f,x)$ contained in $\mu^{-1}(U)$, whence the conclusion.
\end{proof}

\begin{lem}\label{Homeo(X)_tot_disc}
If $X$ is totally disconnected, then so is $\Homeo(X)$. 
\end{lem}

\begin{proof}
Suppose that $X$ is totally disconnected. Let $f, g \in \Homeo(X)$ such that $f \neq g$. Then there is some $x \in X$ satisfying $f(x) \neq g(x)$. Let $X = U \sqcup V$ be a partition of $X$ into disjoint open sets such that $f(x) \in U$ and $g(x) \in V$. Then $\Homeo(X) = \mathcal O_{x, U} \sqcup \mathcal O_{x, V}$ is a decomposition of $\Homeo(X)$ into two disjoint open sets, and $f \in \mathcal O_{x, U}$, whereas $g \in \mathcal O_{x, V}$. 
\end{proof}

\begin{rmq}
The same argument shows that $\Homeo(X)$ (and in fact $\mathcal C(T,X)$, for any space $T$) is Hausdorff whenever $X$ is.
\end{rmq}

Let $\mathcal R \subseteq X^2$ be an equivalence relation on a \emph{compact Hausdorff} space $X$. Let $\Homeo(X, \mathcal R)$ be the subgroup of $\Homeo(X)$ whose elements are the homeomorphisms $f$ such that both $f \times f$ and its inverse preserve $\mathcal R$.

\begin{lem}\label{Homeo(X,R)_closed}
If $\mathcal R$ is closed in $X^2$, then $\Homeo(X, \mathcal R)$ is closed in $\Homeo(X)$.
\end{lem}

\begin{proof}
For $x, y \in X$, the map $\Phi_{x,y} : f \mapsto (f(x), f(y))$ is continuous from $\Homeo(X)$ to $X^2$. As a consequence, the result follows from $\Homeo(X, \mathcal R) = \bigcap_{(x,y) \in \mathcal R} \Phi_{x,y}^{-1}(\mathcal R)$.
\end{proof}

\begin{rmq}\label{rk_gen1}
Although we use it only when $X$ and $X/\mathcal R$ are both compact Hausdorff, Lemma~\ref{Homeo(X,R)_closed} holds in general, even when $\Homeo(X)$ and $\Homeo(X/\mathcal R)$ are not topological groups.
\end{rmq}

Let $p : X \twoheadrightarrow X/\mathcal R$ denote the quotient map. Let us \emph{assume that $X/\mathcal R$ is Hausdorff}. Notice that this implies that $\mathcal R$ is closed in $X^2$, as the inverse image by $p \times p$ of the diagonal of $X/\mathcal R$.
\begin{lem}\label{Homeo(X,R)_to_Homeo(X/R)}
Let $\mathcal R$ be an equivalence relation on a compact Hausdorff space $X$. If the quotient $X/\mathcal R$ is Hausdorff, then the canonical map $p_* :\Homeo(X, \mathcal R) \rightarrow \Homeo(X/\mathcal R)$ is continuous. 
\end{lem}

\begin{proof}
Let $K, U \subseteq X/\mathcal R$ with $K$ compact and $U$ open. Since $X/\mathcal R$ is Hausdorff, $K$ is closed, hence $p^{-1}(K)$ is closed in $X$ compact, so it is compact. Moreover, $p^{-1}(U)$ is open. Clearly, $p_*^{-1}\left(\mathcal O_{K,U}\right) =  \mathcal O_{p^{-1}(K), p^{-1}(U)}$, whence the result.
\end{proof}

\begin{rmq}\label{rk_gen2}
Lemma~\ref{Homeo(X,R)_to_Homeo(X/R)} also holds in a more general setting. Namely, it holds if $p$ is \emph{proper}, in the sense that the inverse image of a compact is compact. Then $p_*$ is a continuous morphism, even if not necessarily between topological groups (the compact-open topology is not compatible with the group structure in general).
\end{rmq}

The next statement deals with the case of a cofinite equivalence relation $\mathcal R$ on any topological space (recall that \emph{cofinite} means that the quotient $X/\mathcal R$ is finite discrete).
\begin{lem}\label{p*_continuous}
Let $p: X \twoheadrightarrow X/\mathcal R$ denote the quotient of a topological space $X$ by a cofinite congruence $\mathcal R$. Then the canonical map $p_* :\Homeo(X, \mathcal R) \rightarrow \Homeo(X/\mathcal R)$ is continuous. 
\end{lem}

\begin{rmq}
Not every surjection of topological spaces is the quotient by a congruence (see for instance~\cref{par_lc_topo_spaces}), but surjections onto discrete spaces are. In particular, the previous statement deals with any surjection onto a finite discrete space.
\end{rmq}

\begin{proof}[Proof of Lemma~\ref{p*_continuous}]
Since $X/\mathcal R$ is finite discrete, the compact-open topology on the finite group $\Homeo(X/\mathcal R)$ is the discrete topology, so we need to show that $p_*$ is locally constant. For every $y \in X/\mathcal R$, let us choose a lift $x_y \in p^{-1}(y)$. Let $\sigma \in \Homeo(X/\mathcal R)$. By definition of $p_*$, if $\varphi \in \Homeo(X, \mathcal R)$, we have $p_*(\varphi) = \sigma$ if and only if $\varphi(x_y) \in p^{-1}(\sigma(y))$ for every $y \in X/\mathcal R$. As a consequence, $p_*^{-1}(\sigma)$ is exactly the (finite) intersection of the open sets $\mathcal O_{x_y, p^{-1}(\sigma(y))}$ (for $y \in X/\mathcal R$), which is open, as claimed.
\end{proof}

\subsection{Actions and quotients}

In all this subsection, $G$ is a compact group acting continuously on a profinite space $X$. Our goal here is the following:
\begin{prop}\label{Quotient_of_Stone_space}
Let a compact group $G$ act continuously on a profinite space $X$. Then $X/G$ is a profinite space.
\end{prop}

\begin{lem}\label{saturating_clopen}
Let $G$ be a compact group acting on a compact Hausdorff set $X$. If a subset $Y$ of $X$ is open (resp.~closed), then $G \cdot Y$ is too.
\end{lem}

\begin{proof}
If $Y$ is open, then $GY = \bigcup_{g \in G} g \cdot Y$ is a union of open sets, so it is open. If $Y$ is closed in $X$ compact, then $Y$ is compact. Then $G \cdot Y$ is the image of the compact set $G \times Y$ by the continuous map $G \times X \rightarrow X$ sending $(g,x)$ to $g \cdot x$, so it is a compact subset of the Hausdorff space $X$, hence it is closed.
\end{proof}

\begin{lem}\label{disconnecting_compact_subsets}
Let $X$ be a totally disconnected space, and let $K$ and $L$ be disjoints compact subsets of $X$. Then there is a clopen subset $U$ of $X$ containing $K$ and disjoint from~$L$.
\end{lem}

\begin{proof}
We first suppose that $K = \{x\}$. Then for all $y \in L$, there exists some clopen set $V_y$ containing $y$ but not $x$. The $V_y$ cover $L$, so we can extract a finite cover, whose union is a clopen set $V$ containing $L$ but not $x$, as announced (its complement $V^c$ is a clopen containing $x$ and disjoint from $L$). Let us now consider the general situation. Given what we have just proved, for each $x \in K$, there is a clopen set $U_x$ containing $x$ and disjoint from $L$. The $U_x$ cover $K$, so we can extract a finite cover, whose reunion is a clopen set containing $K$ and disjoint from $L$.
\end{proof}

\begin{proof}[Proof of Proposition~\ref{Quotient_of_Stone_space}]
Let $x,y \in X$ such that $Gx \neq Gy$. Finding a clopen set in the quotient containing $Gx$ and not $Gy$ is equivalent to finding a clopen set $V$ of $X$ which is saturated (that is, $G \cdot V = V$) and contains $Gx$ but not $Gy$. Since $G$ is compact, so are the orbits in $X$ ($Gt$ is the image of $G$ by the continuous map $g \mapsto g \cdot t$). Hence, $Gx$ and $Gy$ are disjoint compact subsets of the profinite space $X$, so by Lemma~\ref{disconnecting_compact_subsets}, there is a clopen subset $U$ containing $Gx$ and disjoint from $Gy$. Then by Lemma~\ref{saturating_clopen}, the saturated set $V := G \cdot U$ is still clopen. Of course, it contains $Gx$ and it stays disjoint from $Gy$, whence the conclusion.
\end{proof}

\bibliographystyle{alpha}
\bibliography{Ref_profinite_qdles}

\end{document}